\documentclass[11pt,reqno]{amsart}

\usepackage[foot]{amsaddr}

\usepackage[utf8]{inputenc}
\usepackage{amsfonts,latexsym,amssymb,amsmath,amsthm,amsrefs,etex,slashed,cancel}
\usepackage{todonotes}
\usepackage{hyperref}
\usepackage{a4wide}
\usepackage{enumerate}
\usepackage{MnSymbol}
\usepackage{nicefrac}

\renewcommand{\Re}{\operatorname{Re}}
\renewcommand{\Im}{\operatorname{Im}}

\DeclareMathOperator{\supp}{supp}
\newcommand{\der}{\mathrm{d}}
\newcommand{\rmi}{\mathrm{i}}
\newcommand{\ee}{\mathrm{e} }
\newcommand{\sphere}{{\mathbb{S}^{2}}}

\newcommand{\tr}{\mathrm{tr}}
\newcommand{\Tr}{\mathrm{Tr}}
\newcommand{\trans}{\mathtt{t}}
\newcommand{\mi}{\mathrm{min}}
\newcommand{\ma}{\mathrm{max}}

\newcommand{\curl}{\mathrm{curl}}
\newcommand{\Curl}{\mathrm{Curl}}
\newcommand{\curlcurl}{\mathrm{curl} \, \mathrm{curl}}
\renewcommand{\div}{\mathrm{div}}
\newcommand{\grad}{\mathrm{grad}}

\newcommand{\calLt}{\tilde{\mathcal{L}}}
\newcommand{\calL}{\mathcal{L}}
\newcommand{\calMt}{\tilde{\mathcal{M}}}
\newcommand{\calM}{\mathcal{M}}

\newcommand{\R}{\mathbb{R}}
\newcommand{\C}{\mathbb{C}}
\newcommand{\N}{\mathbb{N}}

\newcommand{\aev}{\textrm{ a.e.}}
\newcommand{\del}{\partial}

\newcommand{\comp}{\mathrm{c}}

\newcommand{\loc}{\mathrm{loc}}
\newcommand{\calS}{\mathcal{S}}
\newcommand{\calSt}{\tilde{\mathcal{S}}}

\renewcommand{\tan}{\mathrm{tan}}
\newcommand{\Div}{\mathrm{Div}}
\newcommand{\free}{\mathrm{free}}
\newcommand{\abs}{\mathrm{abs}}
\newcommand{\rel}{\mathrm{rel}}

\newcommand{\id}{\operatorname{id}}

\newtheorem{theorem}{Theorem}[section]
\newtheorem{definition}[theorem]{Definition}

\newtheorem{lemma}[theorem]{Lemma}

\newtheorem{proposition}[theorem]{Proposition}

\newtheorem{rem}[theorem]{Remark}

\title[Relative trace]{The relative trace formula in electromagnetic scattering and boundary layer operators}

\author[A. Strohmaier]{Alexander Strohmaier}
\address{Leibniz University Hannover, Institute of Analysis, Welfengarten 1, 30167 Hannover, Germany}  \email{a.strohmaier@math.uni-hannover.de} 
\thanks{Supported by Leverhulme grant RPG-2017-329}

\author[A. Waters]{Alden Waters}
\address{ University of Groningen, Bernoulli Institute,
Nijenborgh 9,
9747 AG Groningen,
The Netherlands}
 \email{alden.waters@math.uni-hannover.de}

\begin{document}

\begin{abstract}
 This paper establishes trace formulae for a class of operators defined in terms of the functional calculus for the Laplace operator on divergence-free vector fields with relative and absolute boundary conditions on Lipschitz domains in $\R^3$. Spectral and scattering theory of the absolute and relative Laplacian is equivalent to the spectral analysis and scattering theory for Maxwell equations. The trace formulae allow for unbounded functions in the functional calculus that are not admissible in the Birman-Krein formula. In special cases the trace formula reduces to a determinant formula for the Casimir energy that is being used in the physics literature for the computation of the Casimir energy for objects with metallic boundary conditions. Our theorems justify these formulae in the case of electromagnetic scattering on Lipschitz domains, give a rigorous meaning to them as the trace of certain trace-class operators, and clarifies the function spaces on which the determinants need to be taken.
\end{abstract}


\maketitle
\setcounter{tocdepth}{1}
\tableofcontents

\section{Introduction}

In this paper we establish several trace formulae for operators governing the time-harmonic Maxwell equations on an open set $X=\Omega \cup M \subset \R^d$ of the form $\R^d \setminus \partial\Omega$, where $\Omega$ is a bounded (strongly) Lipschitz domain.
Here we will refer to $\Omega$ as the interior and to $M$ as the exterior domain.
We denote by $E$, $H$ the electric and magnetic fields, respectively. The time-harmonic Maxwell-system is given by
\begin{align}\label{system}
&\curl \, E-\rmi \lambda H=0 \nonumber \\& 
\div E=0  \nonumber \\&
\curl \,  H+ \rmi \lambda E=0 \\& \nonumber
\div \,  H=0 \\& \nonumber 
\nu \times E=A \quad \mathrm{on} \quad \partial\Omega\\& \nonumber
\langle \nu,H\rangle =f \quad \mathrm{on} \quad \partial\Omega 
\end{align}
where the first four equations are considered in either $\Omega$ or $M$ separately, or simultaneously by considering this as an equation on $X$. Here $\nu$ is the almost everywhere defined outward pointing unit normal vector field on $\partial \Omega$.
This system is well posed on suitable function spaces under natural consistency conditions on $A$ and $f$. In particular, if $A$ is sufficiently regular and tangential and $\lambda \not=0$ the function $f$ is determined by $A$. 
For the interior problem, given a tangential $A$, the system then has a unique solution for $\lambda$ away from a discrete set of points. For the exterior problem and $\Im(\lambda)>0$ one imposes that $E$ and $H$ are square integrable and then obtains a unique solution for any sufficiently regular tangential $A$. In both cases the solution $E$ can be expressed as 
$$
 E = \tilde{\mathcal{L}}_\lambda \mathcal{L}_\lambda^{-1} A,
$$
where $\tilde{\mathcal{L}}_\lambda$ is the electric field boundary layer potential operator and $\mathcal{L}_\lambda$ is the electric field boundary layer operator. For a continuous tangential vector field $A$ one has
$$
 (\tilde{\mathcal{L}}_\lambda A)(x) = \curl\;\curl \int_{\partial \Omega}  \frac{e^{\rmi \lambda |x-y|}}{4 \pi |x-y|} A(y) \der y,
$$
and $\mathcal{L}_\lambda A$ is obtained by taking the boundary value of $\nu \times \tilde{\mathcal{L}}$. These operators extend to suitable function spaces and we refer to Section \ref{MBLO} for the precise definitions. 
The vector field $H$ and the function $f$ are then determined by $H= -\frac{\rmi}{\lambda}\curl \, E$.
As usual this layer potential operator creates a solutions of the Maxwell system by placing certain sources on the boundary, and the choice of $\tilde{\mathcal{L}}$ is now a standard operator in computational electrodynamics.

For $\lambda \not=0$ the system for $E$ becomes
\begin{align}\label{system}
&-\Delta \, E -\lambda^2 E=0 \nonumber \\& 
\div E=0  \nonumber \\&
\nu \times E=A \quad \mathrm{on} \quad \partial\Omega. \nonumber
\end{align}
The associated spectral problem is therefore that of the Laplace-Beltrami operator $\Delta$, on divergence-free vector fields with the corresponding boundary condition. For the electric field the boundary condition $\nu \times E=0$ on $\partial \Omega$ leads to the relative Laplacian $\Delta_\rel$ by quadratic form considerations. Similarly, for the magnetic field the boundary condition $\nu \cdot H=0$ leads to the absolute Laplace operator $\Delta_\abs$. Both are are self-adjoint operators on $L^2(\R^3,\C^3)= L^2(\Omega,\C^3) \oplus L^2(M,\C^3)$ and their definition and properties are explained in detail in Sections \ref{InteriorLaplace} and \ref{extLap}.
Functional calculus for the relative Laplacian determines the solutions $E$ of the time-harmonic Maxwell-system, whereas functional calculus for the absolute Laplacian determines the solutions $H$ of the system. Here we use the more mathematical notation that is inspired by Hodge theory. The harmonic forms satisfying relative boundary conditions give rise to relative de Rham cohomology classes and the ones satisfying absolute boundary conditions give rise to absolute de Rham cohomology classes.

Before we describe the general case we would like to explain and motivate this in an important special case and when the bounded Lipschitz domain $\Omega \subset \R^3$ consists of two connected components $\Omega_1$ and $\Omega_2$. We then construct the self-adjoint operator $\Delta_\rel$ out of the Laplace operator on $\R^3 \setminus \partial \Omega$ by imposing relative boundary conditions in each side of $\partial \Omega$. The operators $\Delta_{j,\rel}$ are obtained in the same way from the Laplace operator on $\R^3 \setminus \partial \Omega_j$ with boundary conditions only imposed on each side of $\partial \Omega_j$.
The operators $\Delta_\abs$ and $\Delta_{j,\abs}$ are defined analogously with absolute boundary conditions.
It is a special case of our result that the two operators
\begin{align*}
 C_E = (-\Delta_\rel)^{-\frac{1}{2}} \delta \der - (-\Delta_{1,\rel})^{-\frac{1}{2}} \delta \der -(-\Delta_{2,\rel})^{-\frac{1}{2}} \delta \der + (-\Delta_\free)^{-\frac{1}{2}} \delta \der,\\
  C_H = (-\Delta_\abs)^{-\frac{1}{2}} \delta \der - (-\Delta_{1,\abs})^{-\frac{1}{2}} \delta \der -(-\Delta_{2,\abs})^{-\frac{1}{2}} \delta \der + (-\Delta_\free)^{-\frac{1}{2}} \delta \der,
\end{align*}
defined on smooth compactly supported vector fields on $X = \R^3 \setminus \partial \Omega$ extend to trace-class operators on $L^2(\R^3,\C^3)$ and their trace can be expressed in terms of the determinant of a combination of Maxwell boundary layer operators (see Theorems \ref{nicetheorem} and \ref{nicetheorem2}). In fact we will see that their traces coincide, i.e. $\tr\left(C_E \right) = \tr\left(C_H \right)$.
We have used here differential form notation, with $\der$ being the exterior derivative and $\delta$ being the co-derivative.
The trace-class property is due to several cancellations. Any linear combination of operators appearing in the expressions above that is not proportional to this expression is not trace-class. This statement remains true even if one introduces an artificial boundary thereby compactifying the problem.

 In terms of vector-calculus the above two operators can also be written as 
\begin{align*}
 C_E&=(-\Delta_\rel)^{-\frac{1}{2}} \curlcurl - (-\Delta_{1,\rel})^{-\frac{1}{2}} \curlcurl -(-\Delta_{2,\rel})^{-\frac{1}{2}} \curlcurl + (-\Delta_\free)^{-\frac{1}{2}} \curlcurl,\\
 C_H &= \curl(-\Delta_\rel)^{-\frac{1}{2}} \curl - \curl (-\Delta_{1,\rel})^{-\frac{1}{2}} \curl - \curl(-\Delta_{2,\rel})^{-\frac{1}{2}} \curl + \curl (-\Delta_\free)^{-\frac{1}{2}} \curl.
\end{align*}

Apart from being interesting from the point of view of spectral analysis these operators also have a direct physical significance.
Namely $\frac{1}{4}\tr\left(C_E + C_H\right)=\frac{1}{2} \tr\left(C_E\right)$ represents the Casimir energy of the two Lipschitz obstacles $\Omega_1$ and $\Omega_2$. Indeed, as shown in \cite{MR4324382} in a general rigorous framework of quantum field theory, the local trace, i.e. the trace of the integral kernel restricted to the diagonal, of the operator 
$$
 \frac{1}{4}\left((-\Delta_\rel)^{-\frac{1}{2}} \curlcurl - (-\Delta_{\free})^{-\frac{1}{2}} \curlcurl \right) + \frac{1}{4}\left((-\Delta_\abs)^{-\frac{1}{2}} \curlcurl - (-\Delta_{\free})^{-\frac{1}{2}} \curlcurl \right)
$$ 
is the renormalised energy density obtained from the electromagnetic quantum field theory.
The relative resolvent differences $C_E$ and $C_H$ then describe differences of energies. It was shown in \cite{Fang2021AMA}, again in a rigorous QFT framework, that in case of the scalar field that such ``energy differences'' lead to a Casimir force as determined from the quantum stress energy tensor as in \cite{kay1979casimir, Candelas1982}. The same statement is expected to hold for the electromagnetic field, but this will be discussed elsewhere.

The mathematical statements above can therefore also be interpreted as a rigorous proof that the Casimir energy as derived from spectral quantities is well defined in this framework and can be computed from determinants of boundary layer operators. It also clarifies the function spaces needed to compute these quantities for non-smooth boundaries. 

We focus in this paper on Maxwell's equations in dimension three and we will mostly use vector-calculus notations rather than differential forms. 
This has the advantage of keeping the notations and exposition more accessible and we can then also rely on a wealth of previous results on boundary layer operators (\cite{costabel1988,MR1944792, costabelremark, MR2839867, mitrea1995, MR1899489, mitrea1997, dmitrea, kirsch, MR3904426, MR2032868}). Focusing on dimension three also avoids complications with the free Green's function having more complicated expressions or a logarithmic singularity at zero. More importantly the focus on dimension three allows us to stay close to the classical notations in Maxwell theory without having to distract the reader with more complicated notations. \hspace{1cm}\\

Although this is a mathematical paper we also try to give of physics background for the interested reader.
To our knowledge a determinant formula for the Casimir energy first appeared in the physics literature \cite{renne71} where this was derived microscopically and without reference to spectral theory. Physics derivations have also appeared in various contexts based on path integrals and fluctuations of configurations on the surface on the obstacles (\cite{kenneth06,EGJK2007,kenneth08}) and have led to numerical schemes \cite{johnson2011numerical} and asymptotic formulae. The spectral side, often favouring a zeta function regularisation approach as in Casimir's original work \cite{Casimir} was developed somewhat independently. We refer to \cite{bordag2009advances} and \cite{kirsten2002} for a comprehensive overview over the subject. The relation between the various approaches remained unclear even in the physics world (for a very recent report on this see \cite{universe7070225} and references). We also mention the approach of \cite{balian78}, which is also based on a reduction to the boundary.

\subsection{Statement of main results}
We now describe the general setting of our results.
We assume that $\Omega \subset \R^3$ be an open and bounded (strongly) Lipschitz domain in $\R^3$  in the sense that the boundary of $\Omega$ is locally congruent to the graph of a Lipschitz function.
The finitely many connected components will be denoted
by $\Omega_j$ with some index $j$, which ranges from $1$ to $N$. We will think of the closure $\overline{\Omega}$ as a collection of disjoint compact obstacles $\overline{\Omega}_j$ placed in $\R^3$ (see Fig. \ref{objects}).
Removing these obstacles from $\R^3$ results in a non-compact open domain
$M = \R^3 \backslash \overline{\Omega}$ with Lipschitz boundary $\partial \Omega$. We will assume throughout that $M$ is connected. It will also be convenient to introduce $X = \R^3 \setminus \partial \Omega = M \cup \Omega$. 

\begin{figure}[h]
	\centering
	\includegraphics[scale=0.15]{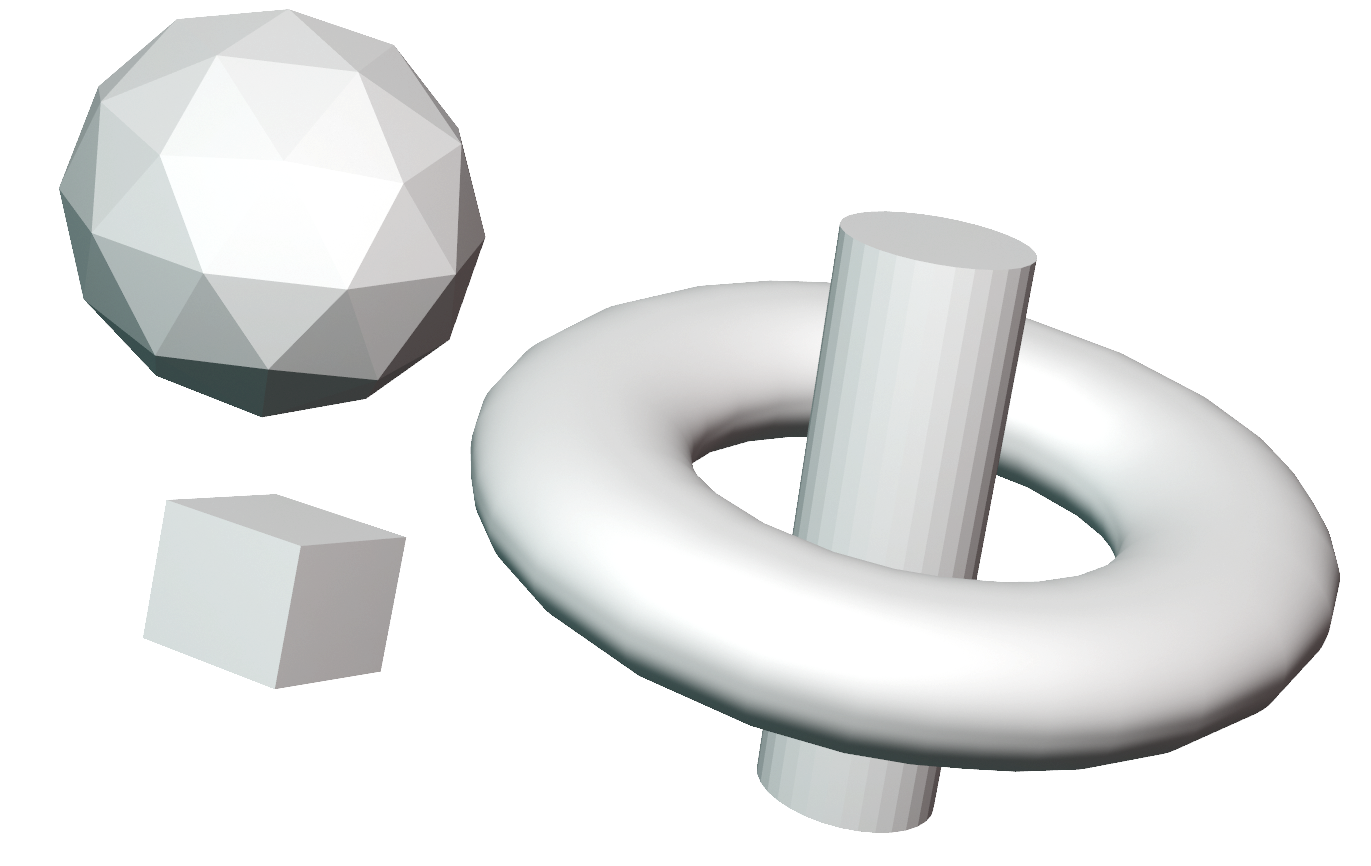}
	\caption{A Lipschitz domain $\Omega$ consisting of four connected components $\Omega_1,\Omega_2,\Omega_3,\Omega_4$.}
	\label{objects}
\end{figure}

On the boundary one has well-defined anisotropic Sobolev spaces $H^{-\frac{1}{2}}(\Div,\partial \Omega)$ (see Section \ref{Lipschitzsection}) and the Maxwell electric field operator $\mathcal{L}_\lambda$
is a  bounded operator $H^{-\frac{1}{2}}(\Div,\partial \Omega) \to H^{-\frac{1}{2}}(\Div,\partial \Omega)$ (see Section \ref{MBLO}). This can be done for each object separately and one can assemble the individual parts $\mathcal{L}_{\lambda, \partial \Omega_j} : H^{-\frac{1}{2}}(\Div,\partial \Omega_j) \to H^{-\frac{1}{2}}(\Div,\partial \Omega_j)$
into an operator $\mathcal{L}_{D,\lambda} = \oplus_{j=1}^N  \mathcal{L}_{\lambda, \partial \Omega_j}$ acting on $H^{-\frac{1}{2}}(\Div,\partial \Omega)$.
\begin{theorem} \label{nicetheorem}
The operator $\mathcal{L}_\lambda \mathcal{L}^{-1}_{D,\lambda}$ is well-defined and a trace-class perturbation of the identity for any complex $\lambda$ with $\Im(\lambda)>0$. It therefore has a well-defined Fredholm determinant  $\mathrm{det}(\mathcal{L}_\lambda \mathcal{L}^{-1}_{D,\lambda})$ on the space $H^{-\frac{1}{2}}(\Div,\partial \Omega)$. Let $\delta$ be the minimal distance between separate objects. Then for any $0 < \delta' < \delta$ the function $$\Xi(\lambda) =\log \det(\mathcal{L}_\lambda \mathcal{L}^{-1}_{D,\lambda}),$$ where the branch of the logarithm has been fixed by continuity, extends to a holomorphic function in a neighborhood of the closed upper half space and it
satisfies the bound
$$
 |\Xi(\lambda)| \leq C \mathrm{e}^{-\Im(\lambda) \delta}
$$
for $\lambda$ in any sector about the positive imaginary axis of angle strictly less than $\pi$.
\end{theorem}
We note that the operators $\mathcal{L}_\lambda^{-1}$ and $\mathcal{L}^{-1}_{D,\lambda}$ have singularities at zero and it is due to a variety of cancellations that the quotient is regular at zero, in particular when the objects have non-trivial topology. Our proof is based on a careful analysis of these singularities.

\begin{figure}[h]
	\centering
	\includegraphics[scale=1.4]{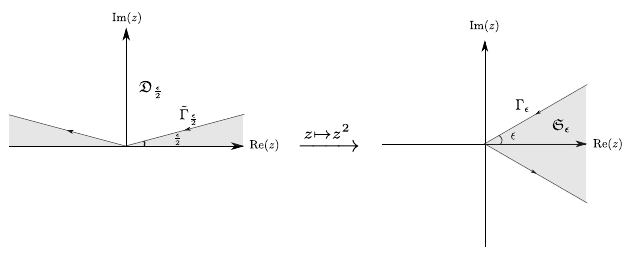}
	\caption{The sectors $\mathfrak{S}_\epsilon$, $\mathfrak{D}_\frac{\epsilon}{2}$ and the corresponding contours.}
	\label{ztozsquare}
\end{figure}

Before we formulate the trace formula we need to define a large class of functions to which it applies. These will be analytic functions in certain sectors and we start by describing these sectors.
Assume $0<\epsilon\leq \pi$ and let $\mathfrak{S}_\epsilon$ be the open sector
\begin{align*}
\mathfrak{S}_\epsilon=\{ z \in \C \mid z\not=0, | \arg(z) | < \epsilon \}
\end{align*}
containing the real axis (see Fig. \ref{ztozsquare}). Associated to these we define the  following spaces of functions. The space $\mathcal{E}_\epsilon$ will be defined by
\begin{gather*}
 \mathcal{E}_\epsilon= \{ f : \mathfrak{S}_\epsilon \to \C \mid f \textrm{ is holomorphic in } \mathfrak{S}_\epsilon, \exists \alpha >0, \forall \epsilon_0>0, \, |f(z)| = O( |z|^\alpha e^{\epsilon_0 |z|}) \}.
\end{gather*}
  We define the space $\mathcal{P}_\epsilon$ as the set of functions in  $\mathcal{E}_\epsilon$ whose restriction to $[0,\infty)$ is polynomially bounded and that extend continuously to the boundary 
  of $\mathfrak{S}_\epsilon$ in the logarithmic cover of the complex plane.
 Reference to the logarithmic cover of the complex plane is only needed in case $\epsilon=\pi$. In this case 
 functions in $\mathcal{P}_\pi$ are required to have continuous limits from above and below on the negative real axis. We do not however require that these limits coincide.
The space $\mathcal{P}_\epsilon$ contains in particular $f(z) = z^{\alpha}, \quad \alpha >0$ for any $0<\epsilon \leq \pi$.

When working with the Laplace operator it is often convenient to change variables and use $\lambda^2$ as a spectral parameter, and in the context of Maxwell theory it turns out to be beneficial to introduce an extra $\lambda^{-2}$-factor. For notational brevity we therefore introduce another class of functions as follows.
\begin{definition} \label{def:TildePEps}
 The space $\widetilde{\mathcal{P}}_\epsilon$ is defined to be the space of functions $f$ such that $f(\lambda)=\lambda^{-2 }g(\lambda^2)$ for some $g \in \mathcal{P}_\epsilon$. In particular $f(\lambda) = O(\lambda^a)$
 for some $a> -2$ near $\lambda=0$.
\end{definition}
Very generally the operator $\Delta_{\rel}$ decomposes into a direct sum of unbounded operators $\Delta_{\rel} = 0 \oplus   \der \delta \oplus \delta \der$ under the weak Hodge-Helmholtz decomposition (see Section \ref{extLap}, \eqref{wHHd}) we have
$$
  f((-\Delta_\rel)^\frac{1}{2}) \curl\curl = f((-\Delta_\rel)^\frac{1}{2}) \delta \der =  f((\delta \der)^\frac{1}{2}) \delta \der
$$
for any Borel function $f$. This implies that for a function $f \in \widetilde{\mathcal{P}}_\epsilon$ the unbounded operator  $f((-\Delta_\rel)^\frac{1}{2}) \curl\curl$
contains $C^\infty_0(X,\C^3)$ in its domain. Indeed, for $\psi \in C^\infty_0(X,\C^3)$ and $k \in \N$ large enough we have the factorisation
$f((\delta \der)^\frac{1}{2}) \delta \der = h((\delta \der)^\frac{1}{2}) (\delta \der + 1)^k \psi$, where
$(\delta \der + 1)^k \psi \in C^\infty_0(X,\C^3)$ and the function $h(\lambda)=(1+\lambda^2)^{-k}\lambda^2 f(\lambda)$ is bounded on the real line.

For $0<\epsilon \leq \pi$ we also define the contours $\Gamma_\epsilon$ in the complex plane as the boundary curves of the sectors 
$\mathfrak{S}_\epsilon$. In case $\epsilon=\pi$ the contour is defined as a contour in the logarithmic cover of the complex plane. We also let $\widetilde{\Gamma}_\frac{\epsilon}{2}$ be the corresponding 
contour after the change of variables, i.e. the pre-image in the upper half space under the map $z \to z^2$ of $\Gamma_\epsilon$ (see Fig. \ref{ztozsquare}).

For $f \in \widetilde{\mathcal{P}}_\epsilon$ we define the {\sl relative operator}
\begin{align*}
 D_{\rel,f} &=  f((-\Delta_\rel)^\frac{1}{2}) \curlcurl - f((-\Delta_\free)^\frac{1}{2}) \curlcurl \\&- \sum\limits_{j=1}^N  \left( f((-\Delta_{j,\rel})^\frac{1}{2}) \curlcurl - f((-\Delta_\free)^\frac{1}{2}) \curlcurl \right),
\end{align*}
where  $f(\lambda) = g(\lambda^2)$.
Similarly,
\begin{align*}
 D_{\abs,f} &=  f((-\Delta_\abs)^\frac{1}{2}) \curlcurl - f((-\Delta_\free)^\frac{1}{2}) \curlcurl \\&- \sum\limits_{j=1}^N  \left( f((-\Delta_{j,\abs})^\frac{1}{2}) \curlcurl - f((-\Delta_\free)^\frac{1}{2}) \curlcurl \right)\\
 &=  \curl f((-\Delta_\rel)^\frac{1}{2}) \curl - \curl f((-\Delta_\free)^\frac{1}{2}) \curl \\&- \sum\limits_{j=1}^N  \left( \curl f((-\Delta_{j,\rel})^\frac{1}{2}) \curl -\curl f((-\Delta_\free)^\frac{1}{2}) \curl \right)
\end{align*}
Since these operators contain $C^\infty_0(X,\C^3)$ in their domain they are densely defined.

We refer to taking these differences as the {\sl relative setting}, indicating that this compares interacting quantities to non-interacting ones. It is unfortunate that the word relative is also used to denote the relative boundary conditions. We alert the reader that these two uses of the word relative are unrelated, but to avoid confusion we have used the symbol $D$ for ``difference'' to denote relative objects.

Our main result reads as follows.

\begin{theorem} \label{nicetheorem2}
 If $f \in \widetilde{\mathcal{P}}_\epsilon$, then operators $D_{\rel,f}$ and $D_{\abs,f}$ extend to trace-class operators $L^2(\R^3,\C^3) \to L^2(\R^3,\C^3)$ and
 $$
  \tr(D_{\rel,f}) = \tr(D_{\abs,f})= \frac{\rmi}{2\pi}\int_{\tilde\Gamma_\frac{\epsilon}{2}} \Xi(\lambda) \frac{\der}{\der \lambda}(\lambda^2 f(\lambda))  \der \lambda,
 $$
 where the contour $\tilde \Gamma_{\frac{\epsilon}{2}}$ is the clockwise-oriented boundary of a sector that includes the imaginary axis.
\end{theorem}

We would like to mention here that expressions formally similar to the relative trace-formula have appeared in the context of the multi-channel scattering theory and were introduced by Buslaev and Merkur'ev (\cite{MR0261896}, see also \cite{MR1944028}) to prove Birman-Krein-type formulae. In this context the test function $f$ is still required to decay sufficiently fast.

An interesting application of the relative trace is that it allows one to define a relative zeta function, namely
$$
 \zeta_{D}(s) = \tr\left( D_{f_s}\right),\quad f_s(\lambda) = \frac{1}{\lambda^{2s +2}} 
$$ 
for $\Re(s) < 0$. As a consequence of Theorem \ref{nicetheorem2} this relative zeta function then satisfies
$$
  \zeta_{D}(s) = \frac{2s}{\pi} \sin(\pi s) \int_0^\infty \lambda^{-2s -1} \Xi(\rmi \lambda) \der \lambda.
 $$
 This formula allows for a meromorphic continuation of $\zeta_D$ with poles of order at most one and residues related to the Taylor coefficients of $\Xi(\rmi \lambda)$ at zero. These coefficients are interesting in their own right and will be investigated elsewhere.
In the special case when $f(\lambda)= \frac{1}{\lambda}$ this gives the expression
$$
 \frac{1}{4} \tr\left( C_E + C_H \right) = \frac{1}{2\pi}\int_0^\infty \Xi(\rmi \lambda) \der \lambda
$$
for the Casimir energy.

Under our more general assumptions on $f$ the operators 
\begin{gather} \label{Breldef}
  B_{\rel,f} = f((-\Delta_\rel)^\frac{1}{2}) \curlcurl - f((-\Delta_\free)^\frac{1}{2}) \curlcurl,\\
  B_{\abs,f} = f((-\Delta_\abs)^\frac{1}{2}) \curlcurl - f((-\Delta_\free)^\frac{1}{2}) \curlcurl, \label{Babsdef}
\end{gather} are not trace-class. One has however the following theorem about the smoothness and integrability properties of their integral kernels.
\begin{theorem} \label{nicetheorem3}
 Let $B_f$ be either $B_{\rel,f}$, defined by \eqref{Breldef}, or  $B_{\rel,f}$, defined by \eqref{Babsdef}. Then $B_f$ has an integral kernel $\kappa \in C^\infty(X \times X, \mathrm{Mat}(3,\C))$, which is smooth away from the boundary.
 If $\Omega_0 \subset X$ has positive distance to the boundary $\partial \Omega$ and $p_{\Omega_0}$ the orthogonal projection
 $L^2(\R^3,\C^3) \to L^2(\Omega_0,\C^3)$, then $p_{\Omega_0} B_f p_{\Omega_0}$ extends to a trace-class operator with trace
 equal to the convergent integral
 $$
  \int_{\Omega_0} \tr\left(\kappa(x,x)\right) \der x.
 $$
 If $f(z) = O(|z|^a)$ for $|z|<1$ we have for
 large $|x|$ the estimate
 $$
  \| \kappa(x,x) \| \leq C_f \frac{1}{|x|^{6+a}}.
 $$
\end{theorem}

\subsection{Discussion}

The theorems presented here are the Maxwell-analogue of the paper \cite{HSW} where a similar statement was proved for the scalar Laplacian in the case of smooth boundary. The Maxwell system on a Lipschitz domain is different in several regards and introduces challenges that are absent in the scalar case: 
\begin{itemize}
\item  Maxwell's equations arise from an abelian gauge theory and the gauge freedom results in the loss of ellipticity of the equations for the electromagnetic field. On the analysis side this manifests itself as the equations taking place on the space of divergence-free vector fields rather than the space of sections of the vector bundle. This can however be fixed by considering the spectral decomposition of the Laplace operator and then employing the Helmholtz-Hodge decomposition to project onto the subspace of divergence-free vector fields. Projecting works well in cases with and without boundary as long as the geometric configuration is fixed. The projector constructed from the Helmholtz decomposition is roughly of the form $-\Delta^{-1} \delta \der =-\Delta^{-1} \curlcurl$ and it involves the non-local functional calculus of the Laplace operator. It therefore depends on the geometric configuration and also the boundary conditions imposed on the Laplace operator. This makes it much harder to directly apply scattering theory which requires an identification of the involved Hilbert spaces. The same problem appears in the context of the Birman-Krein formula in electromagnetic scattering. We have proved a variant of the Birman-Krein formula in \cite{SWB} and we will follow the same formulation here.

\item Unlike the Dirichlet-Laplacian the Laplace operator on the space of vector fields with relative boundary conditions has a non-trivial kernel in the exterior domain. This leads to singularities of the resolvent near zero and manifests itself in the presence of singularities of the boundary layer operators. Additional singularities of the boundary layer operators appear if the obstacles have non-trivial topology, which we do not exclude. To overcome this we carefully analyse the singularities of various Maxwell boundary layer operators at zero and we show that there are various cancellations that render a final result without singularities.

\item An additional complication arises in this paper since we are considering Lipschitz domains instead of smooth ones. This requires more sophisticated harmonic analysis techniques. We rely here on a lot of progress in this subject that has been made during the past several decades, in particular with the identification of the appropriate function spaces.

\end{itemize}

As explained the spectral theory of $\Delta_\rel$ and $\Delta_\abs$ determines the Maxwell-system.
 Suitably interpreted the $\curl$-operator intertwines these two operators in the sense that $\curl\, \Delta_\abs = \Delta_\rel \,\curl$.
In the interior the relative Laplacian on a suitable closed subspace consisting of divergence-free vector fields has the Maxwell eigenvalues as its spectrum and the eigenfunctions describe modes of photons that are confined to $\Omega$.
The exterior relative Laplacian on a suitable closed subspace of divergence-free vector fields describes the scattering of electromagnetic waves, or photons respectively, by the obstacles $\Omega$. The functional calculus on $\Delta_\rel$ on this subspace can be understood in terms of the operators $f(\Delta_\rel) \curlcurl$. The following Birman-Krein formula has been proved.

\subsection{Relation to the Birman-Krein formula}
In case $f$ is an even Schwartz function we have that $$\left(\curl \, \curl( f((-\Delta_\rel)^{\frac{1}{2}}) - f((-\Delta_\free)^{\frac{1}{2}}))  \right) $$ is trace-class and its trace can be computed by the Birman-Krein-type formula
$$
   \tr\left(\curl \, \curl( f((-\Delta_\rel)^{\frac{1}{2}}) - f((-\Delta_\free)^{\frac{1}{2}}))  \right) = \frac{1}{2 \pi \rmi}\int_0^\infty \lambda^2\tr(S_\lambda^{-1}(S_\lambda)') f(\lambda) d \lambda + \sum_{j=1}^\infty f(\mu_j) \mu_j^2,
 $$ 
 where $S_\lambda$ is the scattering matrix for the Maxwell equation and $\mu_j$ are the Maxwell eigenvalues of the interior.
 As a consequence of this formula
 $$
   \tr D_{\rel,f} = -\frac{1}{2\pi \rmi} \int_0^\infty  \log \frac{\det  S_\lambda}{\det ( S_{1,\lambda}) \cdots \det ( S_{N,\lambda})} \frac{\der}{\der \lambda}(\lambda^2 f(\lambda)) \der \lambda,
 $$
 which is valid only under very restrictive assumptions on $f$. The same formula and statements hold for absolute instead of relative boundary conditions.
 
 In the motivating example one cannot use this formula. It would require $f(\lambda) = \frac{1}{\lambda}$, which does not satisfy the assumptions of the Birman-Krein formula. In fact it can be shown that the integrand on the right hand side is not integrable in that case.
One has however the following relation between the function $\Xi$ and the scattering matrices.
\begin{theorem} \label{nicetheorem4}
We have
$$
 \log \frac{\det  S_\lambda}{\det ( S_{1,\lambda}) \cdots \det ( S_{N,\lambda})} =  -\left(\Xi(\lambda) -  \Xi(-\lambda)\right)
$$
for $\lambda \in \R$.
\end{theorem}

This theorem reflects the relation between the spectral shift function and zeta regularised determinants as discovered by Carron (\cite{carron1999determinant}, Theorem 1.3) generalising a formula by Gesztesy and Simon (\cite{MR1395669}, Theorem 1.1).

\subsection{Organisation of the paper} 

The paper is organised as follows. Sections \ref{Lipschitzsection}- \ref{MBLO} provide the required theoretical background for the paper and consist of essentially known material. Section \ref{Lipschitzsection} sets up the basic function spaces needed for boundary layer theory on Lipschitz domains. Section \ref{InteriorLaplace} summarises the spectral properties of the interior relative and absolute Laplace operators, and Section \ref{extLap} reviews the scattering theory for the relative and absolute Laplacians on the exterior. Both are combined into one operator in Section \ref{combreltr}. In this section we also discuss the Birman-Krein formula in the context of our setting. Section \ref{MBLO} introduces the basic Maxwell boundary layer operators and their properties.

The basic estimates and expansions for the layer potential operators needed for the proofs are covered in Section \ref{estimateSec}. This section is presented independently of the main results as its content is interesting in its own right. It covers various aspects of low energy expansions for the electric and magnetic boundary layer operators and inverses. Section \ref{resolventformulae} gives formulae of the resolvent differences in terms of layer potential operators and thereby provides estimates for these differences. Such formulae are sometimes referred to as Krein-type resolvent formulae and this section provides a Maxwell analogue of these.
Sections \ref{multiL} and \ref{multiLzwei} take on the main subject of this paper, namely function $\Xi$, the relative resolvent, and its trace. Section \ref{mainproofs} finally contains the proofs of the main theorems.

\section{Function spaces on Lipschitz domains} \label{Lipschitzsection}
Since $\Omega$ is a Lipschitz domain we have, by Rademacher's theorem, an almost everywhere defined exterior unit vector field $\nu \in L^\infty(\partial \Omega,\R^3)$. We will use the following spaces that now are standard in Maxwell theory:
\begin{itemize} 
 \item $H(\curl,M) = \{ f \in L^2(M,\C^3) \mid \curl f \in L^2(M,\C^3)\}$.
 \item $H(\div,M) = \{ f \in L^2(M,\C^3) \mid \div f \in L^2(M)\}$.
 \item $L^2_\tan(\partial \Omega) = \{ f \in L^2(\partial \Omega,\C^3) \mid \nu\cdot f =0 \aev \textrm{ on } \partial \Omega \}$.
 \item $H^{-\frac{1}{2}}(\Div,\partial \Omega)$, $H^{-\frac{1}{2}}(\Curl,\partial \Omega)$.
 \item $H^{-\frac{1}{2}}(\Div0,\partial \Omega)$, $H^{-\frac{1}{2}}(\Curl0,\partial \Omega)$.
\end{itemize}
These spaces were introduced in \cite{MR1944792} and provide a convenient framework for dealing with Maxwell's equations on Lipschitz domains. We refer to the Appendix of \cite{kirsch} for an extensive discussion and we only summarise the basic properties.

In case $\partial \Omega$ is smooth we have 
\begin{align*}
  H^{-\frac{1}{2}}(\Div,\partial \Omega) &= \{ f \in H^{-\frac{1}{2}}(\partial \Omega;T\partial\Omega) \mid \Div f \in H^{-\frac{1}{2}}(\partial \Omega)\},\\
  H^{-\frac{1}{2}}(\Curl,\partial \Omega) &= \{ f \in H^{-\frac{1}{2}}(\partial \Omega;T\partial\Omega) \mid \Curl f \in H^{-\frac{1}{2}}(\partial \Omega)\},\\
  H^{-\frac{1}{2}}(\Div0,\partial \Omega) &= \{ f \in H^{-\frac{1}{2}}(\partial \Omega;T\partial\Omega) \mid \Div f =0\},\\
  H^{-\frac{1}{2}}(\Curl0,\partial \Omega) &= \{ f \in H^{-\frac{1}{2}}(\partial \Omega;T\partial\Omega) \mid \Curl f =0\},
\end{align*}
where $\Div$ is the surface divergence on $\partial \Omega$, and $\Curl$ is the surface curl. On a general Lipschitz domain this can be defined via Lipschitz coordinate charts, thus locally reducing it to the smooth case. Note that the spaces $H^s_\loc(\R^d)$ are invariant under bi-Lipschitz maps if $|s| \leq 1$.
We refer to \cite{kirsch} for a detailed discussion of the definition via coordinate charts.
We also have the corresponding spaces for the interior domains. Namely we have that 
\begin{align*}
  H(\curl,\Omega) &= \{ f \in L^2(\Omega,\C^3) \mid \curl f \in L^2(\Omega,\C^3)\},\\
  H(\div,\Omega) &= \{ f \in L^2(\Omega,\C^3) \mid \div f \in L^2(\Omega)\}.
 \end{align*}
On $H(\curl,M)$ there are two distinguished and well-defined continuous  trace maps 
\begin{align*}
 \gamma_{T,-} &: H(\curl,M) \to H^{-\frac{1}{2}}(\Curl,\partial \Omega),\\
 \gamma_{t,-} &: H(\curl,M) \to H^{-\frac{1}{2}}(\Div,\partial \Omega),
\end{align*}
which continuously extend the maps $f \mapsto (\nu \times f |_{\partial \Omega}) \times \nu$ and $f \mapsto (\nu \times f |_{\partial \Omega})$ respectively, defined on $C_0(\overline{M},\C^3)$.
Note that for $x \in \partial \Omega$ such that $\nu_x$ is defined the map $v \mapsto (\nu_x  \times v) \times \nu_x$ is the orthogonal projection onto the tangent space of $\partial \Omega$ at $x$. 
Similarly we have the map
\begin{align*}
 \gamma_{\nu,-} &: H(\div,M) \to H^{-\frac{1}{2}}(\partial \Omega),
\end{align*}
continuously extending the normal restriction map $f \mapsto \nu \cdot f|_{\partial \Omega}$.
On the interior domain $\Omega$ we have the analogous maps
\begin{align*}
 \gamma_{T,+} &: H(\curl,\Omega) \to H^{-\frac{1}{2}}(\Curl,\partial \Omega),\\
 \gamma_{t,+} &: H(\curl,\Omega) \to H^{-\frac{1}{2}}(\Div,\partial \Omega),\\
 \gamma_{\nu,+} &: H(\div,\Omega) \to H^{-\frac{1}{2}}(\partial \Omega),
\end{align*}
There is a well defined dual pairing between $H^{-\frac{1}{2}}(\Curl,\partial \Omega)$ and $H^{-\frac{1}{2}}(\Div,\partial \Omega)$
that extends the $L^2$-inner product on $H^\frac{1}{2}(\partial \Omega) \cap L^2_\tan(\partial \Omega)$. We will denote this pairing by $\langle \cdot, \cdot\rangle_{L^2(\partial \Omega)}$, irrespective of the Sobolev order and mildly abusing notation.
The map
$\phi \mapsto \nu \times \phi$ extends to a continuous isomorphism from $H^{-\frac{1}{2}}(\Div,\partial \Omega)$ to  $H^{-\frac{1}{2}}(\Curl,\partial \Omega)$
and vice versa. Moreover, the $L^2$-pairing induces an antilinear isomorphism between $H^{-\frac{1}{2}}(\Div,\partial \Omega)$ and $H^{-\frac{1}{2}}(\Curl,\partial \Omega)$ (see for example \cite{kirsch}*{Lemma 5.61} for both statements). In other words,
the antisymmetric bilinear form $\langle \cdot,\nu \times \cdot \rangle$  on $H^{-\frac{1}{2}}(\Div,\partial \Omega)$ is non-degenerate.
Note here that since  $\nu \in L^\infty(\partial \Omega,\R^3)$ it is not immediately obvious that is defined as a map between Sobolev spaces.

We recall Stokes theorem for $\phi, E\in H(\curl,\Omega)$: 
\begin{align}\label{stokes}
&\langle \gamma_{t,+} E,  \gamma_{T,+}\phi\rangle_{L^2(\partial\Omega)}=\langle \curl\, E,\phi\rangle_{L^2(\Omega)}-\langle E,\curl \, \phi\rangle_{L^2(\Omega)},  \\&
\langle \curlcurl\, E,\phi\rangle_{L^2(\Omega)}-\langle E,\curlcurl\,\phi\rangle_{L^2(\Omega)}\\&=\langle \gamma_{t,+} \curl\, E,\gamma_{T,+}\phi\rangle_{L^2(\partial\Omega)}+\langle \gamma_{t,+} E,\gamma_{T,+} \curl\,\phi\rangle_{L^2(\partial\Omega)}. \nonumber
\end{align}
As before we are slightly abusing notation and write $\langle \cdot,\cdot \rangle_{L^2(\Omega)}$ for pairings extending the $L^2$-inner product.
We also define $H_0(\curl,M)$ as the kernel of $\gamma_{t,-}$ and $H_0(\div,M)$ as the kernel of $\gamma_{\nu,-}$. 
These spaces play a similar role as the Sobolev space of functions $H^1_0(M)$, which can also be characterised as the kernel of the trace map $\gamma: H^1(M) \to H^\frac{1}{2}(\partial M)$.
The spaces $H_0(\curl,\Omega)$ and $H_0(\div,\Omega)$, as well as $H^1_0(\Omega)$ are defined analogously.

If there is no danger of confusion we will omit the $\pm$ and simply write $\gamma_t$ and $\gamma_\nu$ respectively.

We also have surface divergence $\Div$ and surface curl $\Curl$. They satisfy
\begin{gather} \label{Divcurlinter}
  \Div \circ \gamma_{t,+} = - \gamma_{\nu,+} \circ  \curl.
 \end{gather}

\section{Laplace operators on the interior domain} \label{InteriorLaplace}

\subsection{The relative Laplacian}

The operator $$\curl_\mi=\curl |_{H_0(\curl,\Omega)}: H_0(\curl,\Omega)  \to L^2(\Omega,\C^3)$$ is a closed densely defined operator. It coincides with the closure of the operator $\curl$ on the space of compactly supported smooth vector fields on $\Omega$ (\cite{kirsch}*{Theorem 5.25})   and therefore equals the minimal closed extension of $\curl$.

Its adjoint is the maximal extension, i.e. the closed operator $$\curl_\ma: H(\curl,\Omega)  \to L^2(\Omega,\C^3).$$ 
For any closed densely defined operator $A$ the operator $A^*A$ is automatically self-adjoint. If in addition $\mathrm{rg}(A) \subset \ker(A)$, then
 $A^* A + A A^*$ is self-adjoint if it is densely defined (see for example \cite{SWB}*{Section 2}).
It follows that
$\curl_\ma\, \curl_\mi$ with domain 
$$
 \{ f \in  H_0(\curl,\Omega) \mid \curl f \in H(\curl,\Omega)  \}
$$
is a non-negative self-adjoint operator. Similarly, $\div_\ma: H(\div,\Omega) \to L^2(\Omega)$ is a closed operator with adjoint $-\grad_\mi: H^1_0(\Omega) \to L^2(\Omega)$. Therefore, the operator
$- \grad_\mi\; \div_\ma$ is a non-negative self-adjoint operator with domain
$$
 \{ f \in  H(\div,\Omega) \mid \div f \in H_0(\Omega)  \}.
$$
Their sum $\Delta_{\Omega,\rel}=\curl_\ma \curl_\mi - \grad_\mi\; \div_\ma$ is again self-adjoint and then has domain
$$
 \{ f \in H(\div,\Omega) \cap  H_0(\curl,\Omega) \mid \div f \in H_0(\Omega),  \curl f \in H(\curl,\Omega) \},
$$
and on this domain $-\Delta_{\Omega,\rel}$ is given by $ \curlcurl- \grad\; \div$.
The implied boundary conditions of this operator are the so-called relative boundary conditions
$$
 \gamma_{t,+}(f) = 0, \quad \div f|_{\partial \Omega} =0.
$$

In the case of smooth boundary the form-domain of the interior relative Laplace operator is contained in $H^1(\Omega,\C^3)$.
In the more general Lipschitz case this is no longer true, but it is known that the form domain is contained in $H^\frac{1}{2}(\Omega,\C^3)$ (see \cite{costabelremark}*{Theorem 2} and also \cite{MR1899489}).
This is compactly embedded in $L^2(\Omega,\C^3)$ and therefore the interior relative Laplace operator has purely discrete spectrum. We have the classical Hodge-Helmholtz decomposition
$$
 L^2(\Omega) = \mathcal{H}^1(\Omega) \oplus  \overline{\mathrm{rg}(\grad_\mi)} \oplus \overline{\mathrm{rg}(\curl_\ma)} 
$$
into an orthogonal direct sum.  Here $\mathcal{H}^1(\Omega)= \mathrm{\ker}(\Delta_{\Omega,\rel})$ is the finite dimensional space of harmonic vector fields
satisfying the relative boundary conditions. We will see in Section \ref{cohosubsec} that in fact the assumption that $M$ is connected implies that
$\mathcal{H}^1(\Omega) = \{0\}$.

We now describe the spectrum of the relative Laplace operator.
On $\Omega$ we can choose an orthonormal basis $(v_j)$ of Dirichlet eigenfunctions $v_j$ in the domain of the Dirichlet Laplacian with eigenvalues $\lambda_j^2$, i.e.
\begin{gather*}
 -\Delta v_j = \lambda_{D,j}^2 v_j, \quad  \lambda_{D,j}>0, \quad v_j \in \{v \in H^1_0(\Omega,\C^3) \mid \nabla v \in H(\div,\Omega) \},
\end{gather*}
We have $\lambda_j \to \infty$ and we arrange the eigenfunctions such that $\lambda_j \nearrow \infty$.
Then $\frac{1}{\lambda_j} \nabla v_j$ form an orthonormal basis of eigenfunctions in $\overline{\mathrm{rg}(\grad_\mi)}$ of $-\Delta_\Omega$ with eigenvalues $\lambda_j^2$.
One has the usual Weyl-law for Lipschitz domains which can easily be inferred from the Weyl law for smooth domains using domain monotonicity and an approximation by smooth domains
$$
  \lambda_{D,k} \sim \left( \frac{6\pi^2}{\mathrm{Vol}(\Omega)} \right)^{\frac{1}{3}} k^{\frac{1}{3}}, \quad k \to \infty.
$$
The space $\overline{\mathrm{rg}(\curl_\ma)}$ on the other hand is the closure of the subspace spanned by $\phi_j$, where $(\phi_j)$ is an orthonormal basis in $\mathrm{ker}(\div_\ma) \subset L^2(\Omega,\C^3)$ satisfying the eigenvalue equation
$$
  -\Delta_{\Omega,\rel} \phi_j  = \mu_j^2 \phi_j, \; \div \, \phi_j = 0,
$$
with boundary condition $\gamma_t(\phi_j)=0$.  Therefore zero is not an eigenvalue.
The numbers $\mu_j>0$ are the Maxwell eigenvalues and we again assume these are arranged such that $\mu_j \nearrow \infty$.
The Maxwell eigenvalues are known to satisfy a Weyl law (see \cite{MR928156} for Lipschitz domains, but also \cite{MR3113431} and references for a general statement in arbitrary dimension)
$$
  \mu_k \sim \left( \frac{3\pi^2}{\mathrm{Vol}(\Omega)} \right)^{\frac{1}{3}} k^{\frac{1}{3}}, \quad k \to \infty.
$$
The family
$(\phi_j)_{\mu_j>0}$ then forms an orthonormal basis in $\overline{\mathrm{rg}(\curl_\ma)}$ consisting of eigenfunctions of 
$-\Delta_{\Omega,\rel}$ with non-zero eigenvalues $\mu_j^2$. Summarising there is an orthonormal basis of eigenfunctions $\Delta_{\Omega,\rel}$ of the form
\begin{align*}
 \{ \frac{1}{\lambda_j} \grad\, v_j \,|\, j \in \N \} \cup  \{ \phi_j \,|\, \mu_j >0 \},
\end{align*}
where $v_j$ are the Dirichlet eigenfunctions and $\phi_j$ the Maxwell eigenfunctions with Maxwell eigenvalues $\mu_j$. 

\subsection{The absolute Laplacian}

It will also be convenient to consider another operator $\Delta_{\Omega,\abs}$, which is defined by
$$
 -\Delta_{\Omega,\abs} = \curl_\mi \;\curl_\ma - \grad_\ma\; \div_\mi,
$$
with domain
$$
 \{ f \in H_0(\div,\Omega) \cap  H(\curl,\Omega) \mid \div f \in H^1(\Omega),  \curl f \in H_0(\curl,\Omega) \}.
$$

Again, it is known that the form-domain is contained in $H^\frac{1}{2}(\Omega,\C^3)$ (\cite{costabelremark}*{Theorem 2}) and the domain is therefore compactly embedded into $L^2(\Omega,\C^3)$.
In the same way as for the relative Laplacian there is an explicit description of the spectrum which we now give.
Let $(u_j)$ be an orthonormal basis consisting of eigenfunctions of the Neumann Laplacian with eigenvalues $\lambda_{N,j}$. Hence,
$$
 -\Delta u_j = \lambda_{N,j}^2 u_j, \quad \partial_\nu u_j|_{\partial \Omega} =0, \quad u_j \in \{ u \in H^1(\Omega) \mid \nabla u \in H^1_0(\div,\Omega) \}. 
$$
Then the functions $\frac{1}{ \lambda_{N,j}} \nabla u_j$ form an orthonormal set consisting of eigenfunctions of $\Delta_{\Omega,\abs}$.

We can construct another orthogonal set $(\psi_j)$ from the Maxwell eigenfunctions $\phi_j$ of the relative Laplace operator by defining
$$
 \psi_j = \frac{1}{\mu_j}\curl \phi_j.
$$
Since the spectrum is discrete standard Hodge theory applies for the absolute Laplacian and we obtain an orthogonal decomposition
$$
 L^2(\Omega,\C^3) = \mathcal{H}^1_{\abs}(\Omega) \oplus \overline{\mathrm{span}\{\frac{1}{ \lambda_{N,j}}\grad u_j\} }\oplus \overline{ \mathrm{span} \{\psi_j\}},
$$
where $\mathcal{H}^1_{\abs}(\Omega) = \ker \Delta_{\Omega,\abs}$. Unlike in the case of the relative Laplace operator this space is in general nontrivial. We will in the following choose an orthonormal basis $(\psi_{0,k})_k$, where $1 \leq k \leq \dim(\mathcal{H}^1_{\abs}(\Omega))$. Therefore an orthonormal basis in $L^2(\Omega,\C^3)$ consisting of eigenfunctions of the absolute Laplacian is
$$
 \{\psi_{0,k} \mid 1 \leq k \leq \dim(H^1_\abs(\Omega))\} \cup \{\frac{1}{ \lambda_{N,j}}\nabla u_j \mid j \in \N\} \cup \{\psi_j \mid j \in \N\}.
$$

\subsection{Relation to singular and de Rham cohomology groups} \label{cohosubsec}

Since $\Omega$ is an oriented smooth manifold we have, by de Rham's theorem, a natural isomorphism identifying
$H^p_\mathrm{dR}(\Omega,\C)$ with $H^p_\mathrm{sing}(\Omega,\C)= H^p_\mathrm{sing}(\Omega,\mathbb{Z}) \otimes_\mathbb{Z} \C$.
Hodge theory is also applicable for Lipschitz domains in the sense that the natural map from $\ker \Delta_{\Omega,\abs}$ to the first de Rham cohomology group $H^1_\mathrm{dR}(\Omega,\C)$ is an isomorphism. This can for example be inferred from the statement of \cite{MR1809655}*{Th. 11.1 and Th. 11.2} together with the universal coefficient theorem and de Rham's theorem. This theorem also applies to the absolute Laplacian on $2$-forms as defined in \cite{MR1809655}. Since this operator is obtained by conjugation of the relative Laplacian on one forms with the Hodge star-operator $*$, we therefore have that $*\ker \Delta_{\Omega,\rel}$ is isomorphic to $H^2_\mathrm{dR}(\Omega,\C)$. 
Because the inner product is non-degenerate on these spaces, we have the following non-degenerate dual pairing
$$
\ker \Delta_{\Omega,\rel} \times (*\ker \Delta_{\Omega,\rel}) \to \C, \quad (f_1,f_2) \mapsto \int_\Omega f_1 \wedge f_2.
$$

We also have, as a consequence of  Poincar\'e duality, the non-degenerate dual pairing
$$
 H^{1}_{\comp,\mathrm{dR}}(\Omega,\C) \times H^2_\mathrm{dR}(\Omega,\C) \to \C, \quad (f_1,f_2) \mapsto \int_\Omega f_1 \wedge f_2.
 $$
This establishes an isomorphism $\ker \Delta_{\Omega,\rel} \to H^{1}_{\comp,\mathrm{dR}}(\Omega,\C)$, which relates the harmonic forms to the de Rham cohomology groups with compact support. Since elements in $\ker \Delta_{\Omega,\rel}$ are not compactly supported this map is defined indirectly by duality.

Our assumptions imply that in fact $H^{1}_{\comp,\mathrm{dR}}(\Omega,\C)$ is trivial and therefore $\ker \Delta_{\Omega,\rel} =\{0\}$. This reflects the observation that a domain with connected exterior cannot have homologically non-trivial $2$-cycles (inclusions). 

\begin{lemma}\label{vanishcoho}
 Let $U$ be an open $C^0$-domain with compact closure in $\R^d$ with $d \geq 2$ such that $\R^d \setminus \overline{U}$ is connected. Then $H^{1}_{\comp,\mathrm{dR}}(U) =\{0\}$.
\end{lemma}
\begin{proof}
Let $\alpha$ be a smooth closed one-form with compact support in $U$.
By the Poincar\'e lemma there is a smooth function $f : \R^d \to \R$ with $\alpha = \der f$. 
Since $f$ is locally constant in the complement of the support of $\alpha$ it must be constant in 
$\R^d \setminus \overline U$, as this set was assumed to be connected. By continuity $f$ is constant in $\R^d \setminus \overline U$ and, since locally constant, it is constant in a neighborhood of $\R^d \setminus \overline U$.
It follows that $f-c$ is compactly supported in $U$.
 Since $\alpha = \der (f-c)$
the class $\alpha$ vanishes $H^{1}_{\comp,\mathrm{dR}}(U)$ and therefore $H^{1}_{\comp,\mathrm{dR}}(U)=\{0\}$.

\end{proof}

\section{Laplace operators on the exterior domain} \label{extLap} 

As in the interior case the operator $$\curl_\mi=\curl |_{H_0(\curl,M)}: H_0(\curl,M)  \to L^2(M,\C^3)$$ is a closed densely defined operator with adjoint $$\curl_\ma: H(\curl,M)  \to L^2(M,\C^3).$$ It follows that
$\curl_\ma\, \curl_\mi$ with domain 
$$
 \{ f \in  H_0(\curl,M) \mid \curl f \in H(\curl,M)  \}
$$
is a non-negative self-adjoint operator. Similarly, $\div_\ma: H(\div,M) \to L^2(M)$ is a closed operator with adjoint $-\grad_\mi: H^1_0(M) \to L^2(M)$. Therefore, the operator
$- \grad_\mi\; \div_\ma$ is a non-negative self-adjoint operator with domain
$$
 \{ f \in  H(\div,M) \mid \div f \in H_0(M)  \}.
$$
Their sum $-\Delta_{M,\rel}=  \curl_\ma\;\curl_\mi- \grad_\mi\; \div_\ma$ then has domain
$$
 \{ f \in H(\div,M) \cap  H_0(\curl,M) \mid \div f \in H_0(M),  \curl f \in H(\curl,M) \}.
$$
The implied boundary conditions are the exterior relative boundary conditions
$$
 \gamma_{t,-}(f) = 0, \quad \div f|_{\partial \Omega} =0.
$$
The spectrum of the operator $\Delta_{M,\rel}$ consists of a finite multiplicity eigenvalue at zero and a purely absolutely continuous part. This is the consequence of the finite-type meromorphic continuation of the resolvent and Rellich's theorem. We have described this in detail in \cite{OS} for smooth domains, but this part of the paper carries over to Lipschitz domains without change (see \cite{SWB} for a discussion of this point).
The absolutely continuous part of the spectrum can be described well by stationary scattering theory.
For each $\Phi \in C^\infty(\sphere,\C^3)$ and $\lambda>0$ there exists a unique generalised eigenfunction
$E_\lambda(\Phi) \in C^\infty(M,\C^3)$ satisfying the boundary conditions of $\Delta_{M,\rel}$ near $\partial \Omega$ such that
\begin{gather} 
 (-\Delta - \lambda^2) E_\lambda(\Phi)=0,\\
 E_\lambda(\Phi) = \frac{\ee^{-\rmi \lambda r} }{r} \Phi - \frac{\ee^{\rmi \lambda r}}{r} \Psi_\lambda(\Phi) + O\left(\frac{1}{r^2} \right),  \quad \textrm{for} \,\,\,r \to \infty, \label{expan}
\end{gather}
uniformly in the angular variables on the sphere for some $\Psi_\lambda(\Phi) \in C^\infty(\sphere,\C^3)$. 
The expansion \eqref{expan} may be differentiated 
 (c.f. Prop. 2.6 and Appendix E in \cite{OS} for a justification). Here, satisfying the boundary conditions near $\partial \Omega$ means that $\chi E_\lambda(\Phi) \in \mathrm{dom}(\Delta_M)$ for any compactly supported smooth $\chi$ on $M$ such that $\chi=1$ near $\partial \Omega$.

The above implicitly defines the {\sl scattering matrix} as a map $\tilde S_\lambda : C^\infty(\sphere,\C^3) \to C^\infty(\sphere,\C^3)$ by
$\Psi_\lambda(\Phi) = \tau \tilde S_\lambda \Phi$ where $\tau: C^\infty(\sphere;\C^3) \to C^\infty(\sphere;\C^3) $ is the pull-back of the antipodal map. It extends continuously as $\tilde S_\lambda : L^2(\sphere,\C^3) \to L^2(\sphere,\C^3)$.
The map $\tilde A_\lambda= \tilde S_\lambda - \mathrm{id}$ is called the scattering amplitude.
We have the equations
$$
 \curl\,\curl\, E_\lambda(\Phi) =  \lambda^2 E_\lambda({\bf r} \times \Phi \times {\bf r}), \quad  \div\, E_\lambda(\Phi) = -\rmi \lambda E_\lambda^0({\bf r} \cdot \Phi),
$$
where ${\bf r}$ is the radius vector, i.e. the outward pointing unit vector on the sphere. Here $E_\lambda^0({\bf r} \cdot \Phi)$ is the generalised eigenfunction for the exterior Dirichlet problem on scalar-valued functions defined in an analogous way, c.f. Proposition 4.7 in \cite{SWB}. 
In particular this means that in case $\Phi$ is purely tangential, 
${\bf r} \cdot \Phi=0$, the generalised eigenfunction is a solution of the stationary Maxwell equation
\begin{gather*}
 \curlcurl E_\lambda(\Phi) = \lambda^2 E_\lambda(\Phi),\\
 \div E_\lambda(\Phi) = 0,
\end{gather*}
that satisfies the boundary conditions near $\partial \Omega$. 
These equations also imply that the scattering matrix is of the form
$$
 \tilde S_\lambda = \left( \begin{matrix} S^D_\lambda & 0 \\ 0 & S_\lambda \end{matrix} \right),
$$
if $L^2(\sphere,\C^3)$ is decomposed into $L^2(\sphere) {\bf r} \oplus  L^2_\tan(\sphere,\C^3)$. Here $L^2_\tan(\sphere,\C^3)$ is the space of tangential square integrable vector fields on the sphere.
The operator $S^D_\lambda$ is the scattering operator for scalar valued functions with Dirichlet conditions imposed on $\partial \Omega$, and $S_\lambda$ is the Maxwell scattering operator, describing the scattering of electromagnetic waves. Note that we have the weak Hodge-Helmholz decomposition
\begin{gather} \label{wHHd}
 L^2(M) = \mathcal{H}^1_\rel(M) \oplus  \overline{\mathrm{rg}(\grad_\mi)} \oplus \overline{\mathrm{rg}(\curl_\ma)},
\end{gather}
which holds very generally in the abstract context of Hilbert complexes (\cite{MR1174159}).
The first summand is the discrete spectral subspace, and the splitting of its orthogonal complement into the last two subspaces
corresponds to the above decomposition of the scattering matrix. 

\subsection{The exterior absolute Laplacian}

In the same way as for the interior problem there is also an exterior absolute Laplacian $\Delta_{M,\abs}$ defined by
$$
 -\Delta_{M,\abs} = \curl_\mi \curl_\ma- \grad_\ma\; \div_\mi.
$$
The spectrum of $\Delta_{M,\abs}$ consists of a finite multiplicity eigenvalue at zero and an absolutely continuous part. The absolutely continuous part is described by generalised eigenfunctions $E_{\abs,\lambda}(\Phi)$ which are related to the generalised eigenfunctions $E_\lambda(\Phi)$ of the relative Laplacian by
\begin{gather} \label{equabsasdkljn}
 E_{\abs,\lambda}({\bf r} \times \Phi) = -\frac{\rmi}{\lambda} \curl\, E_\lambda(\Phi).
\end{gather}
One checks easily that
$$
 (-\Delta_{M,\rel} -\lambda^2)^{-1} \curl = \curl (-\Delta_{M,\abs} -\lambda^2)^{-1}
$$
on the dense set of compactly supported smooth functions, and, appropriately interpreted, extends by continuity to a larger space.
This will allow us to reduce to statements about the absolute Laplace operator to statements about the relative Laplace operator. For the purposes of this paper it will therefore not be necessary to introduce separate notations for the spectral decomposition. For example the scattering matrix
$$
 \tilde S_{\abs,\lambda} = \left( \begin{matrix} S^N_{\lambda} & 0 \\ 0 & S_{\abs,\lambda} \end{matrix} \right),
$$
for the absolute Laplacian is defined by the expansion of $E_{\abs,\lambda}(\Phi)$. Here $S^N_{\lambda}$ is the scattering matrix for the Neumann Laplace operator on $M$ acting on functions. We then have the equation
\begin{gather}
 S_{\abs,\lambda} (g) = {\bf r} \times S_{\lambda} (g \times {\bf r}) \label{relabsrel}
\end{gather}
for  $g \in L^2_\tan(\sphere,\C^3)$. 
This follows by applying $\curl$ to the expansion \eqref{expan}, the
uniqueness of the generalised eigenfunctions, and Equ. \eqref{equabsasdkljn}. 

\section{The combined relative operators and the Birman-Krein formula} \label{combreltr}

In the following it will be convenient to combine the operators $\Delta_{M,\rel}$ and $\Delta_{\Omega,\rel}$ into a single operator acting on the Hilbert space $L^2(\R^3,\C^3)$. We have $L^2(\R^3,\C^3) = L^2(M,\C^3)  \oplus L^2(\Omega,\C^3)$ and we define the operator $\Delta_\rel:= \Delta_{M,\rel} \oplus \Delta_{\Omega,\rel}$.
In contrast to this we also have the free Laplace operator $\Delta_\free$ with domain $H^2(\R^3,\C^3)$. 
Following the paper \cite{HSW} on the relative trace we also define the operator $\Delta_{j,\rel}$ for each boundary component $\Omega_j$. This will correspond to the operator $\Delta_\rel$ when all the other boundary components are absent, i.e. when $\Omega = \Omega_j$. 
As in \cite{HSW} we would like to consider an analogue of the relative trace for the Laplace operator acting on divergence free vector fields. In this section we assume that $f \in \mathcal{S}(\R)$ is an even Schwartz function, but later on we will focus on another function class. We would like to compute the relative trace

\begin{align*}
 \tr\left(\curlcurl  \left( f((-\Delta_\rel)^\frac{1}{2}) -  f((-\Delta_\free)^{\frac{1}{2}}) - \left( \sum_{j=1}^N f((-\Delta_{j,\rel})^{\frac{1}{2}}) - f((-\Delta_\free)^{\frac{1}{2}}) \right) \right) \right)\\ =  \tr\left( \curlcurl \left( f((-\Delta_\rel)^{\frac{1}{2}}) - \sum_{j=1}^N f((-\Delta_{j,\rel})^{\frac{1}{2}}) + (N-1)f((-\Delta_\free)^{\frac{1}{2}})\right)  \right) ,
\end{align*}
which is the trace of the operator
\begin{align*}
 D_{\rel,f} = \curlcurl \left( f((-\Delta_\rel)^{\frac{1}{2}}) -  f((-\Delta_\free)^{\frac{1}{2}}) - \left( \sum_{j=1}^N f((-\Delta_{j,\rel})^{\frac{1}{2}}) - f((-\Delta_\free)^{\frac{1}{2}}) \right) \right) .
\end{align*}

We have the following Birman-Krein-type formula, proved recently in \cite{SWB} and its simple consequence for the relative trace.
\begin{theorem}[Theorem 1.5 in \cite{SWB}] \label{BKF}
 Let $f \in C^\infty_0(\R)$ be an even function. Then the operator
 $$
  \curl \, \curl \left(f((-\Delta_\rel)^{\frac{1}{2}}) - f((-\Delta_\free)^{\frac{1}{2}})\right) 
 $$
 extends to a trace-class operator on $L^2(\R^3,\C^3)$ and its trace equals
 $$
   \tr\left(\curl \, \curl( f((-\Delta_\rel)^{\frac{1}{2}}) - f((-\Delta_\free)^{\frac{1}{2}}))  \right) = \frac{1}{2 \pi \rmi}\int_0^\infty \lambda^2\tr(S_\lambda^{-1}(S_\lambda)') f(\lambda) d \lambda + \sum_{j=1}^\infty f(\mu_j) \mu_j^2.
 $$ 
 Moreover,
 $$
   \tr\left(D_f\right) = -\int_0^\infty \xi_D(\lambda)(f(\lambda)\lambda^2)'  d \lambda,
 $$ 
 where
 $$
  \xi_D(\lambda) = \frac{1}{2 \pi \rmi} \log \frac{\det  S_\lambda}{\det ( S_{1,\lambda}) \cdots \det ( S_{N,\lambda})}.
 $$
\end{theorem}

A similar statement holds for the absolute Laplacian. Using Equ. \eqref{relabsrel} and $$\curl\;f((-\Delta_\rel)^{\frac{1}{2}}) = f((-\Delta_\abs)^{\frac{1}{2}}) \curl$$
one obtains
\begin{align*}
  \tr\left(\curl \, ( f((-\Delta_\rel)^{\frac{1}{2}}) - f((-\Delta_\free)^{\frac{1}{2}})) \curl \right) &=  \tr\left(\curl \, \curl( f((-\Delta_\abs)^{\frac{1}{2}}) - f((-\Delta_\free)^{\frac{1}{2}}))  \right) \\
   &=  \tr\left(\curl \, \curl( f((-\Delta_\rel)^{\frac{1}{2}}) - f((-\Delta_\free)^{\frac{1}{2}}))  \right) \\
   &= \frac{1}{2 \pi \rmi}\int_0^\infty \lambda^2\tr(S_\lambda^{-1}(S_\lambda)') f(\lambda) d \lambda + \sum_{j=1}^\infty f(\mu_j) \mu_j^2.
\end{align*}

The Birman-Krein formula can be proved for a slightly larger function class than the space of even Schwartz functions, 
but non-decaying functions are not admissible. The rest of the paper is devoted to dealing with exactly the trace-class properties of $D_f$ when $f$ is in a different function class that contains possibly growing functions.

\section{Maxwell boundary layer operators} \label{MBLO}

Maxwell boundary layer theory for Lipschitz domains is a well developed subject in mathematics and in this section we 
summarise the material that we are going to need.
The distributional kernel of the resolvent of the operator $(-\Delta_\free - \lambda^2)^{-1}$ is called the Green's function and in dimension three given explicitly by
\begin{align}\label{fskernel}
 G_{\lambda,\free}(x,y) = \frac{1}{4 \pi} \frac{e^{\rmi \lambda |x-y|}}{|x-y|}.
\end{align}
Note that this kernel is holomorphic at zero. 
As usual we define the single layer potential operator $\calSt_\lambda: H^{-\frac{1}{2}}(\partial \Omega) \to H^{1}_\loc(\R^3)$
by $$\tilde{\mathcal{S}}_\lambda = (-\Delta_\free - \lambda^2)^{-1}  \gamma^*.$$ This is defined for any $\lambda \in \C$ and a holomorphic family of operators.
The single layer operator is defined by taking the trace
$\mathcal{S}_\lambda = \gamma_+ \tilde{\mathcal{S}}_\lambda = \gamma_+ (-\Delta_\free - \lambda^2)^{-1}  \gamma^*$. The interior trace $\gamma_+$ and the exterior trace $\gamma_-$ coincide on the range of $\tilde{\mathcal{S}}_\lambda$ and therefore we could also have used $\gamma_-$ to define this operator. The operator $\mathcal{S}_\lambda$
is a holomorphic family of maps $H^{-\frac{1}{2}}(\partial \Omega) \to H^{\frac{1}{2}}(\partial \Omega)$. Both operators $\tilde{\mathcal{S}}_\lambda$
and $\mathcal{S}_\lambda$ act component-wise on $H^{-\frac{1}{2}}(\partial \Omega,\C^3)$ and define maps to $H^{1}_\loc(\R^3,\C^3)$ and $H^{-\frac{1}{2}}(\partial \Omega,\C^3)$ respectively. We will this distinguish notationally from the map on functions.

We will also need the double layer operator $\mathcal{K}_\lambda$ and its transpose (complex conjugate-adjoint) $\mathcal{K}^{t}_\lambda$. The latter is given by
$$
 \mathcal{K}^{t}_\lambda u = \frac{1}{2}\left( \gamma_+ \nabla_\nu \mathcal{S}_\lambda u + \gamma_- \nabla_\nu \mathcal{S}_\lambda u \right),
$$
and defines a continuous map $\mathcal{K}^{t}_\lambda: H^{-\frac{1}{2}}(\partial \Omega) \to H^{-\frac{1}{2}}(\partial \Omega)$. Its transpose $\mathcal{K}_\lambda$ therefore defines a continuous map $\mathcal{K}_\lambda: H^{\frac{1}{2}}(\partial \Omega) \to H^{\frac{1}{2}}(\partial \Omega)$. The following jump-relations are characteristic
$$
 \gamma_+ \mathcal{S}_\lambda u = \gamma_- \mathcal{S}_\lambda u, \quad \gamma_\pm \nabla_\nu \mathcal{S}_\lambda u = (\mp \frac{1}{2} + \mathcal{K}^{t}_\lambda )u.
$$

We have the following representation formulae for divergence-free solutions  $\phi \in H(\curl,M)  \oplus  H(\curl,\Omega)$ 
of the vector-valued Helmholtz equation
\begin{align*}
 (-\Delta - \lambda^2) \phi&=0,\\
 \div \, \phi =0,
\end{align*}
by single layer potential operators
 \begin{align} \label{singlerepext}
  \phi|_{M} = -\curl\, \tilde{\mathcal{S}}_\lambda ( \gamma_{t,-} \phi) + \nabla \tilde{\mathcal{S}}_\lambda (\gamma_{\nu,-} \phi) - \tilde{\mathcal{S}}_\lambda (\gamma_{t,-} \curl \phi)
 \end{align}
 and likewise
 \begin{align} \label{singlerepint}
  \phi|_{\Omega} = -\curl\, \tilde{\mathcal{S}}_\lambda ( \gamma_{t,+} \phi) + \nabla \tilde{\mathcal{S}}_\lambda (\gamma_{\nu,+} \phi) - \tilde{\mathcal{S}}_\lambda (\gamma_{t,+} \curl \phi),
 \end{align}
 c.f Corollary 3.3 in \cite{mitrea1997}

In Maxwell theory one defines additional layer potential operators as follows. Let $L$ be the distribution defined by
$$
 L_\lambda(x,y) = \curl_x \curl_x G_{\lambda,\free}(x,y) .
$$
This is the kernel of the  operator
$(-\Delta_\free - \lambda^2)^{-1} \curlcurl=\curlcurl(-\Delta_\free - \lambda^2)^{-1} $. It is again holomorphic at $\lambda=0$ as a kernel.
The corresponding operator $L_\lambda$ is related to the operator
$$
 (\lambda^2 + \grad\; \div)(-\Delta_\free - \lambda^2)^{-1},
$$
whose distributional integral kernel equals the so-called dyadic Green's function
$$
 K_\lambda(x,y)= (\lambda^2 + \grad_x\; \div_x)\frac{1}{4 \pi}   \frac{e^{\rmi \lambda |x-y|}}{|x-y|},
$$
which is more commonly used in computational electrodynamics.
However, we also have the following inequality
$$
  L_\lambda(x,y) - K_\lambda(x,y) = \delta(x-y),
$$
hence the kernels agree outside the diagonal. We define now the {\sl Maxwell single layer potential operator} for $u \in H^{\frac{1}{2}}(\partial\Omega,\C^3) \cap L^2_\tan(\partial\Omega)$ as
\begin{gather*}
 u \mapsto \calLt_\lambda u,\;  (\calLt_\lambda u)(x)=\int_{\partial \Omega} L_\lambda(x,y) u(y) dy = 
 \int_{\partial \Omega} K_\lambda(x,y) u(y) dy.
\end{gather*}
Therefore this can also be written as $\calLt_\lambda u= \curlcurl\, \tilde{\mathcal{S}}_\lambda u$. 
Similarly one defines the Maxwell magnetic layer potential operator $\tilde{\mathcal{M}}_{\lambda}$ as 
$\tilde{\mathcal{M}}_{\lambda} u = \curl\,\tilde{\mathcal{S}}_\lambda u$.
For all $\lambda \in \C$ these maps extends continuously to maps as follows
\begin{align*}
 \calLt_\lambda : H^{-\frac{1}{2}}(\Div,\partial \Omega) \to H_\loc(\curl,M) \oplus H_\loc(\curl,\Omega),\\
 \calMt_\lambda : H^{-\frac{1}{2}}(\Div,\partial \Omega) \to H_\loc(\curl,M) \oplus H_\loc(\curl,\Omega).
\end{align*}
It will be convenient to distinguish notationally between the exterior part $\calMt_{-,\lambda}$ and the interior part $\calMt_{+,\lambda}$ of $\calMt_\lambda$. 
The boundedness of these maps is established in \cite{kirsch} for $\Im{\lambda}\geq 0, \lambda \not=0$ but these maps 
extend to holomorphic families on the entire complex plane as we will see later.

The {\sl Maxwell single layer operator} $\calL_\lambda$ is then defined for all $\lambda \in \C$ as a map 
$$
  \calL_\lambda: H^{-\frac{1}{2}}(\Div,\partial \Omega) \to H^{-\frac{1}{2}}(\Div,\partial \Omega), \quad u \mapsto \gamma_t \tilde{\mathcal{L}}_\lambda
$$ 
and is a holomorphic family of bounded operators on $H^{-\frac{1}{2}}(\Div,\partial \Omega)$ in $\lambda$. With respect to the above splitting we then have
$$
 \calMt_\lambda = \calMt_{-,\lambda} \oplus \calMt_{+,\lambda}.
$$
One defines the {\sl magnetic dipole operator} $\calM_\lambda$ for all $\lambda \in \C$ by
\begin{align*}
\mathcal{M}_{\lambda}: H^{-\frac{1}{2}}(\mathrm{Div},\partial \Omega) \to H^{-\frac{1}{2}}(\mathrm{Div},\partial \Omega), \quad
\mathcal{M}_{\lambda}=\frac{1}{2}\left(\gamma_t\tilde{\mathcal{M}}_{-,\lambda}+\gamma_t\tilde{\mathcal{M}}_{+,\lambda}\right).
\end{align*} 
By \cite{kirsch}*{Theorem 5.52} this is a family of bounded operators on the space $H^{-\frac{1}{2}}(\mathrm{Div},\partial \Omega)$ when $\Im(\lambda)>0$.
If  $u = \tilde{\mathcal{M}_{\lambda}} a = \mathrm{curl}\, \calSt_\lambda a$ then we have the jump conditions
\begin{align} \label{jumpjumpjump}
\gamma_{t,\pm} u=\mp\frac{1}{2}a+\mathcal{M}_{\lambda}a \quad \gamma_{t,\pm}\mathrm{curl}\,u=\mathcal{L}_{\lambda}a.
\end{align}
Moreover, the operator $\tilde{\mathcal{L}}_{\lambda}a$ can be written as
\begin{align}\label{divexpansion}
\tilde{\mathcal{L}}_{\lambda}a=\nabla \tilde{\mathcal{S}}_{\lambda}\, \mathrm{Div} a+\lambda^2\tilde{\mathcal{S}}_{\lambda}a\quad a\in H^{-\frac{1}{2}}(\mathrm{Div},\partial\Omega).
\end{align}
We refer to \cite{kirsch}*{Theorem 5.4} for both statements.

If $\Im(\lambda) \geq 0$ is non-zero then
there exists a unique solution of the exterior boundary value problem for every $A\in H^{-\frac{1}{2}}(\mathrm{Div},\partial \Omega)$, which satisfies the Silver-M\"uller radiation condition \cite{kirsch}*{Theorem 5.64}.
For the interior problem there exists a similar statement. 
If  $\lambda \in \C \setminus \{0\}$ is not a Maxwell eigenvalue then there exists a unique solution of the interior boundary value problem for every $A \in H^{-\frac{1}{2}}(\mathrm{Div},\partial \Omega)$.
In both cases, if $\lambda\not=0$ the solution can be written as boundary layer potential of the form 
\begin{align}
E(x)=(\tilde{\mathcal{L}_{\lambda}}a)(x)=\mathrm{curl}^2\langle a,G_\lambda(x,\cdot)\rangle_{\partial\Omega}, \quad
H(x)=\frac{\rmi\,\mathrm{curl}E}{-\lambda} \quad x \notin \partial\Omega
\end{align}
with the density $a\in H^{-\frac{1}{2}}(\mathrm{Div},\partial \Omega)$, which satisfies $\mathcal{L}_{\lambda}a=A$, c.f again Theorem 5.60 in \cite{kirsch}. 

The space of boundary data $(\gamma_t(E),\gamma_t(H))$ of solutions of Maxwell's equations is described by the Calderon projector. To describe this we first observe that given $a,b \in H^{-\frac{1}{2}}(\mathrm{Div},\partial \Omega)$ we obtain for any non-zero $\lambda$ a solution of the interior Maxwell's equation 
$E,H \in H(\curl,\Omega)$ by 
$$
 E = - \tilde{\mathcal{M}}_{\lambda} a + \frac{1}{\rmi \lambda} \tilde{\mathcal{L}}_\lambda b, \quad H = - \tilde{\mathcal{M}}_{\lambda} b - \frac{1}{\rmi \lambda} \tilde{\mathcal{L}}_\lambda a,
$$
and therefore, using \eqref{jumpjumpjump}, the boundary data $(\gamma_t(E),\gamma_t(H))$ is obtained as
$$
 \left( \begin{matrix} \gamma_t (E) \\ \gamma_t (H) \end{matrix} \right) 
= 
  \left( \begin{matrix} \frac{1}{2} - \mathcal{M}_\lambda & \frac{1}{\rmi \lambda} \mathcal{L}_\lambda \\ -\frac{1}{\rmi \lambda} \mathcal{L}_\lambda & \frac{1}{2} - \mathcal{M}_\lambda \end{matrix} \right)  \left( \begin{matrix}a \\ b \end{matrix} \right). 
$$
By the Stratton-Chu representation formula, \cite{kirsch}*{Theorem 5.49}, we have that in case $(E,H)$ solves Maxwell's equations then $E$ and $H$ can be recovered from the boundary data as
$$
 E = - \tilde{\mathcal{M}}_{\lambda} (\gamma_t E) + \frac{1}{\rmi \lambda} \tilde{\mathcal{L}}_\lambda (\gamma_t H), \quad H = - \tilde{\mathcal{M}}_{\lambda} (\gamma_t H) - \frac{1}{\rmi \lambda} \tilde{\mathcal{L}}_\lambda (\gamma_t E).
$$
Hence, the operator
$$
 P_+ =  \left( \begin{matrix} \frac{1}{2} - \mathcal{M}_\lambda & \frac{1}{\rmi \lambda} \mathcal{L}_\lambda \\ -\frac{1}{\rmi \lambda} \mathcal{L}_\lambda & \frac{1}{2} - \mathcal{M}_\lambda \end{matrix} \right)
$$
acting on $H^{-\frac{1}{2}}(\mathrm{Div},\partial \Omega) \oplus H^{-\frac{1}{2}}(\mathrm{Div},\partial \Omega)$
is a projection onto the space of boundary data of solutions of Maxwell's equation in $H(\curl,\Omega) \oplus H(\curl,\Omega)$. This map is called the interior Calderon projector. In the same way the exterior Calderon projector
$P_-$  acting on $H^{-\frac{1}{2}}(\mathrm{Div},\partial \Omega) \oplus H^{-\frac{1}{2}}(\mathrm{Div},\partial \Omega)$ is given by 
$$
 P_- =  \left( \begin{matrix} \frac{1}{2} + \mathcal{M}_\lambda & -\frac{1}{\rmi \lambda} \mathcal{L}_\lambda \\ \frac{1}{\rmi \lambda} \mathcal{L}_\lambda & \frac{1}{2} + \mathcal{M}_\lambda \end{matrix} \right).
$$
It projects onto the space of boundary data of solutions of Maxwell's equation in $H(\curl,\Omega) \oplus H(\curl,\Omega)$
when $\Im(\lambda)>0$ and more generally solutions satisfying a radiation condition for non-zero real $\lambda$.
As usual one has $P_+ + P_- = \mathrm{id}$. 

We now define the voltage-to-current mappings $\Lambda_{\lambda}^{\pm}: H^{-\frac{1}{2}}(\mathrm{Div},\partial \Omega) \to H^{-\frac{1}{2}}(\mathrm{Div},\partial \Omega)$ by
\begin{align}
\Lambda_{\lambda}^{\pm}: \gamma_t(E) \rightarrow \gamma_t(H)
\end{align}
where $(E,H)$ are solutions to the interior and exterior boundary value problem for the Maxwell system \eqref{system}, respectively, whenever these solutions are unique.  
The graphs of $\Lambda_{\lambda}^{\pm}$ in  $H^{-\frac{1}{2}}(\mathrm{Div},\partial \Omega) \oplus H^{-\frac{1}{2}}(\mathrm{Div},\partial \Omega)$ are therefore by definition the ranges of the Calderon projectors $P_\pm$.
The voltage to current maps are henceforth the Maxwell analogues of the interior and exterior Helmholtz Dirichlet to Neumann maps.

The mapping 
$\Lambda_{\lambda}^{+}$ is well defined for any $\lambda \in \C$ which is not a Maxwell eigenvalue or zero. The mapping 
$\Lambda_{\lambda}^{-}$ is well defined for all non-zero $\lambda$ in the closed upper half space. 
In this case these are bounded operators on $H^{-\frac{1}{2}}(\Div,\partial\Omega)$. We will see later that these operators extend meromorphically to the complex plane. In anticipation of this we will not explicitly state the domains when dealing with algebraic identities.
As a consequence of the symmetry $(E,H) \mapsto (H,-E)$ of the Maxwell system and the above relations one obtains the formulae
\begin{align} \label{voltcur}
(\Lambda_{\lambda}^{\pm})^2=-\id \quad \mathrm{and} \quad 
\mathcal{L}_{\lambda}=\rmi \lambda\Lambda_{\lambda}^{\pm}\left(\mp \frac{1}{2}+\mathcal{M}_{\lambda}\right) = -\rmi \lambda \left(\pm \frac{1}{2}+\mathcal{M}_{\lambda}\right) \Lambda_{\lambda}^{\pm},
\end{align}
and as a consequence 
\begin{align} \label{twenty}
-\rmi \lambda^{-1} \calL_\lambda (\Lambda^+_\lambda - \Lambda^-_\lambda) = \id \quad \mathrm{and} \quad \calL_\lambda^2 = -\lambda^2\left(- \frac{1}{2} + \calM_\lambda \right)\left( \frac{1}{2} + \calM_\lambda \right).
\end{align}
These are also manifestations of the Calderon projector being a projection mapping, i.e. $P_\pm^2 = P_\pm$.
We refer to \cite{mitrea1997}*{Lemma 5.10} for these and more statements in the $L^2$-setting.
Notice that we are using the opposite sign convention for $\tilde{\mathcal{S}}_{\lambda}$ than in \cite{mitrea1997}. 

For later reference and completeness we also state the following identities. 

\begin{lemma} \label{lemrel}
For $A \in H^{-\frac{1}{2}}(\Div,\partial \Omega)$ and $f \in H^{\frac{1}{2}}(\partial \Omega)$ we have
\begin{align}
  \div \,\tilde{\mathcal{S}}_\lambda A &= \tilde{\mathcal{S}}_\lambda \Div A,	\label{divsrel}\\
  \curl \,\tilde{\mathcal{S}}_\lambda \nu f &= - \tilde{\mathcal{S}}_\lambda (\nu \times \nabla f),\\
  \Div \mathcal{M}_\lambda A &= -\lambda^2 \nu \cdot \mathcal{S}_\lambda A - \mathcal{K}^\trans_\lambda (\Div A),\\
  (\nu \times \nabla) \mathcal{K}_\lambda f &= \lambda^2 \nu \times \mathcal{S}_\lambda (\nu f) + \mathcal{M}_\lambda (\nu \times \nabla f),\\
   (\nu \times \nabla) \mathcal{K}_0 f &= \mathcal{M}_0 (\nu \times \nabla f). \label{intertwin}
\end{align}
\end{lemma}

These identities were for example proved in \cite{mitrea1997} (Lemmata 4.2, 4.3, 4.4, and 5.11) in slightly different function spaces containing the image of $C^\infty_0(\R^3,\C^3)$ under the tangential restriction map $\gamma_t$. Since $C^\infty_0(\R^3,\C^3)$ is a dense subspace in $H(\curl,\R^3)$ the space $\gamma_t C^\infty_0(\R^3,\C^3)$ is dense in $H^{-\frac{1}{2}}(\Div,\partial \Omega)$. Hence, these equations extend by continuity to the claimed larger space if we use the continuous mapping properties of the potential layer operators.
We note here that the gradient $\nabla$ defines a continuous map $H^{\frac{1}{2}}(\partial \Omega) \to H^{-\frac{1}{2}}(\Curl,\partial \Omega)$
and the map $\nu \times \nabla$ is continuous from $H^{\frac{1}{2}}(\partial \Omega) \to H^{-\frac{1}{2}}(\Div,\partial \Omega)$.

\begin{lemma}\label{transposelemma}
 The map $\mathcal{S}_\lambda$ satisfies $\mathcal{S}_\lambda^* = \mathcal{S}_{\overline{\lambda}}$, where the adjoint is taken with respect to the $L^2$-induced dual pairing between
 $H^{\frac{1}{2}}(\partial \Omega)$ and $H^{-\frac{1}{2}}(\partial \Omega)$. In other words it is its own transpose, $\mathcal{S}_\lambda^\trans = \mathcal{S}_\lambda$.
  We also have $(\mathcal{L}_\lambda (\nu \times))^\trans = \mathcal{L}_\lambda (\nu \times)$, i.e. $\mathcal{L}_\lambda$ is symmetric with respect to the bilinear form induced by $\langle \cdot, \nu \times \cdot \rangle$.
\end{lemma}
\begin{proof}
 The symmetry of the operator $\mathcal{S}_\lambda$ with respect to the real inner product are classical and follow from the symmetry properties of the integral kernel.
See for example Theorem 5.44 in \cite{kirsch}. The statement about $\mathcal{L}^\trans_\lambda$ is Lemma 5.6.1 in \cite{kirsch}.
\end{proof}

The following Lemma is implicit in \cite{kirsch}.
\begin{lemma} \label{invertfred}
 The operator $\pm\frac{1}{2} + \mathcal{M}_\lambda$ is for any $\Im(\lambda)>0$ an isomorphism from $H^{-\frac{1}{2}}(\Div,\partial\Omega)$ to $H^{-\frac{1}{2}}(\Div,\partial\Omega)$.
\end{lemma}
\begin{proof} Assume that $\Im(\lambda)>0$.
 It was shown in \cite{kirsch}*{Theorem 5.52,(d)} that $\mathcal{L}_{\lambda}$ is invertible modulo compact operators and therefore is a Fredholm operator of index zero.
Moreover, by \cite{kirsch}*{Theorem 5.59} we know that $\mathcal{L}_{\lambda}$ is injective and hence invertible.
 Since $\Lambda^\pm_\lambda$ are invertible it follows from \eqref{voltcur} that $\pm\frac{1}{2}+\mathcal{M}_{\lambda}$ is. As usual the inverse is continuous by the open mapping theorem.
\end{proof}
Invertibility of operators $\pm\frac{1}{2}+\mathcal{M}_{\lambda}$ on several other $L^p$-spaces has been shown in the works of M. Mitrea and D. Mitrea. 
 (for example Theorem 4.1 in \cite{mitrea1995}). 
 
 \begin{proposition} \label{fredprop}
 The family $\pm\frac{1}{2} + \mathcal{M}_\lambda$ is a holomorphic family of Fredholm operators of index zero from $H^{-\frac{1}{2}}(\Div,\partial \Omega)$ to  $H^{-\frac{1}{2}}(\Div,\partial \Omega)$. The derivative $ \mathcal{M}'_\lambda = \frac{\der}{\der \lambda} \mathcal{M}_\lambda$ is a continuous family of Hilbert-Schmidt operators on $H^{-\frac{1}{2}}(\Div,\partial \Omega)$.
\end{proposition}
\begin{proof}
 We will show that $\mathcal{M}_\lambda$ is complex differentiable as a family of bounded operators $H^{-\frac{1}{2}}(\Div,\partial \Omega)$ and its derivative is compact.
 The first part of the theorem then follows from 
 $$
  (\pm\frac{1}{2} + \mathcal{M}_\lambda) - (\pm\frac{1}{2} + \mathcal{M}_\rmi) = \int_\rmi^\lambda  \mathcal{M}'_\mu \der \mu
 $$
 and the proposition above. We have used here that Fredholm operators are stable under compact perturbations (see for example Lemma 8.6 in \cite{shubin}).
 It is therefore sufficient to show that $ \mathcal{M}'_\lambda$ exists and is Hilbert-Schmidt. First choose a compactly supported smooth cut-off function $\chi$ supported in $(-2R,2R)$ which equal to one on $[-R,R]$, for sufficiently large $R>0$. The integral kernel of $\tilde \calM_{\pm,\lambda}$ is given by
 $\curl_x \frac{e^{\rmi \lambda |x-y|}}{4 \pi |x-y|}$. For $x$ not far from $\partial \Omega$ we can replace this by $\chi(|x-y|)\curl_x \frac{e^{\rmi \lambda |x-y|}}{4 \pi |x-y|}$.
 Consider the following Taylor expansion 
 \begin{align*}
  &\chi(|x-y|)\curl_x \frac{e^{\rmi \lambda |x-y|}}{4 \pi |x-y|} =\\& \chi(|x-y|) \curl_x \frac{e^{\rmi \mu |x-y|}}{4 \pi |x-y|} +\chi(|x-y|) \curl_x \frac{e^{\rmi \mu |x-y|}}{4 \pi} (\lambda - \mu) + \chi(|x-y|)T_{\lambda}(x-y) (\lambda -\mu)^2
 \end{align*}
 with remainder term $T_{\lambda}$. This gives rise to an operator expansion
 $$
  \calM_{\pm,\lambda} = \calM_{\pm,\mu} + A _\lambda(\lambda - \mu) + B_\lambda (\lambda - \mu)^2.
 $$
 Here the operators $A_\lambda$ and $B_\lambda$ arise as compositions as
 $$
  H^{-\frac{1}{2}}(\Div,\partial \Omega) \overset{\gamma_T^*}\longrightarrow H^{-1}(U) \overset{K_A,K_B}\longrightarrow H^{1}(\R^d) \longrightarrow H(\curl,M) \overset{\gamma_t}\longrightarrow H^{-\frac{1}{2}}(\Div,\partial \Omega),
 $$
 where $K_{A}$ or $K_B$ is the integral operator with kernel $\chi(|x-y|) \curl_x \frac{e^{\rmi \mu |x-y|}}{4 \pi}$ or  $\chi(|x-y|)T_{\lambda}(x-y)$ respectively. Here $U$ is a bounded open neighborhood of $\partial \Omega$.
 It is now sufficient to show that the operator $K_A,K_B$ are bounded as Hilbert-Schmidt operators. In view of Lemma \ref{HilSobLemma} we would like to bound the
 $H^2(\R^d \times \R^d)$-norm of the kernels. Taking two derivatives gives in both cases an integrable convolution kernel in $L^1(\R^d)$ and the 
 $H^2(\R^d \times \R^d)$-norm is then, by Young's inequality, bounded by the $L^1$-norm of this kernel.
\end{proof}

\begin{definition}
 The spaces $\mathcal{B}^\pm_{\partial\Omega} \subset H^{-\frac{1}{2}}(\Div,\partial \Omega)$ of interior/exterior boundary data of absolute harmonic forms is defined as
 \begin{align*}
 \mathcal{B}^+_{\partial\Omega} = \{\gamma_{t,+}(\phi) \mid \phi\in \mathcal{H}^1_{\abs}(\Omega)\},\\
 \mathcal{B}^-_{\partial\Omega} = \{\gamma_{t,-}(\phi) \mid \phi\in \mathcal{H}^1_{\abs}(M)\}.
 \end{align*}
\end{definition}
It is then obvious that $\mathcal{B}^+_{\partial\Omega} = \mathcal{B}^+_{\partial\Omega_1} \oplus \ldots \oplus  \mathcal{B}^+_{\partial\Omega_N}$ with respect to the decomposition
$$H^{-\frac{1}{2}}(\Div,\partial \Omega)= H^{-\frac{1}{2}}(\Div,\partial \Omega_1) \oplus \ldots \oplus H^{-\frac{1}{2}}(\Div,\partial \Omega_N).$$
This is not true for the space $\mathcal{B}^-_{\partial\Omega}$. The spaces $\mathcal{B}^+_{\partial\Omega}$ are also known to be subspaces of $L^2(\partial\Omega,\C^3)$, see \cite{MR1809655}*{Th. 11.2}, but this will not be needed.

 The following was announced in \cite{dmitrea} by D. Mitrea in the context of $L^p$-spaces, with $p$ sufficiently close to $2$. It is a reflection of general Hodge theory for Lipschitz domains and we restate and prove this here for our choice of function spaces.

\begin{proposition}\label{dmitreazero}
We have
\begin{align}
\mathcal{B}^\pm_{\partial \Omega} =\mathrm{ker}\left(\pm\frac{1}{2} + \mathcal{M}_0\right)\subset H^{-\frac{1}{2}}(\Div 0,\partial\Omega) .
\end{align} 
\end{proposition}
\begin{proof}
We will prove this only in case $\mathcal{B}^+_{\partial \Omega}$ since the proof for $\mathcal{B}^-_{\partial \Omega}$, when supplemented by Lemma \ref{vanishcoho}, is exactly the same.
 Suppose that $u \in \mathrm{ker}(\frac{1}{2} + \mathcal{M}_0)$ and define $\phi = -\tilde{\mathcal{M}}_0 u$. Then $\phi$
 is divergence-free and harmonic on $M$ and on $\Omega$. The jump relations \eqref{jumpjumpjump} hold by analytic continuation for all $\lambda \in \C$ and they show that $\gamma_{t,-} \phi=0$ and $\gamma_{t,+} \phi=u$ and  $\gamma_{\nu,+} \phi= \gamma_{\nu,-} \phi$. We first show that $q = \gamma_{\nu,+} \phi$ vanishes, thus establishing the inclusion $\phi|_{\Omega} \in  \mathcal{H}^1_{\abs}(\Omega), \gamma_{t,+} \phi = u$. The proof uses similar arguments as in \cite{MR769382} and reflects the mapping properties of the adjoint double layer operator.

On the exterior $\phi$ is a harmonic vector-field satisfying relative boundary conditions. The decay of $\curl \frac{1}{|x-y|}$ implies that $\phi$ is square integrable. This shows that $\curl \phi$ must vanish in the exterior.
From the representation \eqref{singlerepext} we obtain, using the jump relations and $\gamma_{t,-} \phi=0$,
 $$
    \phi|_M =  \nabla \tilde{\mathcal{S}}_0 q.
 $$
Taking the normal trace one gets $q = \gamma_{\nu,-}\nabla \tilde{\mathcal{S}}_0 q$.
Taking the tangential trace one obtains from the jump relations
 \begin{align*}
   \gamma_{t,-} \nabla \tilde{\mathcal{S}}_0 (\gamma_{\nu,-} \phi) = \nabla_{\partial \Omega} \mathcal{S}_0 q =0.
 \end{align*}
This shows that $w = \mathcal{S}_0 q$ is locally constant (and in particular in $L^2(\partial \Omega)$). 
Using the divergence theorem on the interior of each of the components $\Omega_j$ one finds that
$\int_{\partial \Omega_j} q =0$. This gives $\langle \mathcal{S}_0 q, q \rangle_{L^2(\partial \Omega)} =0$ and therefore 
$$
 \langle \mathcal{S}_0 q,  \nabla_{\nu} \tilde{\mathcal{S}}_0 q \rangle_{L^2(\partial \Omega)}  =0.
$$
Since this is the boundary term in the integration by parts formula for $\langle  \nabla \tilde{\mathcal{S}}_0 q ,  \nabla \tilde{\mathcal{S}}_0 q\rangle =0$ which then implies that $\tilde{\mathcal{S}}_0 q$ is constant. Since it decays we must have $\tilde{\mathcal{S}}_0 q=0$ and therefore $\mathcal{S}_0 q=0$. By invertibility of the single layer operator one obtains $q=0$ as claimed.

We now show the inclusion in the other direction. Suppose that $u = \gamma_{t,+}(h)$, where $h \in \mathcal{H}^1_{abs}(\Omega)$.
This means in particular that $h$ is divergence-free, curl-free, and $\gamma_{\nu,+} h=0$.
Taking the tangential trace in representation \eqref{singlerepint} we obtain
$$
 u = (\frac{1}{2} - \mathcal{M}_0) u
$$
and therefore $(\frac{1}{2} + \mathcal{M}_0) u=0$ as claimed.

It finally remains to show that $\{\gamma_{t,+}(\phi) \mid \phi\in \mathcal{H}^1_{\abs}(\Omega)\} \subset H^{-\frac{1}{2}}(\Div 0,\partial\Omega)$.
This follows immediately from the fact that $\curl \phi =0$ and $\Div \circ \gamma_{t,+} = - \gamma_{\nu,+} \circ  \curl$. 
\end{proof}

A similar but easier argument applies to other elements of the real line and gives the following.
\begin{proposition}\label{menotzero}
If $\lambda = \R \setminus \{0\}$ then $\mathrm{ker}\left(\frac{1}{2} + \mathcal{M}_\lambda \right) = \{0\}$ in case $| \lambda | \not= \mu_k$ for all $k \in \N$, i.e. $|\lambda|$ is not a Maxwell eigenvalue. Moreover,
\begin{align}
 \mathrm{ker}\left(\frac{1}{2} + \mathcal{M}_{\mu_k}\right) = \{ \gamma_{t,+}(u) \mid u \in V_{\mu_k} \},
\end{align} 
where $V_{\mu_k}$ is the  eigenspace of $\Delta_{\Omega,\abs}$ for the eigenvalue $\mu_k^2$ on the subspace of divergence-free vector-fields.
\end{proposition}
\begin{proof}
 The proof is very similar to the proof of the previous Proposition and we therefore only give a brief sketch.
 As before let $u \in \mathrm{ker}\left(\frac{1}{2} + \mathcal{M}_{\mu_k}\right)$ and $\phi = - \tilde{\calM}_\lambda u$. Then $\phi|_M$ is a purely incoming or outgoing solution of the Helmholtz equation (see e.g. \cite{OS}*{Appendix C} for details) satisfying relative boundary conditions. 
 It therefore vanishes. By the jump relations \eqref{jumpjumpjump}
 the function $\phi|_\Omega$ satisfies absolute boundary conditions, is divergence-free, and is a Maxwell eigenfunction with Maxwell eigenvalue $\mu_k$.
 Moreover, again by the jump-relation, $\gamma_{t,+}\phi= u$. This proves the inclusion in one direction.
 Conversely, assume that  $u=\gamma_{t,+}\phi$, where $\phi$ is divergence-free, satisfies absolute boundary conditions, and $-\Delta \phi = \mu_k^2 \phi$.
 Taking the tangential trace in representation \eqref{singlerepint} we obtain
$$
 u = (\frac{1}{2} - \mathcal{M}_{\mu_k}) u
$$
and therefore $(\frac{1}{2} + \mathcal{M}_{\mu_k}) u=0$ as claimed.
\end{proof}

\section{Estimates and low energy expansions for the Layer potential operators} \label{estimateSec}

For $0 < \epsilon< \frac{\pi}{2}$, define the sector $\mathfrak{D}_\epsilon$ in the upper half plane by  
\begin{align*}
 \mathfrak{D}_\epsilon := \{ z \in \C \mid \epsilon < \arg(z) < \pi - \epsilon \}.
\end{align*}
The next proposition establishes properties of the single layer operator $\tilde{\mathcal{S}}_\lambda$ and the operator $\tilde{\mathcal{L}}_\lambda$. 
\begin{proposition} \label{Prop:EstimateS}
 For $\epsilon \in (0,\frac{\pi}{2})$, for all $\lambda\in\mathfrak{D}_{\epsilon}$ we have the following bounds:
 \begin{enumerate}
  \item \label{einsbound} Let $\Omega_0\subset \R^d$ be an open subset and assume $\delta=\mathrm{dist}(\Omega_0,\partial \Omega) > 0$. Let $0 < \delta'< \delta$.
  Assume that $\varphi \in C^1_b(\R^3)$ is bounded with bounded derivative and supported in $\Omega_0$. For each $\lambda \in \mathfrak{D}_\epsilon$ the operators
  \begin{align*}
    \varphi \tilde{\mathcal{L}}_{\lambda}:& H^{-\frac{1}{2}}(\mathrm{Div},\partial \Omega) \to H(\mathrm{curl}, \mathbb{R}^3),\\
    \varphi \tilde{\mathcal{S}}_{\lambda}:& H^{-\frac{1}{2}}(\partial \Omega) \to H^1(\mathbb{R}^3),\\
    \varphi \nabla \tilde{\mathcal{S}}_{\lambda}:& H^{-\frac{1}{2}}(\partial \Omega) \to L^2(\mathbb{R}^3),\\
    \varphi \tilde{\mathcal{M}}_{\lambda}:& H^{-\frac{1}{2}}(\Div,\partial \Omega) \to H(\div,\mathbb{R}^3)
  \end{align*}
 are Hilbert-Schmidt operators. There  exists $C_{\delta',\epsilon}>0$ such that for all $\lambda \in \mathfrak{D}_\epsilon$ we have the following bounds on the Hilbert-Schmidt norms between these spaces
   \begin{align} \label{eqn:HSEstimateL}
     \|\varphi \tilde{\mathcal{L}}_{\lambda} \|_{\mathrm{HS}} &\leq C_{\delta',\epsilon}e^{-\delta' \Im{\lambda}}, \\ \label{eqn:HSEstimateS} \|\varphi \tilde{\mathcal{S}}_{\lambda} \|_{\mathrm{HS}} &\leq |\lambda|^{-\frac{1}{2}}C_{\delta',\epsilon}e^{-\delta' \Im{\lambda}},\\ 
     \|\varphi \nabla \tilde{\mathcal{S}}_{\lambda} \|_{\mathrm{HS}} &\leq C_{\delta',\epsilon}e^{-\delta' \Im{\lambda}}, \label{eqn:nablaHSEstimateS}\\
      \|\varphi \tilde{\mathcal{M}}_{\lambda} \|_{\mathrm{HS}} &\leq C_{\delta',\epsilon}e^{-\delta' \Im{\lambda}}, \label{eqn:HSMEstimateS}\\
     \|\varphi \tilde{\mathcal{S}}_{\lambda} \Div \|_{\mathrm{HS}} &\leq C_{\delta',\epsilon}e^{-\delta' \Im{\lambda}}. \label{eqn:HSEstimateSdiv}
    \end{align}
     \item   For $\lambda \in \mathfrak{D}_\epsilon$ we have the operator-norm bound
   \begin{align} \label{zweiboundL}
  \|\tilde{\mathcal{L}}_\lambda \|_{H^{-\frac{1}{2}}(\mathrm{Div},\partial\Omega) \to H(\mathrm{curl},\mathbb{R}^3)} \leq C_{\epsilon}(1 + |\lambda|^2).
  \end{align}
  \item For $\lambda \in \mathfrak{D}_\epsilon$ we have the operator-norm bound
   \begin{align} \label{zweiboundM}
  \|\tilde{\mathcal{M}}_\lambda \|_{H^{-\frac{1}{2}}(\Div,\partial\Omega) \to L^2(\R^3,\C^3)} &\leq  C_{\epsilon}.
  \end{align}
   \item For $\lambda \in \mathfrak{D}_\epsilon$ we have the operator-norm bounds
   \begin{align} \label{zweiboundS}
  \|\tilde{\mathcal{S}}_\lambda \|_{H^{-\frac{1}{2}}(\partial\Omega) \to H^1(\mathbb{R}^3)} &\leq C_{\epsilon}|\lambda|^{-\frac{1}{2}} (1 + |\lambda|^\frac{1}{2}),\\
  \|\nabla \tilde{\mathcal{S}}_\lambda \|_{H^{-\frac{1}{2}}(\partial\Omega) \to L^2(\mathbb{R}^3,\C^3)} &\leq C_{\epsilon},\label{zweiboundnablaS}
  \end{align}
  \item On the space of functions of mean zero $H^{-\frac{1}{2}}_0(\partial\Omega) = \{u \in H^{-\frac{1}{2}}_0(\partial\Omega)  \mid \langle u, 1\rangle=0 \}$ we have for $\lambda \in \mathfrak{D}_\epsilon$ the improved estimate
  \begin{align} \label{zweiboundSzero}
  \|\tilde{\mathcal{S}}_\lambda |_{H_0^{-\frac{1}{2}}(\partial\Omega)} \|_\mathrm{HS} \leq C_{\epsilon}.
  \end{align}
 \end{enumerate}
  \end{proposition}
  \begin{proof}
The operator $\varphi \tilde{\mathcal{L}}_{\lambda}$ can be written as $\varphi \,\curlcurl\; G_{\lambda,0} \gamma_T^*$. Similarly, we have
$\varphi \tilde{\calM}_\lambda = \varphi \,\curl\, G_{\lambda,0} \gamma_T^*$ and $\varphi \tilde{\calS}_\lambda = \varphi G_{\lambda,0} \gamma_T^*$.
We choose a bounded open neighborhood $U$ of $\partial \Omega$ such that $\mathrm{dist}(\Omega_0,U) > \delta'$.
Since $\gamma_T^*$ continuously maps $H^{-\frac{1}{2}}(\partial \Omega)$ to $H^{-1}(U)$ we only need to show that the map $\curlcurl\, G_\lambda$
is a Hilbert-Schmidt operator from $H^{-1}(U)$ to $H^1(\Omega_0)$ and establish the corresponding bound on its Hilbert-Schmidt norm.
By Lemma \ref{HilSobLemma} the Hilbert-Schmidt norm can be bounded by the $H^2(\Omega_0 \times U)$-norm of the kernel of $\curlcurl\, G_{\lambda,0}$
on $\Omega_0 \times U$. The corresponding bound has been established in Lemma \ref{Prop:EstL2PGfirst}. The same argument  works for $\varphi \tilde{\calM}_\lambda$ and $\varphi \tilde{\mathcal{S}}_{\lambda}$.
 This concludes the proof of the estimates  \eqref{eqn:HSEstimateL}, \eqref{eqn:HSEstimateS}, \eqref{eqn:HSMEstimateS}, \eqref{eqn:nablaHSEstimateS}, \eqref{eqn:HSEstimateSdiv}.
 
Since the operator norm is bounded in terms of the Hilbert-Schmidt norm and by the estimates \eqref{eqn:HSEstimateL},\eqref{eqn:HSEstimateS},\eqref{eqn:nablaHSEstimateS},\eqref{eqn:HSMEstimateS} it is sufficient to prove 
the estimates \eqref{zweiboundL}, \eqref{zweiboundM}, \eqref{zweiboundS}, and \eqref{zweiboundnablaS}  for the operators $\chi \tilde{\mathcal{L}}_\lambda$,  $\chi \tilde{\mathcal{S}}_\lambda$, $\chi \tilde{\mathcal{M}}_\lambda$, $\chi \nabla \tilde{\mathcal{S}}_\lambda$ where 
$\chi \in C^\infty_0(\R^3) $ is a compactly supported function that equals one near $\partial \Omega$. We write
 $$
  \chi \tilde{\mathcal{L}}_\lambda =  \chi \nabla \tilde S_\lambda \Div + \lambda^2 \chi  \tilde S_\lambda.
 $$
 The map $\gamma_t^*$ is from $H^{-\frac{1}{2}}(\partial \Omega)$ to $H^{-1}_\comp(U)$ where $U$ is an open neighborhood of $\partial \Omega$.
 To prove both bounds \eqref{zweiboundL}, \eqref{zweiboundM}, \eqref{zweiboundS}, and \eqref{zweiboundnablaS} it is therefore sufficient to show that the resolvent $(-\Delta_\free -\lambda^2)^{-1}$ is a bounded map
 from $H^{-1}_\mathrm{comp}(\R^3)$ to $H^1_\mathrm{loc}(\R^3)$ uniformly in $\lambda$ for all $\lambda \in \mathfrak{D}_\epsilon$. This means that we need to show that the cut-off resolvent $\chi (-\Delta_\free -\lambda^2)^{-1} \chi$ is a uniformly bounded map
 from $H^{-1}(\R^3)$ to $H^1(\R^3)$ for all $\lambda \in \mathfrak{D}_\epsilon$. To see this let $\eta \in C^\infty_0(\R)$ be a function that is one near $[-R_1,R_1]$, where $R_1$ is the diameter of the support of $\chi$. Let $R$ be large enough so  that $\supp \eta \in (-R,R)$.
 This implies that $\chi (-\Delta -\lambda^2)^{-1} \chi = \chi R_{\eta,\lambda} \chi$, where $R_{\eta,\lambda}$  is the operator
 with integral kernel
 $$
  \eta(|x-y|) \frac{1}{4 \pi |x-y|} e^{\rmi \lambda |x-y|} =: k_\lambda(x-y).
 $$
 It is therefore sufficient to show that $R_{\eta,\lambda}$ is uniformly bounded for all $\lambda \in \mathfrak{D}_\epsilon$ as a map $H^s(\R^3)$ to $H^{s+2}(\R^3)$. Since this is a convolution operator it commutes with the Laplace operator and therefore it is sufficient to show that $R_{\eta,\lambda}$ is uniformly bounded as a map
 $L^2(\R^3)$ to $H^{2}(\R^3)$.  We will show that $(-\Delta +1)R_{\eta,\lambda}$ is uniformly bounded as a map from $L^2(\R^3)$ to $L^{2}(\R^3)$.
 Using $$(-\Delta +1) (-\Delta -\lambda^2)^{-1} = \mathrm{id} + (1 + \lambda^2) (-\Delta -\lambda^2)^{-1}$$ one obtains that the integral kernel of 
 $(-\Delta +1)R_{\eta,\lambda} - \mathrm{id}$ equals
 $$
  \left( -(\Delta_x \eta(|x-y|)) + (1+ \lambda^2) \eta(|x-y| \right) \frac{1}{4 \pi |x-y|} e^{\rmi \lambda |x-y|} - 2\nabla_x \eta(|x-y|) \nabla_x \frac{1}{4 \pi |x-y|} e^{\rmi \lambda |x-y|}.
 $$ 
 This is a convolution operator and we can use Young's inequality to estimate its operator norm. In particular, using spherical coordinates, the estimates
 \begin{align*}
  \int_0^R  \frac{1}{4 \pi r} | e^{\rmi \lambda r}| r^2 \der r   \leq C_R \frac{1}{1+ |\Im \lambda|^2},\\
   \int_0^R  \frac{1}{4 \pi r^2} | e^{\rmi \lambda r}| r^2 \der r   \leq C_R \frac{1}{1+ |\Im \lambda|},\\
 \end{align*}
 show that the convolution kernel is uniformly bounded in $L^1(\R^3)$ for $\lambda \in \mathfrak{D}_\epsilon$.
Thus $R_{\eta,\lambda}$ is uniformly bounded as a map from $L^2(\R^3)$ to $H^{2}(\R^3)$ for $\lambda \in \mathfrak{D}_\epsilon$.

It remains to show the improved estimate \eqref{zweiboundSzero}. We again choose cut-offs $\chi, \psi$ as above and we arrange them so that $\psi + \phi=1$.
Since the cut-off resolvent $\chi (-\Delta_\free -\lambda^2)^{-1} \chi$ is regular near zero as a map $H^{-1}(\R^3)$ to $H^{1}(\R^3)$ we know that
$\chi \tilde{\mathcal{S}}_\lambda: H^{-\frac{1}{2}}(\partial \Omega) \to H(\curl,\R^3)$ is regular near zero. It is therefore sufficient to establish the bound
for $\phi \tilde{\mathcal{S}}_\lambda$ as a map from $H_0^{-\frac{1}{2}}(\partial \Omega)$ to $H^{1}(\R^3)$. We argue similarly as above choosing an open neighborood $U$ such that the support of $\phi$ has positive distance from $U$. For convenience we will also assume that the support of $\phi$ is sufficiently separated from
$\Omega$, more precisely we assume that the support of $\phi$ has positive distance to the convex hull of $\Omega$.
With $ u \in H_0^{-\frac{1}{2}}(\partial \Omega)$ the distribution
$\gamma^*u$ is in the space distributions $H_0^{-1}(U)  = \{v \in H^{-1}_\comp(U) \mid \langle v, 1\rangle =0 \}$ of mean zero. We therefore only need to bound
$\phi (-\Delta_\free - \lambda^2)^{-1}$ as a map from $H_0^{-1}(U)$ to  $H^1(\R^n)$. This map is the restriction of the integral operator with smooth kernel
$$
  g(x,y)=\phi(x) \left(\frac{e^{\rmi \lambda |x-y|}}{4 \pi |x-y|} - \frac{e^{\rmi \lambda |x-z|}}{4 \pi |x-z|} \right)
$$
to $H_0^{-1}(U)$, where $z$ is any fixed point on $\partial \Omega$. One shows that this kernel is in the Sobolev space $H^2(\R^3 \times U)$ and is uniformly bounded in $\lambda \in  \mathfrak{D}_\epsilon$. This kernel and its derivatives are easily bounded using the mean-value inequality
$$
 |\partial^\alpha_x g(x,y)| \leq |y-z| \sup\limits_{\tilde y \in K}\| \partial^\alpha \nabla_x g(x,\tilde y) \| \leq C \sup\limits_{\tilde y \in K}\| \partial^\alpha \nabla_x g(x,\tilde y) \|,
$$
where $K$ is the closure of the convex hull of $\partial \Omega$.
The $L^2$-norm of this expression is uniformly bounded for all $\lambda \in  \mathfrak{D}_\epsilon$ by the same estimate as in \eqref{eqn:MainL2Integral2}. 
This works essentially because with repeated application of the product rule the terms either have improved decay or have an extra $\lambda$-factor. 
\end{proof}

 The proof above can also be applied directly to $\chi \tilde{\mathcal{L}}_\lambda$ in the entire complex plane to bound the operator norm, the norm of the the derivative and the norm of the remainder term. This gives the following result. We will not repeat the proof but simply state the result.
 
 \begin{lemma}
  The families $\tilde{\mathcal{L}}_\lambda: H^{-\frac{1}{2}}(\Div,\partial \Omega) \to H_\loc(\curl,M) \oplus H_\loc(\curl,\Omega)$ 
  and $\mathcal{L}_\lambda: H^{-\frac{1}{2}}(\Div,\partial \Omega) \to  H^{-\frac{1}{2}}(\Div,\partial \Omega)$ are holomorphic families of bounded operators in the complex plane.
 \end{lemma}

 \begin{lemma} \label{merolemma}
  The families $\mathcal{L}_\lambda^{-1}, \Lambda^{\pm}_\lambda$ are meromorphic in $\lambda$ as families of bounded operators on $H^{-\frac{1}{2}}(\Div,\partial \Omega)$. The family $\Lambda^-_\lambda$ has no poles in $\R \setminus \{0\}$ and in the upper half-plane.
 \end{lemma}
  \begin{proof}
  By Prop. \ref{fredprop} and Lemma \ref{invertfred} the operator $\left(\frac{1}{2}+\mathcal{M}_{\lambda}\right)$ is an analytic family of Fredhom operators which is invertible for $\Im{\lambda}>0$. By the analytic Fredholm theorem the inverse
$\left(\frac{1}{2}+\mathcal{M}_{\lambda}\right)^{-1}$ is a meromorphic family of finite type, i.e. the negative Laurent coefficients are finite rank operators.
   We have
   $$
    \mathcal{L}_\lambda^{-2} = -\lambda^{-2} (\frac{1}{2} + \calM_\lambda)^{-1}(\frac{1}{2} - \calM_\lambda)^{-1}
   $$
   which shows that $\mathcal{L}_\lambda^{-2}$ is meromorphic. Since $\mathcal{L}_\lambda$ is holomorphic  this shows that $\mathcal{L}_\lambda^{-1}$ is meromorphic. Finally $\Lambda^\pm$ is meromorphic by \eqref{voltcur}. 
   Poles of $\Lambda^-_\lambda$ are absent in the closed upper half space because of the uniqueness of the exterior boundary value problem. Indeed, the most negative Laurent coefficient would give rise to an outgoing solution of the Helmholtz equation satisfying relative boundary conditions. But such an outgoing solution vanishes.
 \end{proof}
 
 \begin{rem}
  The above cannot be easily concluded from analytic Fredholm theory since the operators $\mathcal{L}_\lambda$ and $\Lambda^{\pm}$ are not Fredholm operators. Indeed, the singular Laurent coefficients are not finite rank operators.
 \end{rem}

 We now aim to show a new formula for the voltage-to-current map in order to find bounds on $\mathcal{L}^{-1}_{\lambda}$ where it is well defined. 

 \begin{theorem}\label{eigenrep}
 The interior voltage-to-current mapping $\Lambda_{\lambda}^+$  satisfies
 $$
 \rmi \Lambda^+_{\lambda} = \frac{1}{\lambda} T + \lambda U_\lambda,
 $$
 where $T$ is a bounded operator on $H^{-\frac{1}{2}}(\Div, \partial \Omega)$ and $U_\lambda$ is a meromorphic family of bounded operators on
 $H^{-\frac{1}{2}}(\Div, \partial \Omega)$ which is regular at $\lambda=0$. We have explicitly 
 \begin{align*}
  T A &= \sum\limits_{k=1}^{\beta_1}\langle A,\gamma_T \psi_{0,k}\rangle_{L^2(\partial\Omega)}\gamma_t(\psi_{0,k}) + \sum\limits_{\lambda_{N,k} > 0}\frac{1}{\lambda_{N,k}^2} \langle A,\gamma_T \nabla v_k\rangle_{L^2(\partial\Omega)}\gamma_t (\nabla v_k),\\
  U_\lambda A &= \sum\limits_{k=1}^\infty\frac{1}{\lambda^2-\mu_k^2}\langle A,\gamma_T \psi_k\rangle_{L^2(\partial\Omega)}\gamma_t(\psi_{k})
 \end{align*}
 for $A \in H^{-\frac{1}{2}}(\Div, \partial \Omega)$. Both sums converge in $H^{-\frac{1}{2}}(\Div, \partial \Omega)$.
Here $\beta_1 = \dim \mathcal{H}^1_{abs}(\Omega)$ is the first Betti number of the domain.
We have $T^2=0$ and $T U_\lambda + U_\lambda T = \id - \lambda^2 U_\lambda$.
\end{theorem}
\begin{proof} We start with an interior solution $E \in H(\curl,\Omega)$ of the Maxwell system and assume $A = \gamma_t(E) \in H^{-\frac{1}{2}}(\Div, \partial \Omega)$.
First note that $E$ satisfies $\div E=0$ but it is not in general in $\mathrm{ker}(\div_0)$ because it may not satisfy the correct boundary conditions. We have that 
\begin{align}
L^2(\Omega,\mathbb{C}^3)=\mathcal{H}^1_\abs(\Omega)\oplus \{ \psi_j \,|\, \mu_j >0 \} \oplus \{\nabla v_k| \lambda_{N,k}\},
\end{align}
where $v_k$ is an orthonormal basis of Neumann eigenfunctions on $\Omega$. Define
\begin{align}
\tilde{\psi}_k=\frac{1}{\lambda_{N,k}}\grad v_k.
\end{align}
Now we can write 
\begin{align}
E=\sum\langle E,\psi_k\rangle \psi_k+\sum\langle E,\tilde{\psi}_k\rangle \tilde{\psi}_k
\end{align}
which we need to show converges in $H(\curl,\Omega)$.
We have that 
\begin{align}
&\langle E,\psi_k\rangle_{L^2(\Omega)}=\frac{1}{\lambda^2-\mu_k^2}\left(\langle -\Delta E,\psi_k\rangle_{L^2(\partial\Omega)}-\langle E,-\Delta\psi_k\rangle_{L^2(\partial\Omega)}\right)=\\& \nonumber
\frac{1}{\lambda^2-\mu_k^2}\left(\langle \gamma_t \curl E,\gamma_T\psi_k\rangle_{L^2(\partial\Omega)}+\langle \gamma_t E,\gamma_T \curl\psi_k\rangle_{L^2(\partial\Omega)}\right)=\\& \frac{1}{\lambda^2-\mu_k^2}\langle \gamma_t \curl E,\gamma_T\psi_k\rangle_{L^2(\partial\Omega)}=\nonumber
\frac{\rmi \lambda}{\lambda^2-\mu_k^2}\langle \gamma_t H,\gamma_T\psi_k\rangle_{L^2(\partial\Omega)}=\\& \nonumber
\frac{\rmi \lambda}{\lambda^2-\mu_k^2}\langle \Lambda_{\lambda}^+A,\gamma_T \psi_k\rangle_{L^2(\partial\Omega)}
\end{align} 
where we have used Stokes theorem \eqref{stokes} as well as Maxwell system properties in \ref{system} repeatedly. Since $E \in L^2(\Omega,\C^3)$ the sum $\sum\limits\langle E,\psi_k\rangle \psi_k$ converges in $L^2(\Omega,\C^3)$. Let $\phi_k$ denote an orthonormal basis of eigenfunctions of $\Delta_{\rel}$. We now note that 
\begin{align*}
&\sum\limits_{\lambda_k\neq 0}\langle E,\psi_k\rangle \curl\psi_k=
\sum\limits_{\mu_k\neq 0}\langle E,\psi_k\rangle \curl\frac{1}{\mu_k}\curl\phi_k=
\sum\limits_{\lambda_k\neq 0}\langle E,\psi_k\rangle \mu_k\phi_k
\end{align*}
converges in $L^2(\Omega,\C^3)$ whenever $(\langle E,\psi_k\rangle_{L^2(\Omega)}\mu_k)_k \in \ell^2$. The latter is true because
\begin{align}
\mu_k\langle E,\psi_k\rangle_{L^2(\Omega)}=\langle E,\curl\phi_k\rangle_{L^2(\Omega)}=\langle \curl E,\phi_k\rangle_{L^2(\Omega)}\in \ell^2
\end{align} 
where we have used the fact $\curl E\in L^2(\Omega,\C^3)$.  Therefore 
\begin{align}
\sum\limits_{k=1}^{\infty}\langle E,\psi_k\rangle\psi_k
\end{align}
converges in $H(\curl,\Omega)$. For the second term, now we have
\begin{align}
\langle E,\tilde{\psi}_k\rangle_{L^2(\Omega)} =&  \nonumber \frac{\rmi}{\lambda}\langle \curl H,\frac{1}{\lambda_{N,k}}\grad v_k\rangle_{L^2(\Omega)}\\=&\frac{\rmi}{\lambda_{N,k}\lambda}\langle \curl H,\grad v_k\rangle_{L^2(\Omega)}=\frac{\rmi}{\lambda_{N,k}\lambda}\langle \gamma_t H,\gamma_T \grad v_k\rangle_{L^2(\partial\Omega)}
\end{align}
this also gives 
\begin{align}
\sum\limits_{\lambda_k\neq 0} \langle E,\tilde{\psi}_k\rangle \tilde{\psi}_k
\end{align} 
converges in $H(\curl,\Omega)$ as $\lambda_{N,k}^{-2}$ is summable. 
Therefore we have that 
\begin{align}
E=\sum\limits_{k=0}^{\infty}\frac{\rmi \lambda}{\lambda^2-\mu_k^2}\langle \gamma_t H,\gamma_T\psi_k\rangle_{L^2(\partial \Omega)}\psi_k+\sum\limits_{\lambda_{N,k}\neq 0}\frac{\rmi}{\lambda_k^2\lambda}\langle \gamma_t H,\gamma_T \grad v_k\rangle_{L^2(\partial \Omega)}\grad v_k
\end{align}
and this representation converges in $H(\curl,\Omega)$. Because of this, we have convergence in $H^{-\frac{1}{2}}(\Div,\partial\Omega)$ of
\begin{align}
A= \nu \times E|_{\partial\Omega} = & \sum\limits_{\mu_k\geq 0}\frac{\rmi \lambda}{\lambda^2-\mu_k^2}\langle \gamma_t H,\gamma_T \psi_k\rangle_{L^2(\partial\Omega)} \nonumber\gamma_t (\psi_k) \\ &+\sum\limits_{\lambda_{N,k}\neq 0}\frac{\rmi}{\lambda_{N,k}^2\lambda}\langle \gamma_t H,\gamma_T\grad v_k\rangle_{L^2(\partial \Omega)}\gamma_t(\grad v_k).
\end{align}
Then using the fact that $(\gamma_t(H))=\Lambda^+_{\lambda}(\gamma_t(E))= \Lambda^+_{\lambda}(A)$ and remarking that $(\Lambda^+)^2=-\id$, we obtain the desired result. 
Expanding the formula $(\rmi \Lambda^+)^2=\id$ also gives the claimed identities.
\end{proof} 

We now aim to show operator bounds on the electric dipole map in order to find bounds on the large $|\lambda|$ behavior of $\mathcal{L}^{-1}_{\lambda}$. Note that for $\lambda \in \mathfrak{D}_\epsilon$, we have the estimate
\begin{align*}
 \Im(\lambda) = |\Im(\lambda)| \leq |\lambda| \leq C_\epsilon \Im(\lambda),
\end{align*}
where $C_\epsilon := \sin(\epsilon)^{-1}$ is independent of $\lambda \in \mathfrak{D}_\epsilon$.\\
\begin{theorem}\label{N:estimate}
There exists a constant $C$ such that for all $\Im(\lambda)>0$ we have the estimate
\begin{align}
\|\Lambda_{\lambda}^{\pm}\|_{H^{-\frac{1}{2}}(\mathrm{Div},\partial \Omega) \mapsto  H^{-\frac{1}{2}}(\mathrm{Div},\partial \Omega) }\leq C \frac{1}{|\lambda|} \left(1+\frac{ |\lambda| (1+ |\lambda|^2)}{\Im{\lambda}} \right). 
\end{align}  
\end{theorem}
\begin{proof}
We first consider the case  $\Re(\lambda^2)<0$, i.e. $|\Im(\lambda)| > |\Re(\lambda)| $. 
We have the following integral identity
\begin{align}
\langle v, \mathrm{curl}u \rangle_{L^2(M)} - \langle \mathrm{curl}v,u \rangle_{L^2(M)}= \langle \gamma_t v, \gamma_T u \rangle_{L^2(\partial \Omega)}
\end{align} 
for $u,v\in H(\mathrm{curl},M)$. Applying this integral identity with $E$ and $H$ gives 
\begin{align}
\rmi \lambda \langle \gamma_t H, \gamma_T E \rangle_{L^2(\partial \Omega)}=\rmi \lambda\left(\langle H, \mathrm{curl} E \rangle_{L^2(M)} - \langle \mathrm{curl}H ,E \rangle_{L^2(M)}\right)
=   \langle \curl E , \curl E \rangle - \lambda^2 \langle E,E \rangle   \nonumber
\end{align} 
Taking the real part we obtain
$$
  |\Re(\lambda^2)| \langle E, E \rangle_{L^2(M)} \leq |\lambda| \cdot | \langle \gamma_t H, \gamma_T E \rangle_{L^2(\partial \Omega)}|.
$$
The antisymmetric bilinear form $\langle \nu \times u,v \rangle_{L^2(\partial \Omega)}$ extends continuously to $H^{-\frac{1}{2}}(\Div,\partial \Omega)$ (see for example Lemma 5.61 in \cite{kirsch}) we therefore have
\begin{align} \label{dualpairinginque}
\|E\|_{L^2(M)}^2\leq C_1 |\lambda|(-\Re(\lambda^2))^{-1} ||\gamma_t E||_{H^{-\frac{1}{2}}(\Div,\partial \Omega)}||\gamma_t(H)||_{H^{-\frac{1}{2}}(\mathrm{Div},\partial \Omega)}
\end{align}
Now we use the continuity of the tangential trace map and obtain
\begin{align*}
\|\gamma_t (\curl E)\|^2_{H^{-\frac{1}{2}}(\Div,\partial \Omega)} \leq C_2 \|(\curl E)\|^2_{H(\curl,M)} =
C_2\left( \| \curl E \|^2_{L^2(M)} +  |\lambda|^4 \| E \|^2_{L^2(M)} \right)\\
= C_2\left( |\lambda|^2\| E \|^2_{L^2(M)} +  |\lambda|^4 \| E \|^2_{L^2(M)} + \langle \gamma_t E, \gamma_T \curl E \rangle_{L^2(\partial \Omega)}\right) \\
\leq C_3\left( |\lambda|^2 (1+ |\lambda|^2) \|E \|^2_{L^2(M)} + \|\gamma_t E\|_{H^{-\frac{1}{2}}(\Div,\partial \Omega)} \|\gamma_t \curl E\|_{H^{-\frac{1}{2}}(\Div,\partial \Omega)} \right)
\end{align*}
Choosing $a= C_3 \|\gamma_t E\|_{H^{-\frac{1}{2}}(\Div,\partial \Omega)} $ and $b=\| \gamma_t \curl E\|_{H^{-\frac{1}{2}}(\Div,\partial \Omega)}$ and using the inequality
$|ab| \leq \frac{1}{2} (a^2 + b^2)$ one obtains
\begin{align*}
\|\gamma_t (\curl E)\|^2_{H^{-\frac{1}{2}}(\Div,\partial \Omega)}
\leq \left( 2 C_3 |\lambda|^2 (1+ |\lambda|^2) \|E \|^2_{L^2(M)} + C_3^2\|\gamma_t E\|^2_{H^{-\frac{1}{2}}(\Div,\partial \Omega)} \right)
\end{align*}
Using \eqref{dualpairinginque} this gives further
\begin{align*}
&\|\gamma_t (\curl E)\|^2_{H^{-\frac{1}{2}}(\Div,\partial \Omega)} \\
&\leq C_4 \left( \frac{ |\lambda|^2 (1+ |\lambda|^2)}{-\Re(\lambda^2)}  \|\gamma_t E\|_{H^{-\frac{1}{2}}(\Div,\partial \Omega)}\|\gamma_t \curl E\|_{H^{-\frac{1}{2}}(\mathrm{Div},\partial \Omega)} + \|\gamma_t E\|^2_{H^{-\frac{1}{2}}(\Div,\partial \Omega)} \right).
\end{align*}
The same trick as before with $a = C_4 \frac{ |\lambda|^2 (1+ |\lambda|^2)}{-\Re(\lambda^2)}  \|\gamma_t E\|_{H^{-\frac{1}{2}}(\Div,\partial \Omega)}$
and $b = \|\gamma_t \curl E\|_{H^{-\frac{1}{2}}(\mathrm{Div},\partial \Omega)}$ yields
\begin{align*}
\|\gamma_t (\curl E)\|^2_{H^{-\frac{1}{2}}(\Div,\partial \Omega)} 
\leq  \left( C_4^2 \frac{ |\lambda|^4 (1+ |\lambda|^2)^2}{(-\Re(\lambda^2))^2} + 2 C_4  \right)\|\gamma_t E\|^2_{H^{-\frac{1}{2}}(\Div,\partial \Omega)},
\end{align*}
which finally gives
\begin{align}
\|\gamma_t (\curl E)\|_{H^{-\frac{1}{2}}(\Div,\partial \Omega)} 
\leq   C \left(1+\frac{ |\lambda|^2 (1+ |\lambda|^2)}{-\Re(\lambda^2)} \right)  \|\gamma_t E\|_{H^{-\frac{1}{2}}(\Div,\partial \Omega)}.
\end{align}
Next consider the case $\Im(\lambda^2)<0$. The same proof with imaginary parts taken instead of real parts gives the estimate
\begin{align}
\|\gamma_t (\curl E)\|_{H^{-\frac{1}{2}}(\Div,\partial \Omega)} 
\leq   C \left(1+\frac{ |\lambda|^2 (1+ |\lambda|^2)}{-\Im(\lambda^2)} \right)  \|\gamma_t E\|_{H^{-\frac{1}{2}}(\Div,\partial \Omega)}.
\end{align}
These two estimates cover the upper half space and are combined into
\begin{align}
\|\gamma_t (\curl E)\|_{H^{-\frac{1}{2}}(\Div,\partial \Omega)} 
\leq   C \left(1+\frac{ |\lambda| (1+ |\lambda|^2)}{\Im(\lambda)} \right)  \|\gamma_t E\|_{H^{-\frac{1}{2}}(\Div,\partial \Omega)},
\end{align}
which holds in the upper half space except when $\Im(\lambda)< \Re(\lambda)$. The estimate holds in this region too as can be seen by replacing $\lambda$ by $-\overline{\lambda}$, which is a symmetry operation of the Maxwell system that preserves the radiation condition.
Hence, the estimate holds in the upper half space.
Since $\rmi \lambda H= \curl E$ this proves the claimed estimate.
The same proof works for the interior with $M$ replaced by $\Omega$.
\end{proof}

\begin{lemma} \label{Mexpand}
The operator $\left(\frac{1}{2}+\mathcal{M}_{\lambda}\right)^{-1}$ is meromorphic of finite type and we have near zero the expansion
\begin{align}
\left(\frac{1}{2}+\mathcal{M}_{\lambda}\right)^{-1}=\frac{P}{\lambda^2}+\frac{B}{\lambda}+Q_\lambda
\end{align}
where $P$ and $B$ are finite rank operators and $Q_\lambda$ is analytic near $\lambda=0$
taking values in the bounded operators on $H^{-\frac{1}{2}}(\Div,\partial \Omega)$.
We also have $$\mathrm{image}(P) \cup \mathrm{image}(B)  \subseteq \mathcal{B}_{\partial \Omega}, \quad P (\nu \times \nabla u)  = B (\nu \times \nabla u)=0$$ for all $u \in H^{\frac{1}{2}}(\partial \Omega)$.
\end{lemma} 
\begin{proof}
By proof of Lemma \ref{merolemma} we know that 
$\left(\frac{1}{2}+\mathcal{M}_{\lambda}\right)^{-1}$ is a meromorphic family of finite type.
The order of the singularity at zero is at most two since for $\lambda\in \mathfrak{D}_{\epsilon}, \lambda\neq 0$ we have
\begin{align}
\left(\frac{1}{2}+\mathcal{M}_{\lambda}\right)^{-1}=-\Lambda^+_\lambda \left(\Lambda^+_\lambda-\Lambda^-_\lambda\right)
\end{align}
and the bound in Theorem \ref{N:estimate} holds. Hence, $\left(\frac{1}{2}+\mathcal{M}_{\lambda}\right)^{-1}$ has the claimed form
$$
 \left(\frac{1}{2}+\mathcal{M}_{\lambda}\right)^{-1}=\frac{P}{\lambda^2}+\frac{B}{\lambda}+Q_\lambda,
$$
with $P,B$ of finite rank.

We must naturally have for these $\lambda$
\begin{align}\label{idexp}
\left(\frac{1}{2}+\mathcal{M}_{\lambda}\right)^{-1}\left(\frac{1}{2}+\mathcal{M}_{\lambda}\right)=\left(\frac{1}{2}+\mathcal{M}_{\lambda}\right)\left(\frac{1}{2}+\mathcal{M}_{\lambda}\right)^{-1}=\id.
\end{align}
Expanding $\left(\frac{1}{2}+\mathcal{M}_{\lambda}\right)$ around $\lambda=0$ we see that it has operator kernel:
\begin{align}
\frac{1}{2}+\frac{1}{4 \pi}\gamma_{t,x} \gamma^*_{T,y}\mathrm{curl}\left(\frac{1}{|x-y|}\right)+O(\lambda^2)
\end{align}
since the first order term in the expansion distributional kernel of the free Green's fucntion is constant, and therefore $\curl$-free. 
Hence, $$\left(\frac{1}{2}+\mathcal{M}_{\lambda}\right)= \frac{1}{2} + \mathcal{M}_{0} + O(\lambda^2)$$
near $\lambda=0$.
Inserting this into \eqref{idexp} and comparing coefficients one obtains
$$
  \left(\frac{1}{2} + \mathcal{M}_{0} \right) P =0, \quad \left(\frac{1}{2} + \mathcal{M}_{0}\right) B=0, \quad   P \left(\frac{1}{2} + \mathcal{M}_{0} \right)  =0,  \quad B \left(\frac{1}{2} + \mathcal{M}_{0}\right) =0.
$$
By Prop. \ref{dmitreazero} we therefore obtain $\mathrm{image}(P),\mathrm{image}(B)  \subseteq \mathcal{B}_{\partial \Omega}$ as claimed.
It remains to show that $$P (\nu \times \nabla u)  = B (\nu \times \nabla u)=0.$$ To see this it is sufficient to show that $\nu \times \nabla u$ is in the range of
 $\frac{1}{2} + \mathcal{M}_{0}$. To see this we use a classical result in potential layer theory, namely the invertibility of $(\frac{1}{2} + \mathcal{K}_0)$
 (see \cite{MR769382}). We then have by Equ. \eqref{intertwin} 
$$
 \nu \times \nabla u = \nu \times \nabla (\frac{1}{2} + \mathcal{K}_0) (\frac{1}{2} + \mathcal{K}_0)^{-1} u = (\frac{1}{2} + \mathcal{M}_0) (\nu \times \nabla (\frac{1}{2} + \mathcal{K}_0)^{-1} u).
$$
\end{proof}

\begin{lemma} \label{Maxwexpand}
 The non-zero poles of $(\frac{1}{2} + \calM_\lambda)^{-1}$ in the closed upper half-space are precisely the Maxwell eigenvalues of $\Omega$. Near a Maxwell eigenvalue $\mu=\mu_k$ we have the expansion
 \begin{align}
\left(\frac{1}{2}+\mathcal{M}_{\lambda}\right)^{-1}=\frac{P_\mu}{(\lambda-\mu)^2}+\frac{B_\mu}{\lambda-\mu}+Q_{\mu,\lambda},
\end{align}
 where $P_\mu$ and $B_\mu$ are finite rank operators with range in $\ker\left(\frac{1}{2}+\mathcal{M}_{\mu}\right)^{-1}$ and $Q_{\mu,\lambda}$ is holomorphic in $\lambda$ near $\mu$.
\end{lemma}
\begin{proof}
 The poles are precisely where $\left(\frac{1}{2}+\mathcal{M}_{\lambda}\right)$ is not injective. On the closed upper half space this means that the only poles are at zero and at the Maxwell eigenvalues, by Prop. \ref{dmitreazero} and Prop. \ref{menotzero}. The statement now follows immediately from the formula
 \begin{align}
\left(\frac{1}{2}+\mathcal{M}_{\lambda}\right)^{-1}=-\Lambda^+_\lambda \left(\Lambda^+_\lambda-\Lambda^-_\lambda\right)
\end{align}
the expansion of Theorem \ref{eigenrep} and the fact that $\Lambda^-$ is holomorphic near $\R \setminus \{0\}$ by Lemma \ref{merolemma}.
\end{proof}

\begin{theorem}\label{supercor}
For any $\epsilon>0$ we have there exists a constant $C>0$ such that
\begin{align*}
  \|\mathcal{L}_{\lambda}^{-1} \|_{H^{-\frac{1}{2}}(\mathrm{Div},\partial\Omega)\rightarrow H^{-\frac{1}{2}}(\Div,\partial\Omega)} &\leq  \frac{1+|\lambda|^2}{|\lambda|^2} C\left(1+|\lambda|^{2}\right),\\
  \| \Div\circ (\mathcal{L}_{\lambda}^{-1}) \circ(\nu \times \nabla)\|_{H^{-\frac{1}{2}}(\partial\Omega)\rightarrow H^{-\frac{1}{2}}(\partial\Omega)} &\leq  C\left(1+|\lambda|^{2}\right)
\end{align*}
for all $\lambda$ in the sector $\mathfrak{D}_\epsilon$.
\end{theorem}
\begin{proof}
 We use the identity, derived from \eqref{twenty},
 \begin{align}
 \mathcal{L}_{\lambda}^{-1}=-\frac{\rmi \,(\Lambda^+_{\lambda}-\Lambda^-_{\lambda})}{\lambda}
 \end{align}
 to reduce the analysis to that of $\Lambda^{\pm}_{\lambda}$.  The bounds on the operator norm on the space $H^{-\frac{1}{2}}(\mathrm{Div},\partial \Omega)$ then follow immediately from Theorem \ref{N:estimate}.
 By \eqref{voltcur} we have the identity 
 $$
  \mathcal{L}_{\lambda}^{-1} = \frac{1}{\rmi\, \lambda} \Lambda^+_\lambda (\frac{1}{2} + \calM_\lambda)^{-1}.
 $$
 Using Theorem \ref{eigenrep} we obtain
 \begin{align} \nonumber
  \mathcal{L}_{\lambda}^{-1} &= - \left( \frac{1}{\lambda^2} T + U_\lambda \right) \left( \frac{1}{\lambda^2} P +  \frac{1}{\lambda} B + Q_\lambda \right) \\
  &= - \frac{1}{\lambda^2} \left( T Q_\lambda + U_\lambda P \right) - \frac{1}{\lambda} U_\lambda B + U_\lambda Q_\lambda.  \label{Lexpand}
 \end{align}
We have used that $T P = T B = 0$ which follows from Lemma \ref{Mexpand} and Theorem \ref{eigenrep}. 
Since $\Div \circ T =0$, $P \circ ( \nu \times \nabla) =0$, and $B \circ (\nu \times \nabla) =0$ we then obtain
$$
 \Div\circ (\mathcal{L}_{\lambda}^{-1}) \circ(\nu \times \nabla) =\Div\circ U_\lambda Q_\lambda \circ(\nu \times \nabla),
$$
which is regular at zero.
 \end{proof}

\section{Resolvent formulae and estimates} \label{resolventformulae}

\begin{proposition} \label{propeightone}
Assume that $\Im(\lambda)>0$.
For $f \in C^\infty_0(\R^3,\C^3)$ we have the following formulae for the difference of resolvents:
\begin{align} \label{eqn:ResolvDiffQ}
 \left(( -\Delta_\rel- \lambda^2)^{-1} - (-\Delta_\free- \lambda^2)^{-1}\right)\curl\curl f =  -\tilde{\mathcal{L}}_\lambda(\mathcal{L}_{\lambda})^{-1} (\nu \times) \tilde{\mathcal{L}}^\trans_{\lambda} f,\\
 \left(( -\Delta_\rel- \lambda^2)^{-1} - (-\Delta_\free- \lambda^2)^{-1}\right)\curl f =  -\tilde{\mathcal{L}}_\lambda(\mathcal{L}_{\lambda})^{-1} (\nu \times) \tilde{\mathcal{M}}^\trans_{\lambda} f,\\
 \curl \left(( -\Delta_\rel- \lambda^2)^{-1} - (-\Delta_\free- \lambda^2)^{-1}\right)\curl f =  -\lambda^2 \tilde{\mathcal{M}}_\lambda(\mathcal{L}_{\lambda})^{-1} (\nu \times) \tilde{\mathcal{M}}^\trans_{\lambda} f. \label{eqn:ResolvDiffQabs}
\end{align}
Here $ \tilde{\mathcal{L}}^\trans_{\lambda}$ is the transpose operator to 
$ \tilde{\mathcal{L}}_\lambda$ obtained from the real $L^2$-inner product, i.e. $ \tilde{\mathcal{L}}^\trans_{\lambda} f = \overline{ \tilde{\mathcal{L}}^*_{\lambda} \overline f }$. Similarly, $\tilde{\mathcal{M}}^\trans_{\lambda}$ is the transpose of $\tilde{\mathcal{M}}_{\lambda}$.
\end{proposition}
\begin{proof}
 We begin with the first formula.
  We know that $\tilde{\mathcal{L}}_\lambda$ maps to functions satisfying the Helmholtz equation
 $\left(-\Delta- \lambda^2 \right) v =0$. Therefore we only need to show that, given $f \in C^\infty_0(\R^3,\C^3)$, the function
 $$
   u = (-\Delta_\free- \lambda^2)^{-1}\curlcurl f  -\tilde{\mathcal{L}}_\lambda(\mathcal{L}_{\lambda})^{-1} (\nu \times) \tilde{\mathcal{L}}^\trans_{\lambda} f
 $$
 satisfies relative boundary conditions. Since clearly $\div\, u=0$ we only need to check that $\gamma_t u=0$.
One computes
\begin{align*}
 \gamma_t u &= \gamma_t \curlcurl (-\Delta_\free- \lambda^2)^{-1} f - \mathcal{L}_\lambda(\mathcal{L}_{\lambda})^{-1} (\nu \times) \tilde{\mathcal{L}}^\trans_{\lambda} f\\
 &= \gamma_t \curlcurl (-\Delta_\free- \lambda^2)^{-1} f - (\nu \times) \gamma_T   \curlcurl (-\Delta_\free- \lambda^2)^{-1} f =0,
\end{align*}
which gives the result.\\
Next consider the second formula. We again only need to check that $\gamma_{t,\pm}(u)=0$ where
$$
 u = (-\Delta_\free- \lambda^2)^{-1}\curl f  -\tilde{\mathcal{L}}_\lambda(\mathcal{L}_{\lambda})^{-1} (\nu \times) \tilde{\mathcal{M}}^\trans_{\lambda} f.
$$
The third formula follows from the second by applying the $\curl$-operator from the left and using
$\curl \,\curl \,\curl\, \tilde{\calS}_\lambda= \lambda^2 \curl\,\tilde{\calS}_\lambda$.
\end{proof}

This can be used to show the following.

\begin{theorem} \label{Thm:DifferenceProperties} 
Let $\epsilon>0$ and also suppose that $\Omega_0$ is a smooth open set in $\R^3$ whose complement contains $\overline{\Omega}$. Let $\delta=\mathrm{dist}(\partial\Omega,\Omega_0)$. If $p$ is the projection onto $L^2(\Omega_0;\mathbb{C}^3)$ in $L^2(\R^3;\mathbb{C}^3)$ then the operators
 \begin{align*}
  p(-\Delta_\rel - \lambda^2)^{-1}\curl\curl\, p - p(-\Delta_\free - \lambda^2)^{-1}\curl\curl\, p,\\
  p(-\Delta_\abs - \lambda^2)^{-1}\curl\curl\, p - p(-\Delta_\free - \lambda^2)^{-1}\curl\curl\, p,
 \end{align*}
are trace class for all $\lambda \in \mathfrak{D}_\epsilon$ as operators on $L^2(\R^3;\mathbb{C}^3)$. Moreover for any $\delta'\in (0,\delta)$, their trace norms satisfy the bounds
 \begin{align} \label{eqn:TraceNormBound}
  \|p(-\Delta_\rel - \lambda^2)^{-1}\curl\curl \,p - p(-\Delta_\free - \lambda^2)^{-1}\curl\curl \,p   \|_{1} \leq C_{\delta',\epsilon}e^{-\delta'\Im(\lambda)},\\
   \|p(-\Delta_\abs - \lambda^2)^{-1}\curl\curl \,p - p(-\Delta_\free - \lambda^2)^{-1}\curl\curl \,p   \|_{1} \leq C_{\delta',\epsilon}e^{-\delta'\Im(\lambda)},
 \end{align}
 for all $\lambda \in \mathfrak{D}_\epsilon$. Moreover, both operators have integral kernels
 $\kappa_{\rel,\lambda},\kappa_{\abs,\lambda}$ that are smooth on $\Omega_0 \times \Omega_0$ for all $\lambda \in  \mathfrak{D}_\epsilon$. There exists $C_{\Omega_0,\epsilon}>0$, depending on $\Omega_0$ and $\epsilon$ such that
  \begin{align} \label{eqn:DiagKernelEst}
  \| k_{\rel,\lambda}(x,x) \| + \| k_{\abs,\lambda}(x,x) \|  \leq \left( C_{\Omega_0,\epsilon}\frac{e^{-\mathrm{dist}(x,\del\Omega) \Im\lambda} }{(\mathrm{dist}(x,\del\Omega))^{4}}\right).
 \end{align}
 \end{theorem} 
\begin{proof}
Given $\delta' \in (0,\delta)$ we choose a compactly supported smooth cut-off function $\chi$ which vanishes in $\Omega_0$ such that the support of $\varphi=1-\chi$ has distance at least $\delta'$ from $\Omega$. Then, since $\varphi p = p$ it is sufficient to show the estimates with $p$ replaced by $\varphi$.
From \eqref{eqn:ResolvDiffQ}, we have
\begin{align} \nonumber
&\varphi(-\Delta_\rel - \lambda^2)^{-1}(\curl\curl) \varphi -\varphi(-\Delta_\free - \lambda^2)^{-1}(\curl\curl) \varphi =\\& -\varphi\tilde{\mathcal{L}}_\lambda \mathcal{L}_\lambda^{-1}(\nu \times) \tilde{\mathcal{L}}^\trans_{\lambda}
 \varphi = - (\varphi \tilde{\mathcal{L}}_\lambda) \mathcal{L}_\lambda^{-1}  (\nu \times) (\varphi\tilde{\mathcal{L}}_{\lambda})^\trans. \label{rhsres}
\end{align}
The operator $\varphi \mathcal{\tilde{L}}_{\lambda}$ is Hilbert-Schmidt by Proposition \ref{Prop:EstimateS}. Since $\mathcal{L}_\lambda^{-1}$ is bounded by Corollary \eqref{supercor} on the correct domains, this factorises the right hand side of \eqref{rhsres}  into a product of the two Hilbert-Schmidt operators
$(\varphi \tilde{\mathcal{L}}_\lambda), (\varphi\tilde{\mathcal{L}}_{\lambda})^\trans$ and a bounded operator $\mathcal{L}_\lambda^{-1}  (\nu \times)$.
This shows it is trace-class (see for example \cite{shubin}, (A.3.4) and (A.3.2)). We need to show the bound for the trace-norm. We now employ the more explicit
description of $\varphi\tilde{\mathcal{L}}_{\lambda} =\varphi (\nabla  \tilde{\mathcal{S}}_\lambda \Div  + \lambda^2 \tilde{\mathcal{S}}_\lambda)$.
This gives
\begin{align*}
  &(\varphi \tilde{\mathcal{L}}_\lambda) \mathcal{L}_\lambda^{-1}  (\nu \times) (\varphi\tilde{\mathcal{L}}_{\lambda})^\trans = 
  \left( \varphi \nabla  \tilde{\mathcal{S}}_\lambda \Div  + \lambda^2 \varphi\tilde{\mathcal{S}}_\lambda \right) \mathcal{L}_\lambda^{-1}  \left((\nu \times) \nabla  \tilde{\mathcal{S}}^\trans_\lambda \div \varphi  + \lambda^2 (\nu \times)(\varphi \tilde{\mathcal{S}}_\lambda)^\trans \right) \\&=
  \left( \varphi \nabla  \tilde{\mathcal{S}}_\lambda \Div  + \lambda^2 \varphi\tilde{\mathcal{S}}_\lambda \right) \mathcal{L}_\lambda^{-1}  \left((\nu \times) \nabla  \tilde{\mathcal{S}}^\trans_\lambda \div \varphi  + \lambda^2 (\nu \times)(\varphi \tilde{\mathcal{S}}_\lambda)^\trans \right) 
  \\&= \varphi \nabla  \tilde{\mathcal{S}}_\lambda \Div  \mathcal{L}_\lambda^{-1} (\nu \times) \nabla  \tilde{\mathcal{S}}^\trans_\lambda \div \varphi + \lambda^4  \varphi\tilde{\mathcal{S}}_\lambda \mathcal{L}_\lambda^{-1}  (\nu \times)(\varphi \tilde{\mathcal{S}}_\lambda)^\trans +  \lambda^2 \varphi \nabla  \tilde{\mathcal{S}}_\lambda \Div \mathcal{L}_\lambda^{-1}  (\nu \times)(\varphi \tilde{\mathcal{S}}_\lambda)^\trans \\&+ \lambda^2 \varphi\tilde{\mathcal{S}}_\lambda \mathcal{L}_\lambda^{-1} (\nu \times )\nabla  \tilde{\mathcal{S}}^\trans_\lambda \div\varphi = \mathrm{(I)} + \mathrm{(II)} + \mathrm{(III)} + \mathrm{(IV)}.
\end{align*}
We will show that the estimate holds for the individual terms.
The trace-norm of $\mathrm{(I)}$ is bounded by $\| \varphi \nabla  \tilde{\mathcal{S}}_\lambda  \|^2_{\mathrm{HS}} \cdot \|  \Div  \mathcal{L}_\lambda^{-1} (\nu \times) \nabla \|$, using the fact that the Hilbert-Schmidt norm is invariant under transposition. This is bounded by $C e^{-\Im(\lambda) \delta'}$ in the sector by Prop. \ref{supercor} and the estimate \eqref{eqn:nablaHSEstimateS} of Prop. \ref{Prop:EstimateS}.

The trace-norm of term $\mathrm{(II)}$ is bounded by
$|\lambda|^4  \| \varphi\tilde{\mathcal{S}}_\lambda\|^2_{\mathrm{HS}}  \| \mathcal{L}_\lambda^{-1} \|$. This is again bounded by $C e^{-\Im(\lambda) \delta'}$ by  Prop. \ref{supercor}  and \eqref{eqn:HSEstimateS} of Prop. \ref{Prop:EstimateS}.
Expression $\mathrm{(III)}$ is the transpose of $\mathrm{(IV)}$ as one computes easily from Lemma \ref{transposelemma}. It is therefore
sufficient to bound the trace-norm of $\mathrm{(IV)}$. We have that 
\begin{align}
 \mathrm{(IV)} &= \lambda^2 \varphi  \tilde{\mathcal{S}}_\lambda \left(- \frac{1}{\lambda^2} \left( T Q_\lambda + U_\lambda P \right) - \frac{1}{\lambda} U_\lambda B + U_\lambda Q_\lambda\right) (\nu \times) \nabla (\varphi \tilde{\mathcal{S}}_\lambda)^\trans \\&= \lambda^2 (\varphi \tilde{\mathcal{S}}_\lambda) 
  \left( \frac{1}{\lambda^2}T Q_\lambda  + U_\lambda Q_\lambda\right)(\nu \times \nabla) 
 (\varphi \tilde{\mathcal{S}}_\lambda)^\trans,
\end{align}
where we have used Lemma \ref{Mexpand}, the expansion \eqref{Lexpand} and the fact that $P(\nu \times \nabla)=0$ and $B(\nu \times \nabla)=0$. 
The range of $T$ consists of distributions in $H^{-\frac{1}{2}}_0(\partial\Omega,\C^3) \cap H^{-\frac{1}{2}}(\Div, \partial \Omega) $. To see this, note that the range of $T$ consists, by Theorem \ref{eigenrep}, of limits in $H^{-\frac{1}{2}}(\Div, \partial \Omega)$ of boundary values of $\curl$-free vector fields.
Applying the integration by parts formula \eqref{stokes}
with $\phi \in \mathrm{rg}(T)$ and $E$ a constant unit vector field, noting that $\curl \phi = \curl E =0$, one obtains that
$\langle \gamma_t \phi, \gamma E \rangle_{L^2(\partial \Omega,\C^3)}=\langle \gamma_t \phi, \gamma_T E \rangle_{L^2(\partial \Omega,\C^3)}=0$ as claimed.
It follows that the trace-norm of $\mathrm{(III)}$ and $\mathrm{(IV)}$ are bounded by
$C e^{-\Im(\lambda) \delta'}$  by Prop. \ref{supercor}, and by the estimates \eqref{eqn:nablaHSEstimateS}, \eqref{eqn:HSEstimateS}.

Next we use  \eqref{eqn:ResolvDiffQabs} to obtain
\begin{align} \nonumber
&\varphi(-\Delta_\abs - \lambda^2)^{-1}(\curl\curl) \varphi -\varphi(-\Delta_\free - \lambda^2)^{-1}(\curl\curl) \varphi =\\& - \lambda^2\varphi\tilde{\mathcal{M}}_\lambda \mathcal{L}_\lambda^{-1}(\nu \times) \tilde{\mathcal{M}}^\trans_{\lambda}
 \varphi = - \lambda^2(\varphi \tilde{\mathcal{M}}_\lambda) \mathcal{L}_\lambda^{-1}  (\nu \times) (\varphi\tilde{\mathcal{M}}_{\lambda})^\trans. \label{rhsresabs}
\end{align}
The operators $\varphi\tilde{\mathcal{M}}_\lambda, (\varphi\tilde{\mathcal{M}}_{\lambda})^\trans$ are Hilbert-Schmidt and their Hilbert-Schmidt norms are bounded by $e^{-\delta' \Im(\lambda)}$ by Prop. \ref{Prop:EstimateS}, Equ. \eqref{eqn:HSMEstimateS}. This gives the claimed estimate for the trace-norm since the operator $\lambda^2 \mathcal{L}_\lambda^{-1}$ is polynomially bounded in any sector by Theorem \ref{supercor}.

It remains to show the estimate on the diagonal of the integral kernel. This is done the same way using pointwise estimate 
$$
 \| \partial^\alpha_x \tilde \calS_\lambda(x,\cdot) \|_{H^{-1}(\partial \Omega)} \leq C \frac{1}{(\mathrm{dist}(x,\partial \Omega))^{1+|\alpha|}} e^{-\frac{1}{2}\Im(\lambda)\mathrm{dist}(x,\partial \Omega)}
$$
which is easily obtained directly from the integral kernel, noting that differentiation in the $x$ or $y$-variable gives a linear combination of terms that are bounded
by 
$$
 \frac{\lambda^k (\mathrm{dist}(x,\partial \Omega))^k}{(\mathrm{dist}(x,\partial \Omega))^{1+|\alpha|}} e^{-\Im(\lambda)\mathrm{dist}(x,\partial \Omega)} \leq C_{k,\epsilon} \frac{1}{(\mathrm{dist}(x,\partial \Omega))^{1+|\alpha|}} e^{-\frac{1}{2}\Im(\lambda)\mathrm{dist}(x,\partial \Omega)}
$$
with $0 \leq k \leq \alpha$. One now applies this estimate to each of the four terms $\mathrm{(I)}, \mathrm{(II)}, \mathrm{(III)}, \mathrm{(IV)}$ and observes that every factor of $\lambda$ can be absorbed using the bound
$$
  |\lambda|^k e^{-\Im(\lambda)\mathrm{dist}(x,\partial \Omega)} =C_{k,\epsilon} \frac{1}{\mathrm{dist}(x,\partial \Omega)^k}e^{-\frac{1}{2}\Im(\lambda)\mathrm{dist}(x,\partial \Omega)}. 
$$
This gives the first claimed estimate. The second estimate follows the same way, since the above implies
$$
 \| \tilde \calM_\lambda(x,\cdot) \|_{H^{-1}(\partial \Omega)} \leq C \frac{1}{(\mathrm{dist}(x,\partial \Omega))^{2}} e^{-\frac{1}{2}\Im(\lambda)\mathrm{dist}(x,\partial \Omega)}.
$$
\end{proof}

\section{The function $\Xi$}\label{multiL}
Recall that the boundary $\partial\Omega$ consists of $N$ connected components $\partial\Omega_j$. To keep the discussion meaningful we will assume throughout this section that $N \geq 2$. This gives a natural decomposition
$$
 H^{-\frac{1}{2}}(\Div,\partial \Omega) = \bigoplus_{j=1}^N H^{-\frac{1}{2}}(\Div,\partial \Omega_j).
$$
Let $q_j$ be the orthogonal projection $H^{-\frac{1}{2}}(\Div,\partial \Omega) \to H^{-\frac{1}{2}}(\Div,\partial \Omega_j)$, 
and $\mathcal{L}_{j,\lambda}=q_j\mathcal{L}_{\lambda}q_j$. We then can write 
\begin{align}
\mathcal{L}_{\lambda}=\sum\limits_{j=1}^N\mathcal{L}_{j,\lambda}+\sum\limits_{j\neq k}q_j\mathcal{L}_{\lambda}q_k=\mathcal{L}_{D,\lambda}+\mathcal{T}_{\lambda}.
\end{align}
We remark that $\mathcal{L}_{j,\lambda}$ which is regarded as a map from $H^{-\frac{1}{2}}(\Div,\partial\Omega)\to H^{-\frac{1}{2}}(\Div, \partial\Omega)$ is independent of the other components. The sum $\mathcal{L}_{D,\lambda}$ describes the diagonal part of the operator $\mathcal{L}$ with respect to the decomposition above. 

We have a similar decomposition for the operator $$\calM_\lambda = \calM_{D,\lambda} + \mathcal{J}_\lambda.$$

We set
\begin{align}
\delta= \min_{j\neq k} \mathrm{dist}(\partial\Omega_j,\partial\Omega_k)>0.
\end{align}
Then we have the following proposition:
\begin{proposition}\label{Prop:CalT}
The families $\mathcal{T}_{\lambda},\mathcal{J}_{\lambda}: H^{-\frac{1}{2}}(\Div,\partial \Omega) \to H^{-\frac{1}{2}}(\Div,\partial \Omega)$ are holomorphic families of trace-class operators in the complex plane.
For any $\epsilon>0$ and any $\delta' \in (0,\delta)$ the following estimates for their trace-norms $\| \cdot \|_1$ hold:
\begin{align}
&\|\mathcal{T}_{\lambda}\|_1 \leq C_{\delta',\epsilon} e^{-\delta'\Im\lambda},\quad \|\mathcal{J}_{\lambda}\|_1 \leq C_{\delta',\epsilon} e^{-\delta'\Im\lambda} \\&
\|\frac{\der}{\der\lambda}\mathcal{T}_{\lambda} \|_1 \leq C_{\delta',\epsilon} e^{-\delta'\Im\lambda}, \quad \|\frac{\der}{\der\lambda}\mathcal{J}_{\lambda} \|_1 \leq C_{\delta',\epsilon} e^{-\delta'\Im\lambda}
\end{align} 
for all $\lambda$ the sector $\mathfrak{D}_{\epsilon}$. We also have
\begin{align}
&\|\mathcal{T}_{\lambda}|_{H^{-\frac{1}{2}}(\Div0,\partial \Omega)} \|_1 \leq  C_{\delta',\epsilon} |\lambda|^2 e^{-\delta'\Im\lambda}\\&
\|\frac{\der}{\der\lambda}\mathcal{T}_{\lambda}|_{H^{-\frac{1}{2}}(\Div0,\partial \Omega)} \|_1 \leq C_{\delta',\epsilon} |\lambda| e^{-\delta'\Im\lambda}.
\end{align} 
\end{proposition}
\begin{proof}
We will prove this estimate only for $\mathcal{T}_\lambda$ as the estimate for $\mathcal{J}_\lambda$ is proved in the same way.
It is sufficient to show this for the individual terms $q_j \mathcal{L}_\lambda q_k$ with $j \not= k$. We choose an open bounded neighborhood $U$ of $\partial \Omega_j$
and an open bounded neighborhood $V$ of $\partial \Omega_k$ such that $\mathrm{dist}(U,V)> \delta'$.
The first two estimates are implied by Lemma \ref{HilSobTrLemma} by observing that the operator is the composition
$$
 H^{-\frac{1}{2}}(\Div,\partial \Omega) \to H^{-1}(V) \to H^1(U) \to  H^{-\frac{1}{2}}(\Div,\partial \Omega),
$$
and the map $H^{-1}(V) \to H^1(U)$ has smooth integral kernel 
$$
 \chi(x,y) \curlcurl_x \frac{e^{\rmi \lambda |x-y|}}{4 \pi |x-y|}
$$
for a suitable cut-off function that is compactly supported in $U \times V$. The same argument applies to the $\lambda$-derivative.

To show the bounds on the restriction to ${H^{-\frac{1}{2}}(\Div0,\partial \Omega)}$ one uses that
we have $\mathcal{L}_\lambda = \gamma_t \nabla \mathcal{S}_\lambda \Div + \lambda^2 \mathcal{S}_\lambda$. To bound the trace-norm of
$ \lambda^2 q_j \mathcal{S}_\lambda q_k$ one uses exactly the same argument as above applied to the kernel
$$
 \lambda^2 \chi(x,y) \frac{e^{\rmi \lambda |x-y|}}{4 \pi |x-y|}
$$
and its $\lambda$-derivative.
\end{proof}

\begin{proposition}\label{trace1prime}
Fix $\epsilon>0$. Then $(\mathcal{L}_{\lambda}\mathcal{L}_{D,\lambda}^{-1}-\id) : H^{-\frac{1}{2}}(\Div,\partial \Omega) \to H^{-\frac{1}{2}}(\Div,\partial \Omega)$ is a 
meromorphic family of trace-class operators with no poles in the closed upper half-plane. In the sector we have for any 
$\delta' \in (0,\delta)$ the estimate 
\begin{align}
 \|\mathcal{L}_{D,\lambda}^{-1} \mathcal{L}_{\lambda}-\id \|_{1} =\|\mathcal{L}_{\lambda}\mathcal{L}_{D,\lambda}^{-1}-\id \|_{1} \leq C_{\delta',\epsilon} e^{-\delta'\Im\lambda}.
\end{align}
\end{proposition}
\begin{proof}
We use \eqref{voltcur} we obtain
\begin{align*}
 \mathcal{L}_{\lambda}\mathcal{L}_{D,\lambda}^{-1} -\id &= \left( \frac{1}{2}+\mathcal{M}_{\lambda}\right) \left( \frac{1}{2}+\mathcal{M}_{D,\lambda}\right)^{-1} - \id,
  \end{align*}
 bearing in mind that $\Lambda^+_\lambda = \Lambda^+_{D,\lambda}$.
With $$\left( \frac{1}{2}+\mathcal{M}_{D,\lambda}\right)^{-1} = \frac{1}{\lambda^2} P_D + \frac{1}{\lambda} B_D + Q_\lambda$$ we remark that
$$\left( \frac{1}{2}+\mathcal{M}_{D,0}\right) P_D = \left( \frac{1}{2}+\mathcal{M}_{D,0}\right) B_D =0$$ but then also
$$\left( \frac{1}{2}+\mathcal{M}_{0}\right) P_D = \left( \frac{1}{2}+\mathcal{M}_{0}\right) B_D =0$$ because 
according to Prop. \ref{dmitreazero} we know that the kernels of
$\left( \frac{1}{2}+\mathcal{M}_{0}\right)$ and $\left( \frac{1}{2}+\mathcal{M}_{D,0}\right)$ coincide.
We have used here, as in the proof of Lemma \ref{Mexpand}, that the first order terms in the expansion of $\mathcal{M}_\lambda$ at vanish at $\lambda=0$, i.e. $(\frac{\der}{\der \lambda} \mathcal{M}_\lambda) |_{\lambda=0} = (\frac{\der}{\der \lambda} \mathcal{M}_{D,\lambda}) |_{\lambda=0} =0$.
Using the abbreviation $\mathcal{J}_{\lambda} = \mathcal{M}_{\lambda} - \mathcal{M}_{D,\lambda}$ this implies
$\mathcal{J}_{0} P_D = \mathcal{J}_{0} B_D=0$. Moreover, $\mathcal{J}_{\lambda}$ is trace-class. This shows that
$$
 \left( \frac{1}{2}+\mathcal{M}_{\lambda}\right) \left( \frac{1}{2}+\mathcal{M}_{D,\lambda}\right)^{-1} - \id
$$
is a meromorphic family of trace-class operators and zero is not a pole. Interior Maxwell eigenvalues are not poles by the same argument, since the kernel of 
$\left( \frac{1}{2}+\mathcal{M}_{\mu}\right)$ coincides with the kernel of $\left( \frac{1}{2}+\mathcal{M}_{D,\mu}\right)$ and by the expansion of Lemma \ref{Maxwexpand}.

Moreover, $\left( \frac{1}{2}+\mathcal{M}_{D,\lambda}\right)$ is invertible for all the other points in the closed upper half-space, and hence there are no poles there. To show the estimate in the sector we note that
\begin{align}
 \mathcal{L}_{\lambda} \mathcal{L}_{D,\lambda}^{-1} -\id=\mathcal{T}_{\lambda}\mathcal{L}_{D,\lambda}^{-1}.
\end{align}
Then the bound for large $|\lambda|$ is a result of Corollary \ref{supercor} and Proposition \ref{Prop:CalT}. 
\end{proof}

\begin{proposition} \label{proporprpoasd}
The Fredholm determinant $\mathrm{det}(\mathcal{L}_{\lambda}\mathcal{L}_{D,\lambda}^{-1})$  in the space $H^{-\frac{1}{2}}(\Div,\partial \Omega)$ is well-defined and holomorphic in a neighborhood of the closed upper half space. For any $\epsilon>0$ and $\delta' \in (0,\delta)$ we have the bound
\begin{align}
 |\mathrm{det}(\mathcal{L}_{\lambda}\mathcal{L}_{D,\lambda}^{-1}) -1 |\leq C_{\delta',\epsilon}e^{-\delta' \Im(\lambda)} 
\end{align} 
for all $\lambda$ in the sector $\mathfrak{D}_{\epsilon}$. Moreover, $\mathrm{det}(\mathcal{L}_{\lambda}\mathcal{L}_{D,\lambda}^{-1})$ is non-zero in the closed upper half space.
\end{proposition}
\begin{proof}
The trace of $\left(\mathcal{L}_{\lambda}\mathcal{L}_{D,\lambda}^{-1}-\id\right)$ is bounded by Proposition \ref{trace1prime}. 
Using the bound 
$$
 | \det(1 + A) -1 | \leq \| A\|_1 e^{1 + \|A\|_1}
$$
for the Fredholm determinant (see for example \cite{MR482328}*{Equ. (3.7)}) one obtains 
\begin{align}
|\Xi(\lambda)|\leq |\log\det( \mathcal{L}_{\lambda}\mathcal{L}_{D,\lambda}^{-1})|\leq C_{\delta',\epsilon}e^{-\delta'\Im\lambda}.
\end{align} 
By analyticity of $\left(\mathcal{L}_{\lambda}\mathcal{L}_{D,\lambda}^{-1}-\id\right) = \mathcal{J}_\lambda (\frac{1}{2} + \mathcal{M}_{D,\lambda})^{-1}$ as a family of trace-class operators in the upper half space and near zero the determinant also depends analytically on $\lambda$ (e.g. \cite{MR482328}*{Theorem 3.3}).
By invertibility of the operator in the closed upper half space the determinant never vanishes (\cite{MR482328}*{Theorem 3.9}) and therefore $\log\det$ is analytic in union of the upper half space and a neighborhood of zero.
\end{proof}

Since the determinant does not vanish near the closed upper half space we can choose a simply connected open neighborhood $\mathcal{U}$ of the closed upper half space and it then defines a holomorphic function $\mathcal{U} \to \C \setminus \{0\}$ which we can lift to a holomorphic function on the logarithmic cover of the complex plane, where we choose the branch cut to be the negative real line $(-\infty,0)$.
Composition with $\log$ is then well-defined and we write $\log \det(\mathcal{L}_{\lambda}\mathcal{L}_{D,\lambda}^{-1})$ to mean this composition.
This means that this function and the branch of the logarithm is fixed by requiring this to be a holomorphic function that decays exponentially fast along the positive imaginary axis.

\begin{definition}
 The function $\Xi$ is defined in a sufficiently small simply connected open neighborhood of the closed upper half space by
 $$
 \Xi(\lambda) = \log \det(\mathcal{L}_{\lambda}\mathcal{L}_{D,\lambda}^{-1}),
$$
where the branch of the logarithm is chosen as explained above.
\end{definition}

\begin{theorem} \label{mygoodness}
The function $\Xi(\lambda)$ is holomorphic near the closed upper half space and for any $\epsilon>0$ and $\delta' \in (0,\delta)$ we have the bounds
\begin{align}
|\Xi(\lambda)|\leq C_{\delta',\epsilon}e^{-\delta' \Im(\lambda)} \quad |\Xi'(\lambda)|\leq C_{\delta',\epsilon} e^{-\delta'\Im \lambda}
\end{align} 
for $\lambda$ in the sector $\mathfrak{D}_{\epsilon}$.
\end{theorem}
\begin{proof}
 The first bound is a direct consequence of the proposition above. The second bound is a direct consequence of the maximum modulus principle.
\end{proof}

\section{Relative Trace Formula} \label{multiLzwei}
We consider the two Maxwell resolvent differences
 \begin{align*}
 R_{D,\rel,\lambda} &=  \left(\left( (-\Delta_\rel-\lambda^2)^{-1} -  (-\Delta_\free-\lambda^2)^{-1} \right) - \sum_{j=1}^N \left( (-\Delta_{\rel,j} -\lambda^2)^{-1} -  (-\Delta_\free-\lambda^2)^{-1} \right)\right)\curl\curl,\\
 R_{D,\abs,\lambda} &=  \left(\left( (-\Delta_\abs-\lambda^2)^{-1} -  (-\Delta_\free-\lambda^2)^{-1} \right) - \sum_{j=1}^N \left( (-\Delta_{\abs,j} -\lambda^2)^{-1} -  (-\Delta_\free-\lambda^2)^{-1} \right)\right)\curl\curl.
 \end{align*}
  
 Using \eqref{eqn:ResolvDiffQ} and \eqref{eqn:ResolvDiffQabs}, we conclude 
 \begin{align*}
 &\left((-\Delta_{\rel,j} -\lambda^2)^{-1} -  (-\Delta_{\free}-\lambda^2)^{-1}\right)\curl\curl = -\tilde{\mathcal{L}}_\lambda \mathcal{L}_{j,\lambda}^{-1} (\nu \times)\tilde{\mathcal{L}}_\lambda^\trans,\\
 &\left((-\Delta_{\abs,j} -\lambda^2)^{-1} -  (-\Delta_{\free}-\lambda^2)^{-1}\right)\curl\curl = -\lambda^2 \tilde{\mathcal{M}}_\lambda \mathcal{L}_{j,\lambda}^{-1} (\nu \times)\tilde{\mathcal{M}}_\lambda^\trans,
 \end{align*}
  and hence  
 \begin{align*}
 R_{D,\rel,\lambda} =  - \tilde{\mathcal{L}}_\lambda \mathcal{L}_\lambda^{-1} (\nu \times)\tilde{\mathcal{L}}^\trans_\lambda + \tilde{\mathcal{L}}_{\lambda}\mathcal{L}_{D,\lambda}^{-1}(\nu \times) \tilde{\mathcal{L}}^\trans_\lambda,\\
 R_{D,\abs,\lambda} =  - \tilde{\mathcal{M}}_\lambda \mathcal{L}_\lambda^{-1} (\nu \times)\tilde{\mathcal{M}}^\trans_\lambda + \tilde{\mathcal{M}}_{\lambda}\mathcal{L}_{D,\lambda}^{-1}(\nu \times) \tilde{\mathcal{M}}^\trans_\lambda.
 \end{align*}
 
 We have the following improvement of Theorem \ref{Thm:DifferenceProperties} in the relative setting:
 \begin{proposition} \label{decayrelres}
 Let $\epsilon>0$ and let $\delta'>0$ be smaller than $\delta=\mathrm{dist}(\partial \Omega_j,\partial \Omega_k)$. Then the operators $R_{D,\rel,\lambda},R_{D,\abs,\lambda}: L^2(\R^3,\C^3) \to L^2(\R^3,\C^3)$ are trace-class for all $\lambda \in \mathfrak{D}_{\epsilon}$
 and their trace norm can be estimated by
 \begin{align*}
  \| R_{D,\rel,\lambda} \|_{1} + \| R_{D,\abs,\lambda} \|_{1} &\leq C_{\delta',\epsilon}  e^{-\delta' \Im(\lambda)}, \quad \lambda \in \mathfrak{D}_{\epsilon}.
 \end{align*}
 \end{proposition}
 \begin{proof}
 First note that
 $$
  (\mathcal{L}_{\lambda}^{-1}-\mathcal{L}_{\lambda,D}^{-1}) =- (\mathcal{L}_{\lambda}^{-1}\mathcal{T}_{\lambda}\mathcal{L}_{\lambda,D}^{-1})
 $$
 is a meromorphic family of trace-class operators $H^{-\frac{1}{2}}(\Div,\partial \Omega) \to H^{-\frac{1}{2}}(\Div,\partial \Omega)$ in the complex plane.
 For $|\lambda|>1$ the bound then follows from Prop. \ref{Prop:CalT} and the bounds Theorem \ref{supercor}. 
 In particular the expansion
 $$
 (\mathcal{L}_{\lambda}^{-1}-\mathcal{L}_{\lambda,D}^{-1}) = \frac{1}{\lambda^2} L_2 + \frac{1}{\lambda} L_1 + L_{0,\lambda}
 $$
 resulting from \eqref{Lexpand} is in terms of trace-class operators $L_2$, $L_1$ and the holomorphic family of trace-class operators $L_{0,\lambda}$.
 Specifically,
 \begin{align*} 
  L_2 &= -  T (Q^{(0)} -Q^{(0)}_D) -( U^{(0)} P - U^{(0)}_D P_D) = T W_2 + V_2,
  \\ L_1 &= -T (Q^{(1)} -Q^{(1)}_D) -(U^{(0)}B-U^{(0)}_D B_D) - ( U^{(1)} P - U^{(1)}_D P_D) = T W_1 + V_1,
 \end{align*}
 where $Q^{(0)}, Q^{(1)}$ are the expansion coefficients of $$Q_\lambda = Q^{(0)} + Q^{(1)} \lambda + O(|\lambda|^2)$$ near $\lambda=0$. The same notation is used for the expansion coefficients of $Q_{D,\lambda}, U_\lambda, U_{D,\lambda}$.
 Since the operator $$(\frac{1}{2} + \mathcal{M}_\lambda)^{-1} - (\frac{1}{2} + \mathcal{M}_{D,\lambda})^{-1} =-(\frac{1}{2} + \mathcal{M}_\lambda)^{-1} \mathcal{J}_\lambda (\frac{1}{2} + \mathcal{M}_{D,\lambda})^{-1}$$ is a meromorphic family of trace-class operators we know that the expansion coefficients $W_2=Q^{(0)} -Q^{(0)}_D$ 
 and $W_1=Q^{(1)} -Q^{(1)}_D$ are trace-class. We also record that $V_2 (\nu \times \nabla)=0$ and $V_1 (\nu \times \nabla)=0$ and recall that $\Div\, T=0$.
 Now we are ready to estimate the resolvent differences. We first focus on $R_{D,\rel,\lambda}$. We have 
 \begin{align*} 
  R_{D,\rel,\lambda}& =-\tilde{\mathcal{L}}_{\lambda}(\mathcal{L}_{\lambda}^{-1}-\mathcal{L}_{D,\lambda}^{-1})(\nu \times)\tilde{\mathcal{L}}^\trans_{\lambda}\\
  &= -\tilde{\mathcal{L}}_{\lambda} \left(\frac{1}{\lambda^2}(T W_2 + V_2) + \frac{1}{\lambda}(T W_1 + V_1) +  L_{0,\lambda}\right) (\nu \times)\tilde{\mathcal{L}}^\trans_{\lambda}.
  \end{align*}
We expand this further using $\tilde{\mathcal{L}}_\lambda = \nabla \tilde{\mathcal{S}}_\lambda \Div + \lambda^2 \tilde{\mathcal{S}}_\lambda$ to obtain
that modulo terms that have bounded trace-norm near $\lambda=0$ the operator $R_{D,\lambda}$ equals
\begin{align*}
  &\left( \nabla  \tilde{\mathcal{S}}_\lambda \Div  + \lambda^2\tilde{\mathcal{S}}_\lambda \right)  \left(\frac{1}{\lambda^2}(T W_2 + V_2) + \frac{1}{\lambda}(T W_1 + V_1) \right) \left((\nu \times) \nabla  \tilde{\mathcal{S}}^\trans_\lambda \div  + \lambda^2 (\nu \times)(\tilde{\mathcal{S}}_\lambda)^\trans \right) \\&=
  \nabla  \tilde{\mathcal{S}}_\lambda \Div \left((T W_2 + V_2) + \lambda(T W_1 + V_1) \right) (\nu \times)(\tilde{\mathcal{S}}_\lambda)^\trans \\&+ \tilde{\mathcal{S}}_\lambda \left((T W_2 + V_2) + \lambda(T W_1 + V_1) \right)(\nu \times) \nabla  \tilde{\mathcal{S}}^\trans_\lambda \div\\&+
  \lambda^4 \tilde{\mathcal{S}}_\lambda \left(\frac{1}{\lambda^2}(T W_2 + V_2) + \frac{1}{\lambda}(T W_1 + V_1) \right)  (\nu \times)(\tilde{\mathcal{S}}_\lambda)^\trans = \mathrm{(I)} + \mathrm{(II)} + \mathrm{(III)}.
\end{align*}
Since $\mathcal{L}_{D,\lambda}$ and $\mathcal{L}_{\lambda}$ are self-adjoint with respect to the antisymmetric bilinear and since
$$
 \mathcal{L}_\lambda^{-1} -  \mathcal{L}_{D,\lambda}^{-1} =  \left(\frac{1}{\lambda^2}(T W_2 + V_2) + \frac{1}{\lambda}(T W_1 + V_1) +  L_{0,\lambda}\right)
$$
one obtains that
$$
\left(\left(\frac{1}{\lambda^2}(T W_2 + V_2) + \frac{1}{\lambda}(T W_1 + V_1) \right)  (\nu \times)\right)^\trans= \left(\frac{1}{\lambda^2}(T W_2 + V_2) + \frac{1}{\lambda}(T W_1 + V_1) \right) (\nu \times)
$$ and therefore $\mathrm{(II)}$ is the transpose of $\mathrm{(I)}$. $\mathrm{(III)}$ has bounded trace-norm near $\lambda=0$. Finally 
$$
 \mathrm{(II)}= \tilde{\mathcal{S}}_\lambda \left((T W_2 + V_2) + \lambda(T W_1 + V_1) \right)(\nu \times) \nabla  \tilde{\mathcal{S}}^\trans_\lambda \div = \tilde{\mathcal{S}}_\lambda \left((T W_2) + \lambda(T W_1) \right)(\nu \times) \nabla  \tilde{\mathcal{S}}^\trans_\lambda \div,
$$
has bounded trace-norm near $\lambda=0$ as $ \tilde{\mathcal{S}}_\lambda T,  \nabla  \tilde{\mathcal{S}}^\trans_\lambda \div$ have bounded operator norm and $W_2$ and $W_1$ are trace-class.
Finally we consider $R_{D,\abs,\lambda}$. We compute as above
\begin{align*} 
  R_{D,\abs,\lambda}& =-\lambda^2 \tilde{\mathcal{M}}_{\lambda}(\mathcal{L}_{\lambda}^{-1}-\mathcal{L}_{D,\lambda}^{-1})(\nu \times)\tilde{\mathcal{M}}^\trans_{\lambda}\\
  &= -\lambda^2 \tilde{\mathcal{M}}_{\lambda} \left(\frac{1}{\lambda^2}(T W_2 + V_2) + \frac{1}{\lambda}(T W_1 + V_1) +  L_{0,\lambda}\right) (\nu \times)\tilde{\mathcal{M}}^\trans_{\lambda}\\
  &=-\tilde{\mathcal{M}}_{\lambda} \left((T W_2 + V_2) + \lambda(T W_1 + V_1) +  \lambda^2 L_{0,\lambda}\right) (\nu \times)\tilde{\mathcal{M}}^\trans_{\lambda},
  \end{align*}
  whose trace-norm is bounded near zero since $\tilde{\mathcal{M}}_{\lambda}$ is uniformly bounded.
\end{proof}

\begin{lemma} \label{tracexi}
We have
$$
 \tr\left( R_{D,\rel,\lambda} \right)= \tr\left( R_{D,\abs,\lambda} \right) = -\frac{\lambda}{2}\Xi'(\lambda).
$$
\end{lemma}
\begin{proof}
One has
\begin{align*}
 (\nu \times) \tilde{\mathcal{L}}^\trans_\lambda \tilde{\mathcal{L}}_\lambda &=  \gamma_t \curl\, \curl\, \curl\, \curl\,(-\Delta_\free -\lambda^2)^{-2} \gamma^\trans_T =
  \gamma_t \curl\, \curl\, (-\Delta_\free) (-\Delta_\free -\lambda^2)^{-2} \gamma^\trans_T 
  \\&=  \gamma_t \curl\, \curl\, (-\Delta_\free -\lambda^2)^{-1} \gamma^\trans_T +\gamma_t \curl\, \curl\, \lambda^2 (-\Delta_\free -\lambda^2)^{-2} \gamma^\trans_T=
  \mathcal{L}_\lambda + \frac{\lambda}{2} \frac{\der}{\der \lambda}\mathcal{L}_\lambda.
\end{align*}
Similarly, we also have
\begin{align*}
 \lambda^2(\nu \times) \tilde{\mathcal{M}}^\trans_\lambda \tilde{\mathcal{M}}_\lambda &=  \lambda^2\gamma_t \curl\, \curl\,(-\Delta_\free -\lambda^2)^{-2} \gamma^\trans_T =
\frac{\lambda}{2} \frac{\der}{\der \lambda}\mathcal{L}_\lambda.
\end{align*}

Using invariance of the trace in $H^{-\frac{1}{2}}(\Div,\partial \Omega)$ under cyclic permutations we get
\begin{align*}
\tr\left( R_{D,\rel,\lambda} \right) &= -  \tr\left( - \tilde{\mathcal{L}}_\lambda ( \mathcal{L}_\lambda^{-1} -  \mathcal{L}_{D,\lambda}^{-1})(\nu \times)\tilde{\mathcal{L}}^\trans_\lambda \right) = - \tr\left(  (\mathcal{L}_\lambda + \frac{\lambda}{2} \frac{\der}{\der \lambda}\mathcal{L}_\lambda)( \mathcal{L}_\lambda^{-1} -  \mathcal{L}_{D,\lambda}^{-1}) \right) \\&= -\tr\left( (\id - \mathcal{L}_\lambda \mathcal{L}_{D,\lambda}^{-1}) \right) - \frac{\lambda}{2} \frac{\der}{\der \lambda} \log \det \left(\mathcal{L}_\lambda \mathcal{L}_{D,\lambda}^{-1}\right) = - \frac{\lambda}{2} \frac{\der}{\der \lambda} \log \det \left(\mathcal{L}_\lambda \mathcal{L}_{D,\lambda}^{-1}\right) 
\end{align*}
Here we have used that $\tr\left( \mathcal{L}_{D,\lambda}^{-1} \frac{\der}{\der \lambda}(\mathcal{T}_{\lambda}) \right)=0$, $\tr\left( \mathcal{L}_{D,\lambda}^{-1} \mathcal{T}_{\lambda} \right)=0$. Indeed this follows as 
 $$
  \Tr \left(  \mathcal{L}_{\lambda,D}^{-1}( \frac{d}{d \lambda} \mathcal{T}_{\lambda} ) \right) = \sum_{j\not=k}\Tr \left(  \mathcal{L}_{\lambda,D}^{-1}(q_j (\frac{d}{d \lambda} \mathcal{L}_{\lambda})  q_k)\right) = \sum_{j\not=k}\Tr\left(  (q_j\mathcal{L}_{\lambda,D}^{-1}  \frac{d}{d \lambda} \mathcal{L}_{\lambda} )q_k \right) =0.
 $$
We also used the fact that  for a holomorphic family of trace-class operators $A(\lambda)$ we have that $\log \det (\id + A(\lambda))$ is holomorphic and we have the identity
$$\frac{\der}{\der \lambda} \log \det (\id + A(\lambda)) = \tr\left(  (\id + A(\lambda))^{-1} \frac{\der}{\der \lambda} A(\lambda)\right),$$ so that
$$
\frac{\der}{\der \lambda} \log \det \left(\mathcal{L}_\lambda \mathcal{L}_{D,\lambda}^{-1}\right)  = \tr\left( \mathcal{L}_\lambda^{-1} \frac{\der}{\der \lambda}(\mathcal{L}_\lambda)- \mathcal{L}_{D,\lambda}^{-1}\frac{\der}{\der \lambda}(\mathcal{L}_{D,\lambda})  \right). 
$$
In the same way,
\begin{align*}
\tr\left( R_{D,\abs,\lambda} \right) &= -  \tr\left( \lambda^2 \tilde{\mathcal{M}}_\lambda ( \mathcal{L}_\lambda^{-1} -  \mathcal{L}_{D,\lambda}^{-1})(\nu \times)\tilde{\mathcal{M}}^\trans_\lambda \right) = - \tr\left( ( \frac{\lambda}{2} \frac{\der}{\der \lambda}\mathcal{L}_\lambda) ( \mathcal{L}_\lambda^{-1} -  \mathcal{L}_{D,\lambda}^{-1}) \right) \\&= - \frac{\lambda}{2} \frac{\der}{\der \lambda} \log \det \left(\mathcal{L}_\lambda \mathcal{L}_{D,\lambda}^{-1}\right) 
\end{align*}
\end{proof}

\section{Proof of the main theorems} \label{mainproofs}

\begin{proof}[Proof of Theorem \ref{nicetheorem}]
This theorem is the combination of Prop. \ref{proporprpoasd} and Theorem \ref{mygoodness} in Section \ref{multiL}.
\end{proof}

\begin{proof}[Proof of Theorem \ref{nicetheorem2}]
We set $f(z) = z^{-2} g(z^2)$ where $g \in \mathcal{P}_\epsilon$. By the decay properties of $\Xi$ it is sufficient to show equality for small $\epsilon$, so we assume 
$\epsilon < \frac{\pi}{4}$. Then the function $e^{- \frac{1}{n} z^2}$ is holomorphic in the sector $\mathfrak{S}_\epsilon$ and decays faster than exponentially. The function $g_n(z) = e^{- \frac{1}{n} z^2} g(z)$ is therefore an admissible function for the Riesz-Dunford functional calculus and we therefore have
\begin{align*}
 f_n((-\Delta_\rel)^\frac{1}{2}) \curl\, \curl =   f_n((\delta \der)^\frac{1}{2}) \delta \der = g_n(\delta \der) =  -\frac{1}{2 \pi \rmi}\int_{\Gamma_\epsilon} (\delta \der - z)^{-1} g_n(z) \der z
 \end{align*}
and similarly for the other terms appearing in $R_{D,\rel,\lambda}$. The integral converges despite the pole of order one at zero since $g \in \mathcal{P}_\epsilon$ implies that $g_n(z) = O(|z|^\alpha)$ for some $\alpha>0$ near $z=0$.
Here $f_n(z) = z^{-2} g_n(z^2)$. 
If $h \in C^\infty_0(X,\C^3)$ then we have convergence of  $g_n(\delta \der)$ to $g(\delta \der)$ in $L^2$. Indeed, by our definition of the function class $\mathcal{P}_\epsilon$, the function $g$ is polynomially bounded on the real line and therefore
$h$ is in the domain of the operator $g(\delta \der)$. Consequently
the function $g$ is square integrable with respect to the measure $\langle \der E_\lambda h,h \rangle$, where $\der E_\lambda$ is the spectral measure of $\delta \der$.
Then we have
$$
 \| \left( g(\delta \der)-g_n(\delta \der) \right) h\|_{L^2} = \int_\R (1-e^{- \frac{1}{n}x^2})^2 |g|^2(x) \langle \der E_\lambda h,h \rangle,
$$
which tends to zero as $n \to \infty$ by the dominated convergence theorem.

We note now that 
$$
  (-\Delta_\rel-z)^{-1}\delta \der = (\delta \der - z)^{-1} \delta \der = \id +  z (\delta \der - z)^{-1},
 $$
 and again this formula applies to the other terms in $R_{D,\rel,\lambda}$. This gives
 $$
  D_{\rel,f_n} = -\frac{1}{2 \pi \rmi}\int_{\Gamma_\epsilon} R_{D,\rel}(\sqrt{z}) \frac{1}{z} g_n(z) \der z = - \frac{1}{\rmi \pi} \int_{\tilde \Gamma_{\epsilon/2}} R_{D,\rel}(\lambda) \lambda f_n(\lambda) \der \lambda.
 $$
 Moreover, $D_{\rel,f_n} h$ converges in $L^2$ to $D_{\rel,f} h$ for any $h \in C^\infty_0(X,\C^3)$.
 By the decays properties of $R_{D,\rel}(\lambda)$, Proposition \ref{decayrelres}, the integral converges in the Banach space of trace-class operators, and the sequence
 $D_{f_n}$ is Cauchy in the Banach space of trace-class operators. We conclude that $D_{\rel,f}$ is trace-class.\\
 To compute the trace we can again use the convergence of the integral in the space of trace-class operators and therefore, using Lemma \ref{tracexi}, we obtain
 $$
  D_{\rel,f} =  - \frac{1}{\rmi \pi} \int_{\tilde \Gamma_\frac{\epsilon}{2}} R_{\rel,D}(\lambda) \lambda f(\lambda) \der \lambda =   \frac{1}{2 \pi \rmi} \int_{\tilde \Gamma_\frac{\epsilon}{2}} (\Xi'(\lambda)) \lambda^2 f(\lambda) \der \lambda.
 $$
 Integration by parts and the decay of $\Xi$, Theorem \ref{mygoodness}, then completes the proof for $D_{\rel,f}$. The proof for $D_{\abs,f}$ is exactly the same.
\end{proof}

\begin{proof}[Proof of Theorem \ref{nicetheorem3}]
 We first establish the smoothness away from the objects. To see this we again use the Riesz-Dunford functional calculus. 
 Let $\kappa_\lambda(x,y)$ be the integral kernel of the difference
 $$
  (-\Delta_\rel - \lambda^2)^{-1} \curlcurl - (-\Delta_\free - \lambda^2)^{-1} \curlcurl.
 $$
 Let $U$ be an open neighborhood of $\partial \Omega$ such that $\mathrm{dist}(U,\Omega_0)> \delta' \in (0,\delta)$.
 Then, on $\Omega_0 \times \Omega_0$ the integral kernel of $\kappa_\lambda(x,y)$ satisfies the estimate
 $$
  \| \kappa_\lambda(x,y) \|_{C^k(K)} \leq C_{k,K} e^{-\delta'\Im{\lambda} }
 $$
 for any compact subset $K \subset \Omega_0 \times \Omega_0$. This can be seen directly from \eqref{eqn:ResolvDiffQ} as Lemma \ref{Prop:EstL2PGfirst} implies that the integral kernel of $\tilde{\calL}_\lambda$ is smooth, and $C^\infty$-seminorms satisfy an exponential decay estimate on $\Omega_0 \times U$, whereas the norm of
 $\calL_\lambda^{-1}$ is polynomially bounded by Theorem \ref{supercor}.
 By the same argument as in the proof of Theorem \ref{nicetheorem2} above the integral
 $$
  2 \int_{\tilde \Gamma_{\epsilon/2}} \kappa_\lambda(x,y) \lambda f(\lambda) \der \lambda
 $$
 then converges in $C^\infty(\Omega_0 \times \Omega_0)$ to the integral integral kernel of $B_f$ restricted to $\Omega_0 \times \Omega_0$. Hence this kernel is smooth
 on $\Omega_0 \times \Omega_0$. It remains to show the decay estimate.
 For large $|x|$ we have by \eqref{eqn:DiagKernelEst} the estimate
 $$
   \| \kappa_\lambda(x,x) \|  \leq \frac{C}{|x|^4} e^{-\delta'  |x| \Im{\lambda} }.
 $$
 Then, using functional calculus as before, we have the representation
 $$
  \kappa(x,x) = \frac{\rmi}{\pi}\int_{\tilde \Gamma_\frac{\epsilon}{2}} \kappa_\lambda(x,x) \lambda f(\lambda) \der \lambda,
 $$
 which gives the estimate
 $$
   \| \kappa(x,x) \|  \leq  \int_1^\infty \frac{C}{|x|^4} \lambda e^{-\delta_1 \lambda |x|} \der \lambda +   \int_0^1 \frac{C}{|x|^4} \lambda e^{-\delta_1 \lambda |x|} \lambda^{a} \der \lambda \leq \frac{C_1}{|x|^{6+a}}.
 $$
 This shows that $\kappa(x,x)$ is integrable and by Mercer's theorem the integral of $\tr(\kappa(x,x))$ is equal to the trace, as claimed.
\end{proof}

\begin{proof}[Proof of Theorem \ref{nicetheorem4}]
Define the relative spectral shift function 
$$\xi_D(\lambda) =  \frac{1}{2\pi \rmi} \log \frac{\det  S_\lambda}{\det ( S_{1,\lambda}) \cdots \det ( S_{N,\lambda})}.$$
By the Birman-Krein formula we have
$$
  \tr D_{\rel,f} = -\int_0^\infty \xi_D(\lambda)   \frac{\der}{\der \lambda}(\lambda^2 f(\lambda)) \der \lambda
  $$
  for any even Schwartz function $f$.
  
   Recall that $\Xi'$ has a meromorphic extension to the complex plane and it is holomorphic on the real line. Now assume that $f$ is a compactly supported even test function
   and let $\tilde f$ be a compactly supported almost analytic extension (see for example \cite{MR1345723} p.169/170).
   Let $\der m(z)= \der x \der y$ be the Lebesgue measure on $\C$.
  By the Helffer-Sj\"ostrand formula (\cite{MR1037319, MR1345723}), combined with the substitution $z \mapsto z^2$, we have
   $$
   \curl\curl f(\Delta_\rel^{1/2})= \frac{2}{\pi} \curlcurl  \int\limits_{\mathrm{Im}(z)>0}z \frac{\partial \tilde f}{\partial \overline z} (\Delta_\rel-z^2)^{-1} \der m(z),
 $$
   Therefore
 $$
   D_{\rel,f}  = \frac{2}{\pi} \int\limits_{\mathrm{Im}(z)>0}z \frac{\partial \tilde f}{\partial \overline z}  R_{D,\rel}(z)   \der m(z),
   $$ 
   and hence, by Lemma \ref{tracexi}, we have
   $$
    \Tr (D_{\rel,f}) = - \frac{1}{\pi} \int\limits_{\mathrm{Im}(z)>0} z^2 \frac{\partial \tilde f}{\partial \overline z} \Xi'(z)  \der m(z).
   $$
   Using Stokes' theorem in the form of \cite{MR1996773}*{p.62/63}, we therefore obtain
   $$
     \Tr (D_f) =  \frac{\rmi}{2\pi} \int_\R  \left( \Xi'(x) +\Xi'(-x) \right)x^2  f(x)   \der x.
   $$
 Comparing this with the Birman-Krein formula in Theorem \ref{BKF} gives $\frac{\rmi}{2\pi}  \left( \Xi'(x) +\Xi'(-x) \right) = \xi_D'(x)$. Since both functions are meromorphic this shows that this identity holds everywhere. We conclude that $\frac{\rmi}{2\pi}  \left( \Xi(\lambda) -\Xi(-\lambda) \right) -\xi_D(\lambda)$
 is constant. Clearly, $\left( \Xi(\lambda) -\Xi(-\lambda) \right)$ vanishes at zero, so the statement follows if we can show that $\xi_D(0)=0$. The estimate \cite{OS}*{Theorem 1.10} shows that
 $S_0= S_{1,0}=\ldots = S_{N,0} = \id$, which then indeed implies $\xi_D(0)=0$. The paper \cite{OS} assumes the boundary of $\Omega$ to be smooth, but the section on the expansions in this paper carry over unmodified to the Lipschitz case (see also the remarks in \cite{SWB} where this is made explicit).
 \end{proof}

\section{Appendix}

\subsection{Norm Estimates}
In the following we assume that $\Omega$ and $M$ are as in the main body of the text. Recall that the integral kernel of the free resolvent is given by \eqref{fskernel}. We will subsequently prove norm and pointwise estimates for $G_{\lambda,0}$ and its derivatives, which are used in the main body of the text. 

\begin{lemma} \label{Prop:EstL2PGfirst} 
 Let $\Omega_0 \subset M$ be an open set with $\mathrm{dist}(\Omega_0, \partial \Omega)= \delta >0$ and choose $\epsilon \in (0,\pi]$. 
 Let $U$ be a bounded open neighborhood of the boundary $\partial \Omega$ such that $\mathrm{dist}(\Omega_0, U)>0$ and fix $\delta'>0$ such that
 $\delta'<\mathrm{dist}(\Omega_0, U) \leq \delta$. Then for any $k \in \N_0$ there exists $C_{k,\delta',\epsilon}>0>0$ such that  we have
 \begin{align}
 \| G_{\lambda,0}\|^2_{H^{k}(\Omega_0 \times U)} &\leq C_{k,\delta',\epsilon}\, \frac{(1+\Im{\lambda}) e^{-2\delta' \Im{\lambda}} }{\Im{\lambda}} ,\\
 \| \nabla_x G_{\lambda,0}\|^2_{H^{k}(\Omega_0 \times U)} &\leq C_{k,\delta',\epsilon}\, e^{-2\delta' \Im{\lambda}},
 \end{align}
 for all $\lambda \in \mathfrak{D}_\epsilon$. Here $\nabla_x$ denotes differentiation in the first variable, i.e.   $(\nabla_x G_\lambda)(x,y) = \nabla_x G_\lambda(x,y)$.
\end{lemma}
\begin{proof}

Let us set $\lambda = \theta |\lambda|$ and note that $\Im(\theta) \geq \sin(\epsilon) > 0$. Since the kernel $G_{\lambda,0}$ satisfies the Helmholtz equation in both variables away from the diagonal we have $((-\Delta_x)^{k} + (-\Delta_y)^k)G_{\lambda,0}(x,y)=2\lambda^{2k}G_{\lambda,0}(x,y)$. We then change variables so that $r := |x-y| \geq \delta_0$. By homogeneity, all of the integration will be carried out in this variable, with the angular variables only contributing a constant. Substituting $s := \Im\lambda \ r$, into the formula for the Green's function implies for all $k\in\mathbb{N}$ that
\begin{align} \label{eqn:MainL2Integral}
  \|\Delta^k G_{\lambda,0}\|^2_{L^2(\Omega_0' \times U')} 
  &\leq C_k (\Im\lambda)^{4k}\int\limits_{\delta_0}^\infty |G_{\lambda,0}(r)|^2 r^{2} dr \leq C_k (\Im\lambda)^{4k} \int\limits_{\delta_0}^{\infty}e^{-2\Im\lambda r}\,dr
\end{align}
Here we have enlarged the domains slightly  so that $\Omega_0' \times U'$ has positive distance from $\Omega_0 \times U$ and $\mathrm{dist}(\Omega_0',U')> \delta'$.
This allows us to estimate the Sobolev norms using Lemma \ref{sobolestimate}.
We then have 
\begin{align}
\int\limits_{\delta_0}^{\infty}e^{-2\Im\lambda r}\,dr=\frac{e^{-2\delta_0\Im\lambda}}{-2\Im\lambda}.
\end{align}
Let $C_{\delta',\epsilon,k}$ denote a generic constant depending on $\delta',\epsilon,k$. Using \eqref{eqn:MainL2Integral} and interpolation, as a result we can conclude for all $k \geq 0$ we have
\begin{align}\label{hi} 
 \|G_{\lambda,0}\|^2_{H^k(\Omega_0 \times U)} \leq C_{\delta',\epsilon,k}\,\frac{(1+\Im{\lambda}) e^{-2\delta' \Im{\lambda}} }{\Im{\lambda}} .
\end{align}
The second inequality follows by replacing $G_{\lambda,0}$ by $\nabla_xG_{\lambda,0}$ in \eqref{eqn:MainL2Integral}. We then have 
\begin{align} \label{eqn:MainL2Integral2}
& \|\Delta_x^k \nabla_xG_{\lambda,0}\|^2_{L^2(\Omega_0' \times U')} 
  \leq \\& C_k (\Im\lambda)^{4k}\int\limits_{\delta_0}^\infty |\nabla_xG_{\lambda,0}(r)|^2 r^{2} dr \leq C_k (\Im\lambda)^{4k} \int\limits_{\delta_0}^{\infty}\left(|\Im\lambda|^2+\frac{1}{r^2}\right)e^{-2\Im\lambda r}\,dr \nonumber 
  \leq C_k e^{-2\delta' \Im{\lambda}}.
\end{align}
\end{proof}

We now combine these estimates to get an estimate on the Maxwell layer potential operator.

\begin{lemma} \label{Prop:EstL2PG}
Let $\Omega_0\subset M$ be an open set with $\mathrm{dist}(\Omega_0, \Omega)= \delta >0$ and $\lambda \in \mathfrak{D}_\epsilon$. Then, for any $0< \delta' < \delta$  there exists $C_{\delta',\epsilon}>0$ such that  we have
 \begin{align}\label{Pbound2}
  \| \tilde{\mathcal{L}}_{\lambda} \|^2_{H^{-\frac{1}{2}}(\Div,\partial\Omega)\rightarrow H(\curl,\Omega_0)} \leq C_{\delta',\epsilon}\,e^{-2\delta' \Im{\lambda}}
 \end{align}
 and
 \begin{align}\label{Pbound}
  \| \tilde{\mathcal{L}}_{\lambda} \|^2_{H^{-\frac{1}{2}}(\Div 0,\partial\Omega)\rightarrow H(\curl,\Omega_0)} \leq C_{\delta',\epsilon}\, |\Im\lambda|^3 e^{-2\delta' \Im{\lambda}}. 
 \end{align}
\end{lemma}
\begin{proof}
We choose as in Lemma  \ref{Prop:EstL2PGfirst} a bounded open neighborhood of $\partial \Omega$. For $a \in H^{-\frac{1}{2}}(\partial \Omega)$ the distribution $\gamma_t^*(a)$ is, by duality, in $H^{-1}_\comp(U)$. The first inequality then follows by using Lemma \ref{Prop:EstL2PGfirst} bearing in mind that integration defines a continuous map
$$
 H^k(\Omega_0 \times U) \times H^{-s}_\comp(U) \to H^{k-s}(\Omega_0)
$$
for $k$ large enough.
The second inequality follows from the  identity \eqref{divexpansion}, namely that we can write 
\begin{align}
\tilde{\mathcal{L}}_{\lambda}a=\nabla \tilde{\mathcal{S}}_{\lambda} \mathrm{Div} a+\lambda^2\tilde{\mathcal{S}}_{\lambda}a\quad a\in H^{-\frac{1}{2}}(\Div,\partial\Omega)
\end{align}
and again using Lemma  \ref{Prop:EstL2PGfirst} in the same way as above.
\end{proof}

\begin{lemma} \label{HilSobLemma}
 Let $k \in H^2(\R^d \times \R^d)$. Then $k$ is the integral kernel of a Hilbert-Schmidt operator
 $$
  K : H^{-1}(\R^d) \to H^1(\R^d),
 $$
 with Hilbert-Schmidt norm bounded by $\| k \|_{H^2(\R^d \times \R^d)}$.
\end{lemma}
\begin{proof}
 Let $K$ be the integral operator with kernel $k$.
 Since $(-\Delta+1)^\frac{1}{2}$ is an isometry from $L^2(\R^d)$ to $H^{-1}(\R^d)$ and from $H^1(\R^d)$ to $L^2(\R^d)$ it suffices to show that
 $(-\Delta+1)^\frac{1}{2} K (-\Delta+1)^\frac{1}{2}$ is Hilbert-Schmidt from $L^2(\R^d)$ to $L^2(\R^d)$ and bound its Hilbert-Schmidt norm.
 This is equivalent to the distributional integral kernel of $(-\Delta+1)^\frac{1}{2} K (-\Delta+1)^\frac{1}{2}$ to be in $L^2(\R^d \times \R^d)$ (see for example  \cite{shubin}). The Hilbert-Schmidt
 norm is equal to the $L^2$-norm of the kernel. The Fourier transform is $(\xi^2+1)^\frac{1}{2} (\eta^2+1)^\frac{1}{2} \hat k(\xi,\eta)$ and this is in $L^2$ with the $L^2$-norm bounded by $\| k \|_{H^2(\R^d \times \R^d)}$ thanks to the inequality
 $$
  \frac{(\xi^2+1)^\frac{1}{2} (\eta^2+1)^\frac{1}{2}}{\xi^2+\eta^2+1} \leq 1.
 $$
\end{proof}

\begin{lemma} \label{HilSobTrLemma}
 Let $k \in H^4_\comp(\R^3 \times \R^3)$ be supported in a compact set $Q \times Q \subset \R^{3}\times \R^3$. Then $k$ is the integral kernel of a nuclear operator
 $$
  K : H^{-1}(\R^3) \to H^1(\R^3),
 $$
 with trace norm bounded by $C_Q \| k \|_{H^4(\R^d \times \R^d)}$.
\end{lemma}
\begin{proof}
 Since $k$ is compactly supported in $Q$ we can assume without loss of generality that $Q$ is a subset of a torus $\mathbb{T}^n$ by imposing periodic boundary conditions on a sufficiently large rectangle and remarking that the Sobolev norms on the torus restricted to a neighborhood of $Q$ are then equivalent to those of $\R^d$ restricted to that neighborhood. We can therefore assume without loss of generality that we are on a compact manifold $Y$.
 We can then write $K$ as $K = (-\Delta_Y + 1)^{-1}  (-\Delta_Y + 1) K$. The operator  $(-\Delta_Y + 1)^{-1}$ is Hilbert-Schmidt from $H^1(Y)$ to $H^1(Y)$, as for example can be seen from Weyl's law.
 The operator $(-\Delta_Y + 1) K$ is Hilbert-Schmidt by Lemma \ref{HilSobLemma}. Since we have written the operator as a product of two Hilbert-Schmidt operators it is nuclear and the corresponding estimate for the nuclear norm follows by estimating in terms of the Hilbert-Schmidt norms.
\end{proof}

\begin{lemma} \label{sobolestimate}
 Suppose that $\Omega \subset \R^d$ is an open subset and assume that $\Omega' \subset \R^d$ is a larger subset such that
 $\overline{\Omega} \subset \Omega'$ and $\mathrm{dist}(\partial \Omega,\partial \Omega')>0$. Let $N \in \N$.
 Then for any $f \in L^2(\Omega')$ with $(-\Delta)^k f \in L^2(\Omega')$ for all $k =0,1,\ldots,N$ we have $f|_\Omega \in H^{2N}(\Omega)$ and there exists a constant
 $C_{N,\Omega',\Omega}>0$, independent of $f$, such that $\| f|_\Omega \|_{H^{2N}(\Omega)} \leq C_{N,\Omega',\Omega} \sum_{k \leq N}\| (-\Delta)^k f\|_{L^2(\Omega')}$.
\end{lemma}
\begin{proof}
 This is the usual proof of interior regularity applied to the possibly non-compact domain $\Omega'$. We will show that $f \in H^s(\Omega'), \Delta f \in H^s(\Omega')$ implies $f \in H^{s+2}(\Omega)$ with the corresponding norm-estimates. The result then follows from this statement by iterating using a sequence of intermediate domains
 $\Omega \subset \Omega_1 \subset \ldots \subset \Omega_{N-1} \subset \Omega'$.
 We will choose $U$ so that $\Omega \subset U \subset \Omega'$ while is still such that $\mathrm{dist}(\partial U,\partial \Omega')>0,\mathrm{dist}(\partial U,\partial \Omega)>0$.
 We can choose a regularised distance function and construct a function $\chi \in C^\infty_b(\R^d)$ which is compactly supported in $\Omega'$ which equals one in a neighborhood of $U$.
 Then, if  $f \in H^s(\Omega'), \Delta f \in H^s(\Omega')$ we have
 $$
  (1-\Delta)( \chi f) =  (\chi-\Delta( \chi)) f -  \chi \Delta f - 2 (\nabla \chi) \nabla f.
 $$
 From this we see that $(- \Delta +1) (\chi f) \in H^{s-1}(\R^d)$ and therefore $\chi f \in H^{s+1}(\R^d)$. Hence, the restriction of $f$ to $U$ is in $H^{s+1}(\Omega_1)$.
 Now we choose another cut-off function $\eta$ in $C^\infty_b(\R^d)$ supported in $U$ that equals one near $\Omega$. Then $(\nabla \eta) \nabla f$ is in $H^s(U)$ and we now conclude in the same way that $f|_{\Omega} \in H^{s+2}(\Omega)$.
 \end{proof}

\end{document}